\documentclass{amsart}

\usepackage{amsmath, amstext, amssymb, amsopn, amsthm, amscd, amsfonts, 
 epsfig, times,  fourier, phonetic}
 \usepackage{stmaryrd} 

\usepackage[auto]{contour}

\usepackage{upgreek}
 \usepackage{pifont}
 \usepackage{color}

\DeclareMathAlphabet{\mathpzc}{OT1}{pzc}{m}{it}

\newcommand{\R}{\mathbb R}
\newcommand{\C}{\mathbb C}

\newcommand{\F}{\mathbb F}

\newcommand{\K}{\mathbb K}
\newcommand{\N}{\mathbb N}
\newcommand{\PR}{\mathbb P}
\newcommand{\Q}{\mathbb Q}
\newcommand{\Z}{\mathbb Z}

\newcommand{\SI}{\mathbb S}

\newtheorem{theo}{Theorem}
\numberwithin{theo}{section}
\newtheorem{lemm}{Lemma}
\numberwithin{lemm}{section}
\newtheorem{prop}{Proposition}
\numberwithin{prop}{section}
\newtheorem{coro}{Corollary}
\numberwithin{coro}{section}

\newtheorem*{theocgt}{Closed Graph Theorem}

\theoremstyle{definition}
\newtheorem{defi}{Definition}
\numberwithin{defi}{section}

\theoremstyle{remark}
\newtheorem{nota}{Notation}
\numberwithin{nota}{section}
\newtheorem{exam}{Example}
\numberwithin{exam}{section}
\newtheorem*{illexam}{Illustrative Example}
\newtheorem*{illexamc}{Illustrative Example (Continued)}

\newtheorem{note}{Note}
\numberwithin{note}{section}

\newtheorem*{clai}{Claim}

\title{The $\uptheta$-adics}
\author{T.M. Gendron and A. Zenteno}
\address{Instituto de Matemáticas, Unidad Cuernavaca, U.N.A.M., Av. Universidad s/n, Lomas de Chamilpa, Cuernavaca, Morelos
}
\address{Centro de Investigación en Matemáticas, A.C.
Jalisco s/n. Col. Valenciana
36023, Guanajuato, Gto., México
}
\email{e-mail: tim@matcuer.unam.mx}
\email{e-mail: adrian.zenteno@cimat.mx}

\date{\today}

\begin{document}
 \maketitle
 \begin{abstract}This paper introduces an archimedean, locally Cantor multi-field $\mathcal{O}_{\uptheta}$ which
  gives an analog of the $p$-adic number field at a place at infinity of a real quadratic extension $K$ of $\Q$.  
  This analog is defined using a unit $1<\uptheta\in \mathcal{O}_{K}^{\times}$, which plays the same 
  role as the prime $p$ does in $\Z_{p}$; the elements of  $\mathcal{O}_{\uptheta}$ are then greedy Laurent series 
  in the base $\uptheta$.  There is a canonical inclusion of the integers $\mathcal{O}_{K}$ with 
  dense image in  $\mathcal{O}_{\uptheta}$ and the operations of sum and product extend to multivalued operations having at most three values, making $\mathcal{O}_{\uptheta}$ a multi-field in 
  the sense of Marty.  We show that the (geometric) completions of 1-dimensional quasicrystals contained in $\mathcal{O}_{K}$ map canonically to $\mathcal{O}_{\uptheta}$.  The motivation for this work arises in part from a desire to obtain a more arithmetic treatment of a place at infinity by
replacing $\R$ with $\mathcal{O}_{\uptheta}$, with an eye toward obtaining a finer version of class field theory incorporating the ideal arithmetic  of quasicrystal rings.
 \end{abstract}
 \tableofcontents

 \section*{Introduction}
 
For a very long time, the arithmetic at the places at infinity of an algebraic number field has resisted attempts to invest it with the sort of deep, explicitly algebraic expression which governs arithmetic at the finite places.  In concrete terms, there exist no accepted notions of integers, (factorization into) prime ideals and residue class field for the local archimedean fields $\R$ and $\C$.  This "featurelessness"  is compounded by the fact that these two fields are the only archimedian completions
one encounters in the theory of algebraic number fields: whereas for a fixed prime $p$ there are infinitely many non-isomorphic extensions $K_{\mathfrak{p}}$ of $\Q_{p}$ obtained by completing the various finite extensions $K/\Q$,  there are but two possible extensions obtained by completing $K/\Q$ at $\infty$: $\C/\R$ or the trivial extension $\R/\R$.

Bearing further witness to their deficit, these completions, while appearing as factors of the idèle groups,  are relegated to the background in the main theorems of Class Field Theory: essentially playing no role in the theorems of reciprocity.  Perhaps expressing his own chagrin at this state of affairs, A. Weil \cite{Weil} called for a Galois interpretation of
 the connected component of the idèle class group of a number field, but, as of this writing, there does not seem to be a satisfactory response to this request.  
 
 At the same time, it seems clear that if there is to be any hope of manufacturing a successful geometrization of algebraic number fields -- such as that which exists in the category of global function fields -- 
 one must bestow the places at infinity of each finite extension $K/\Q$ with a refined notion of completion, from which one may extract an arithmetic which reflects particular properties of $K$ and which may be distinguished from
 the refined infinite place completions of other extensions of $\Q$.

 All of this is to say that these flaws lead one to call into question the role of the fields $\R$ and $\C$ in algebraic number theory, and moreover, to ask if there are not better alternatives.
 The intention of this paper is to provide such an an alternative in the case of a real quadratic extension of $\Q$, which positions the 
infinite places in closer proximity to the finite places by way of an "adic construction", and at the same time, parallels the completions of infinite places in global fields of positive characteristic, in that
it provides a natural receptacle for completions of {\it quasicrystal rings}: characteristic zero analogs of the rings associated to infinite places in function fields.

Let us take a moment to explain the latter point.   Fix $K/\F_{q}(T)$ a finite extension induced by the morphism of curves
 $\uppi :  \Upsigma_{K}\rightarrow\PR^{1}$. Each point $\widetilde{\infty}\in  \uppi^{-1} (\infty)$ defines a non-archimedean, locally Cantor completion
 $K_{\widetilde{\infty}}$
 which is the field of fractions of a discrete valuation ring, and so may be viewed practically on an equal footing\footnote{In fact, the distinction between infinite and finite is essentially formal: one may replace $\infty$ by {\it any} point 
$ \infty_{0}\in \PR^{1}$ and obtain the same results.} with the completions at the remaining places. 
These remaining places (points of $ \Upsigma_{K}$) define the prime spectrum of a "small" (rank 1) Dedekind domain 
\[ A_{\widetilde{\infty}}= {\rm Reg}( \Upsigma_{K}\setminus \{\widetilde{\infty}\} ) = \text{ ring of functions regular outside of }\widetilde{\infty}\quad \subset\quad \mathcal{O}_{K}  ,
\]
which embeds discretely in $K_{\widetilde{\infty}}$ and plays a central role in explicit Class Field Theory, but which has no classical counterpart in characteristic zero.
 The lack of an analog of 
$ \Upsigma_{K}$ and the associated ''small ring'' $A_{\widetilde{\infty}}$  are, to a definite extent, yet further symptoms of the problem
afflicting the infinite places in characteristic zero.

Returning to the case of a number field $K/\Q$, given an infinite place $\upsigma$ of $K$, we may define an analog\footnote{$A_{\widetilde{\infty}}$ may be obtained by an analogous formula.}
of the small ring $A_{\widetilde{\infty}}$ by 
\begin{align}\label{IntroQCRing}A_{\upsigma} = \left\{\upsigma (\upalpha) \in\upsigma (\mathcal{O}_{K})\;\big|\;\; | \upalpha|_{\upsigma'}  \leq 1\; \forall  \text{ infinite place }\upsigma'\not= \upsigma \right\}\subset \K\end{align}
where $|\cdot |_{\upsigma'}$ is the valuation defined by the infinite place $\upsigma'$ and $\K=$ either $\R$ or $\C$.  Thus, $A_{\upsigma}$ contains all of the Pisot numbers and roots of unity contained in $\mathcal{O}_{K}$.   $A_{\upsigma}$  is Delaunay\footnote{Uniformly discrete and relatively dense.}, a multiplicative monoid and an approximate group with respect to $+$.   That is, $A_{\upsigma}$ is a {\it quasicrystal ring} in the sense of \cite{GLL}.  The latter reference describes the theory of quasicrystal ideals in $A_{\upsigma}$,
in which 
\begin{enumerate}
\item[A.] There are infinitely many primes but a {\it unique} maximal  quasicrystal ideal
\begin{align}\label{UniqMax} \mathfrak{m} =   \left\{\upsigma (\upalpha) \in A_{\upsigma} \big|\;\; | \upalpha|_{\upsigma'}  < 1\; \forall  \text{ infinite place }\upsigma'\not= \upsigma \right\} . \end{align}
\item[B.] The analog of (finiteness of) the ideal class group  is a {\it topological monoid} of quasicrystal ideal classes  
\[{\rm Cl}( A_{\upsigma})\] homeomorphic to a Cantor set.  In particular, there are uncountably many equivalence classes
of quasicrystal ideals.
\end{enumerate}
This quasicrystal ideal structure at infinity holds out the promise of extending Class Field Theory in a meaningful way to the archimedean places.

There is a natural geometric completion
$ \widehat{A}_{\upsigma}$ of 
$ A_{\upsigma}$ obtained by taking the closure of the set of translations ("cosets") 
\[  \big\{  \upalpha +  A_{\upsigma}  \; | \;\; \upalpha \in  A_{\upsigma} \big\} \]
in the space of all quasicrystals, and the result is a Stone space.   $ \widehat{A}_{\upsigma}$ defines a transversal in the solenoid \[ \widehat{\SI}_{\upsigma}= \text{ the closure of the set of translations }  \big\{ k +  A_{\upsigma}  \; | \;\;  k \in \K \big\},\] which is the analogue of {\it quotient torus}
in the category of quasicrystals.  In particular, the completion 
\[ \F_{\mathfrak{m}} = \text{closure of }  \big\{  \upalpha +  \mathfrak{m} \; | \;\; \upalpha \in  A_{\upsigma} \big\} \]
may be thought of as the quotient $A_{\upsigma}/\mathfrak{m}$, giving a notion (at least set-theoretically) of residue class field at infinity.

One would like to argue in favor of $ \widehat{A}_{\upsigma}$ as being the analog of the $\mathfrak{p}$-adic
integers, with residue class field  $ \F_{\mathfrak{m}} $, and obtain a replacement for $\R$ or $\C$ by forming a field of fractions of $ \widehat{A}_{\upsigma}$.   However, the geometric nature of the completion $ \widehat{A}_{\upsigma}$ makes it difficult to extract the exact expression of its arithmetic operations.  We are thus led to the following 

\begin{quote}
{\bf Problem:} Find a (locally Cantor) field and natural maps from $\widehat{A}_{\upsigma}$ and 
$\F_{\mathfrak{m}}$ into it, which allow one to define arithmetic operations on both of $\widehat{A}_{\upsigma}$ and $\F_{\mathfrak{m}}$.  
\end{quote}

The present work gives a solution to this problem in the case of $K/\Q$ real and quadratic.
  For such an extension, fix an infinite prime so that we may view $K\subset\R$ and
let $\uptheta\in\mathcal{O}_{K}^{\times}$ be a unit satisfying $\uptheta >1$.  We introduce in \S \ref{ThetaAdicHypSection} a completion of $\mathcal{O}_{K}$, 
 \[ \mathcal{O}_{\uptheta},\] whose definition is modeled on that of the $\mathfrak{p}$-adic field  $K_{\mathfrak{p}}$, $\mathfrak{p}\subset\mathcal{O}_{K}$ a prime ideal.
 For the moment,  we assume that the norm $N(\uptheta )=-1$; we  discuss what changes must be made in the case $N(\uptheta )=1$ at the end of the Introduction.

 The elements of  $\mathcal{O}_{\uptheta}$ are called {\it $\uptheta$-adic numbers},  and are  defined by Laurent series 
  in the base $\uptheta$ 
  \begin{align}\label{BasicElementForm}   \upalpha=\pm \sum_{i=n}^{\infty} b_{i}\uptheta^{i},\end{align}
  in which  $b_{i}\in \{ 0,\dots ,\lfloor \uptheta\rfloor\}$ for all $i$ and the partial sums are {\it greedy} Laurent polynomials in $\uptheta$.  See \S  \ref{GreedySection} for the definition of the greedy condition.  
  
Due to 
  the fact that the coefficients in (\ref{BasicElementForm}) are integers whereas the base $\uptheta$ is irrational, the arithmetic of the $\uptheta$-adics is considerably more complex than that of 
  the $\mathfrak{p}$-adics.   This produces a number of distinct features not found in the  $\mathfrak{p}$-adics,
the most prominent of which is the fact that the operations of sum and product
 in $ \mathcal{O}_{\uptheta}$ are {\it multivalued},  and in fact, we show later on in \S \ref{FieldSection}  that $ \mathcal{O}_{\uptheta}$ is a {\it multifield} in the sense of Marty \cite{Marty}, \cite{DO}.  (See \S \ref{ThetaAdicHypSection} for an account of Marty multistructures.)
 Also, in contrast with $K_{\mathfrak{p}}$,
 $ \mathcal{O}_{\uptheta} $ is a completion of $\mathcal{O}_{K}$ rather than $K$: that is, there is a canonical inclusion $\mathcal{O}_{K}\hookrightarrow  \mathcal{O}_{\uptheta} $ with dense image.  This may be viewed  as a companion to the density of $\mathcal{O}_{K}\subset \R $.  It suggests moreover that the correct analog of the $\mathfrak{p}$-adic integers $\mathcal{O}_{\mathfrak{p}}$ should be given by a suitable embedding
 of the completion $\widehat{A}_{\upsigma}$.

On the other hand, in parallel with the $\mathfrak{p}$-adics, $\mathcal{O}_{\uptheta}$ may be endowed with a discrete valued function defined
for $\upalpha$ of the form (\ref{BasicElementForm}) by
 \[  | \upalpha |_{\uptheta}  := \uptheta^{-n}, \] 
an analog of the $\mathfrak{p}$-adic norm.  See \S \ref{InfraNormSection}.
This function, however, is not a norm, but rather an {\it infranorm}: by definition, an infranorm $|\cdot |$ satisfies
the following weak version of the ultrametric triangle inequality, 
\begin{align}\label{IntroInfraTri}  | x+y| \leq \uprho \max\{ |x|, |y|\} ,\quad  \uprho\geq 1.\end{align}
Infranorms and inframetrics were introduced in  \cite{FLV} in order to model the round-trip delay geometry of the internet.   In our setting, $\uprho=\uptheta^{2}>2$, which implies that the infranorm is {\it archimedean}.

Although $ | \cdot |_{\uptheta} $ is not a norm, we may still define Cauchy sequences with respect to it, whose equivalence classes form  a completion $\widehat{\mathcal{O}}_{\uptheta}$ of $\mathcal{O}_{K}$, which {\it a priori} is an integral domain. See \S \ref{ThetaCompletionSection}.  This completion has the property that distinct 
Laurent series in the $\uptheta$-adic multifield $\mathcal{O}_{\uptheta}$  may define the same Cauchy class, so that, in particular,  $\mathcal{O}_{\uptheta}\not= \widehat{\mathcal{O}}_{\uptheta}$:  in fact,
in \S \ref{FieldSection}, we show that 
$\widehat{\mathcal{O}}_{\uptheta}$ is a {\it field} isomorphic to $\R$.

In spite of this, the infranorm $ | \cdot |_{\uptheta} $  and its completion $\widehat{\mathcal{O}}_{\uptheta}$  together  provide a useful apparatus for proving results about the $\uptheta$-adic multifield $\mathcal{O}_{\uptheta}$.  For example, it is seen that the natural map \[\uppi:\mathcal{O}_{\uptheta}\longrightarrow \widehat{\mathcal{O}}_{\uptheta}\cong\R\] is a surjective homomorphism (of multi-fields), establishing
the reals as a kind of collapsing of the  $\uptheta$-adics.  This map is injective outside a countable set and has pre-images of cardinality at most three: then, from the homomorphic quality of $\uppi$, we may deduce that the sum and product multioperations in $\mathcal{O}_{\uptheta}$ have likewise cardinality at most three.  See \S \ref{CardSection}.

The field structure of $\widehat{\mathcal{O}}_{\uptheta}$ (as well as the multifield structure of $\mathcal{O}_{\uptheta}$) is proved
in  \S \ref{FieldSection}.
This is accomplished by showing that any rational integer $n\in \Z\subset \widehat{\mathcal{O}}_{\uptheta}$ is invertible through an explicit recursive algorithm.  This in turn implies that any element of
$\mathcal{O}_{K} =\Z[\uptheta ]$ is invertible and by density of the latter in $\widehat{\mathcal{O}}_{\uptheta}$, we obtain the field structure.   Once we know that $\widehat{\mathcal{O}}_{\uptheta}$ is a field,  we are able to deduce that the map defined by Galois conjugation gives the isomorphism $\widehat{\mathcal{O}}_{\uptheta}\cong \R$.

Finally, in \S \ref{QCCompletionSection}, we show that the natural map $ A_{\upsigma}\subset \mathcal{O}_{K} $ gives rise to continuous injections
\[ \widehat{A}_{\upsigma}, \; \F_{\mathfrak{m}} \hookrightarrow \mathcal{O}_{\uptheta}  . \]
This answers the question posed earlier in the Introduction.   

As for the case $N(\uptheta )=1$, here it is no longer the case that every element of $\mathcal{O}_{K}$ is a Laurent polynomial in the unit $\uptheta$, as there are elements whose Galois conjugate has the opposite sign, such as $T=\uptheta-1$.  By expanding what is meant by Laurent series in $\uptheta$ to include expressions of the form 
\[T^{\upepsilon}\sum_{i=m}^{\infty} b_{i}\uptheta^{i},\quad \upepsilon \in \{ 0,-1\},\]
we obtain in this way a completion $\mathcal{O}_{\uptheta}$ of $\mathcal{O}_{K}$ which is a Marty multifield, and essentially all of what is developed in the case $N(\uptheta )=-1$ can be developed in this case.   Nonetheless, the arithmetic for $N(\uptheta )=1$ is quite distinct from that of $N(\uptheta )=-1$, and requires its own set of techniques.  The paper is therefore organized  to treat these two cases disjointly, and this is signaled by unnumbered subsections labelled {\small \fbox{$N(\uptheta )=-1$} }and {\small \fbox{$N(\uptheta )=1$}} for each.

We close with a few more words regarding our motivation in constructing the $\uptheta$-adics, in connection to its possible relationship with a conjectural Class Field Theory for the
quasicrystal ring $A_{\upsigma}$.  
  If one were 
 to replace in the definition of the idèles the archimedean factor  $\R^{2}$ by $\mathcal{O}_{\uptheta} \times \mathcal{O}_{\uptheta'}$, which is the basic proposal of this paper, it may be possible to extend the fundamental theorems of Class Field Theory to take into account
 the {\it quasicrystal arithmetic} at the infinite primes.   
 
 In particular, one hopes to incorporate the Cantor ideal class monoid ${\rm Cl} (A_{\upsigma})$ into the reciprocity correspondence.
 Recalling
 Weil's query regarding a possible Galois interpretation of the connected component of the idèle class group, here, the Galois interpretation may be understood  in terms of a conjectural quasicrystal Galois theory that refines the structure of the Weil group and foretells a new and generalized class of {\it Cantor algebraic} extensions of $\Q$: among which should be $H_{A_{\upsigma}}$,  a hypothetical analog of the Hilbert class field for the quasicrystal ring $A_{\upsigma}$.

Such a quasicrystal Class Field Theory, with its concomitant Cantor algebraic number theory,  might be of service in proving an analog of the Theorem of Weber-Fueter for real quadratic extensions of $\Q$, which adapts
the solution in the function field setting appearing in  \cite{DGIII}.  There, a  quantum (i.e.\ multivalued) analog $j^{\rm qt}$ of the modular invariant was used to explicitly generate the Hilbert class field $H_{\mathcal{O}_{K}} $
associated to the integral closure $\mathcal{O}_{K}$ of $\F_{q}[T]$ in a real quadratic extension $K/\F_{q} (T)$.  For a fundamental unit $f\in\mathcal{O}_{K}^{\times}$, it was shown that
$j^{\rm qt}(f)$ consists of a finite Galois orbit in $H_{A_{\widetilde{\infty}}}$ (the Hilbert class field of $A_{\widetilde{\infty}}$ for $\widetilde{\infty}$ as above), whose norm provides a primitive generator of $H_{\mathcal{O}_{K}} $.   

The quantum modular invariant may be defined in the number field
setting \cite{CG}, however for $\uptheta$ a fundamental unit, $j^{\rm qt}(\uptheta )$ is the continuous image of a  Cantor set \cite{GLL} (and conjecturally {\it is} Cantor), whose multivalues are the $j$-invariants of quasicrystal ideals of the quasicrystal ring $A_{\upsigma}$.   In order
to adapt the proof found in  \cite{DGIII} to the number field setting, the strategy would be  to incorporate $j^{\rm qt}(\uptheta )$ into the context of algebraic number theory,  and make use of the predicted quasicrystal Class Field Theory to show that it may be used to explicitly generate $H_{K}$ for $K/\Q$ quadratic and real. 

\vspace{3mm}

\noindent {\bf Acknowledgements:} The second author would like to thank the Instituto de Matemáticas, Unidad Cuernavaca, of U.N.A.M. for support during his postdoctoral stay there.  His work was
also partially supported by the CONAHCYT grant
CBF2023-2024-224.

 \section{Greedy Expansions and Quasicrystals}\label{GreedySection}
 
 We follow the account found in \cite{Gaz1}.
Fix once and for all \[ \uptheta\in\R, \quad \uptheta>1.\] 
Let $x\in\R$, $x\geq0$ and let $k\in\Z$ be such that
 \[  \uptheta^{k} \leq x < \uptheta^{k+1}. \]
 Define
 \[ b_{k} := \left[   \frac{x}{\uptheta^{k}} \right] \quad \text{and}\quad r_{k}:= \left\{  \frac{x}{\uptheta^{k}}\right\} \]
 where $[t]=$ integer part of $t$ and $\{ t\}$ is the fractional part. Recursively, for $i<k$, we define
 \[b_{i}:= [\uptheta r_{i+1}] \quad \text{and}\quad r_{i}:=  \left\{ \uptheta r_{i+1}\right\}  . \]
 The Laurent series in $\uptheta^{-1}$,
 \begin{align}\label{greedy}  \sum^{i=k}_{-\infty} b_{i}\uptheta^{i},\end{align}
 converges to $x$ and is called the {\bf {\em greedy expansion}} of $x$.
The coefficients in (\ref{greedy}) satisfy 
 \[  b_{i}\in \{ 0,1,\dots , \lceil  \uptheta  \rceil -1\} , \]
 where $\lceil  x  \rceil $ is the ceiling of $x$.
While there are other ways in which we may express $x$ as a {\it $\uptheta$-expansion} \cite{Renyi}, \cite{Parry} -- a sum of powers of $\uptheta$  with natural number coefficients -- the greedy expansion {\rm (\ref{greedy})} is, by its recursive form, unique.   

The greedy expansion of $-x$ is defined $-x = - \sum^{k}_{i=\infty} b_{i}\uptheta^{i}$.
We will often write \[ x=_{\rm gr}  \sum^{k}_{-\infty} b_{i}\uptheta^{i} \] to emphasize that a $\uptheta$-expansion is the greedy one.
It follows immediately from the greedy algorithm that the greedy expansion is well-behaved with respect to multiplication by powers of $\uptheta$:
\begin{align}\label{PowerBehavior}  \uptheta^{N}x=_{\rm gr}  \sum^{k}_{-\infty} b_{i}\uptheta^{i+N} . \end{align}
To conclude, there is a bijection 
\[ \R \longleftrightarrow\left \{ \pm \sum^{k}_{-\infty} b_{i}\uptheta^{i} \text{ a greedy expansion}\right\} .\]

There is a closely related recursion due to R\'{e}nyi, defined in the following way.  Write, for $x\in [0,1]$,
\[      T_{\uptheta}(x) := \{ \uptheta x\} = \uptheta x - [ \uptheta x];\]
this defines a function  
\[   T_{\uptheta}: [0,1]\longrightarrow [0,1].\]
Define, for $i=1,2,\dots $,
\[ t_{i}(x)  := [\uptheta T^{i-1}_{\uptheta} (x)]\]
(where  $T^{0}_{\uptheta}(x):=x$).
The sequence 
\[ d_{\uptheta}(x):= t_{1}(x)t_{2}(x)\cdots \]
is called the {\bf {\em R\'{e}nyi development}} of $x$.  When $x=1$ we write $t_{i}=t_{i}(1)$ so that 
\[ d_{\uptheta}(1) = t_{1}t_{2}\cdots .\]  We note that the R\'{e}nyi development of 1 produces a $\uptheta$-expansion of $1$,
\begin{align}\label{RenyiThetaExp1} 1= t_{1}\uptheta^{-1} + t_{2}\uptheta^{-2} + \cdots ,\end{align}
however this is {\it not} the greedy expansion of $1$, since $1=_{\rm gr} 1$.

The $\boldsymbol\uptheta${\bf {\em -integers}} are defined as the set of those $x$ for which the greedy expansion is a polynomial in $\uptheta$:
\[ \Z_{\uptheta} : = \left\{ x\in \R :\; |x|=_{\rm gr}  \sum^{k}_{i=0} b_{i}\uptheta^{i} \right\}.\]  We have the following characterization of $\Z_{\uptheta}$:

\begin{theo}{{\rm (Parry \cite{Parry})}} \label{ParryTheorem} The $\uptheta$-expansion $x= \sum_{i=0}^{n} b_{i}\uptheta^{i}  $ is greedy {\rm (}i.e.\ $x\in \Z_{\uptheta}${\rm )} if and only if for all $i=0,1,\dots ,n$,
\[  b_{i}\cdots b_{0} <_{\rm lex} t_{1}t_{2}\cdots t_{i+1}     \]
where $<_{\rm lex}$ is the lexicographical order.
\end{theo}


\begin{exam}\label{QuadraticEx} Let $\uptheta>1$ a real quadratic unit, so that $N(\uptheta )$ = norm of $\uptheta$ = $\pm 1$.  
\begin{enumerate}
\item[{\small \sf Case} \ding{202}] $N(\uptheta )=-1$.

\vspace{3mm}

\noindent Then $\uptheta^{2}=a\uptheta +1$, where $1\leq a\in\N$, and the R\'{e}nyi development of 1 is \[ d_{\uptheta}(1) = a1 .\] See \cite{Gaz1}, page 115.  By Parry's Theorem,  
$x=\pm \sum_{i=0}^{n} b_{i}\uptheta^{i}  $ is greedy if and only if for all $i$,
\begin{itemize}
\item[i.)] $b_{i}\in \{ 0,\dots ,a\}$ and
\item[ii.)]  If $b_{i}=a$ then $b_{i-1} = 0$.
\end{itemize}
Thus, for all $k$, the  {\bf {\em Fibonacci relation}}  \footnote{Since the continued fraction expansion of $\uptheta$ is
$ [a;aa\cdots ]$, the sequence of best approximations  $\{ q_{n}\}$ is defined by the Fibonacci recursion $q_{k+2}= aq_{k+1}+q_{k}$, which coincides formally 
with (\ref{powerrecursion}).}.

  \begin{align}\label{powerrecursion} \uptheta^{k+2}=a\uptheta^{k+1} +\uptheta^{k}\end{align} is not permitted in any polynomial greedy expansion.
\item[{\sf \small Case} \ding{203}] $N(\uptheta )=1$.

\vspace{3mm} 

\noindent Then $\uptheta^{2}=a\uptheta -1$, where now $3\leq a\in\N$, and the R\'{e}nyi development of 1 is infinite:
 \[ d_{\uptheta}(1) = (a-1)(a-2)(a-2)\cdots .\]
 Then $x= \sum_{i=0}^{n} b_{i}\uptheta^{i}  $ is greedy if and only if 
 \begin{itemize}
\item[i.)] For all  $i$, $b_{i}\in \{ 0,\dots ,a-1\}$ and
\item[ii.)] There are no subwords of the shape
\begin{align}\label{forbiddenblock} b_{i}\cdots b_{i-k}=  (a-1)(a-2)\cdots (a-2)(a-1),\quad k\geq 1 .\end{align}
Note that we include the case $k=1$ i.e.\ the word $(a-1)(a-1)$.
In the sequel, we refer to a word of this shape as a {\bf {\em forbidden block}}.  
\end{itemize}
\end{enumerate}
\end{exam}

 The Parry characterization {\it does not} apply to Laurent series in $\uptheta^{-1}$:
 \begin{enumerate}
\item[ \ding{202}] For $N(\uptheta )=-1$,
we have the identity
\begin{align}\label{InfFibRel} \uptheta^{k+2}= a\uptheta^{k+1} + a \uptheta^{k-1}+\cdots , \end{align}
in which the right hand side is clearly not greedy, although
it contains no Fibonacci relations.   We call such an expression an {\bf {\em infinite Fibonacci relation}}.
\item[ \ding{203}]  In the case of $N(\uptheta )=1$,  we have the non-greedy $\uptheta$-expansion
\begin{align}\label{InfForBlockGreedy}  \uptheta^{k}= (a-1)\uptheta^{k-1} + (a-2)\left( \uptheta^{k-2} + \uptheta^{k-3} + \cdots \right), \end{align}
in which the right hand side contains no forbidden blocks. 
We define an {\bf {\em infinite forbidden block}}
to be an infinite word of the shape 
\begin{align}\label{InfForBlock} (a-1)(a-2)(a-2)\cdots . \end{align}
\end{enumerate}

\begin{theo} Let $\uptheta$ be a quadratic unit and suppose that $x\in\R$ has the $\uptheta$-expansion
\begin{align}\label{xinfexp} x= \sum^{i=m}_{-\infty} b_{i}\uptheta^{i} , \end{align} in which every partial sum is a greedy polynomial. Then the 
expansion {\rm (\ref{xinfexp})} is greedy if and only if
\begin{enumerate}
\item[1.] $N(\uptheta )= -1$ and it contains no infinite Fibonacci relation.
\item[2.] $N(\uptheta )= 1$ and the word associated to it contains no infinite forbidden block.
\end{enumerate}
\end{theo}

\begin{proof}  We will prove the {\it if } part;  the {\it only if } part is trivial.  Suppose first that $N(\uptheta )= -1$ and that the series (\ref{xinfexp}) is not greedy.
Then the greedy expansion of $x$ must differ from (\ref{xinfexp}) at some power $\uptheta^{i}$.  Without loss of generality, we may assume that this power is $i=m$: thus  the greedy coefficient at $\uptheta^{m}$ is larger than $b_{m}$.
Since the hypothesis on (\ref{xinfexp}) is that all the partial sums are greedy,  the only way to increase the $m$th coefficient from $b_{m}$
to $b_{m}+1$ is if $x$ contains the infinite sum $ a(\uptheta^{m-1}+\uptheta^{m-3}+\cdots )  $.  In fact, it must be equal to this sum, for any additional terms would produce either carries or Fibonacci relations.

Similarly, when $N(\uptheta )= 1$, the only way to increase the $m$th coefficient from $b_{m}$
to $b_{m}+1$ and conserve the hypothesis that all partial sums are greedy, is the existence of an infinite forbidden block in (\ref{xinfexp}): and, again,
$x$ must be equal to such a block.
\end{proof}

We now describe the additive arithmetic of greedy expansions: that is, we elaborate the rules of the ``$\uptheta$-adic adding machine''.

Let us begin with the case $N(\uptheta )=-1$ so that $\uptheta^{2} =a\uptheta +1$. 
Using (\ref{powerrecursion}) we have the following carrying formula:

\begin{prop} For all $1\leq i \leq a$, the greedy expansion of $  (a+i) \uptheta^{n}$ is given by
\begin{align}\label{CarryN=-1}    (a+i) \uptheta^{n} = \uptheta^{n+1} + (i-1) \uptheta^{n} + (a-1)\uptheta ^{n-1} + \uptheta^{n-2}  .\end{align}
\end{prop}


\begin{proof}  
For $i=1$, since $a+1>\uptheta$, (\ref{powerrecursion}) implies
\[   \uptheta^{3} < (a+1)\uptheta^{2} < 2\uptheta^{3} .\]  Thus the greedy expansion of $(a+1)\uptheta^{2}$ begins with $\uptheta^{3}$:
\[     (a+1)\uptheta^{2} =_{\rm gr} \uptheta^{3}+x = a\uptheta^{2}+ \uptheta + x. \]
Solving for $x$, we find 
\[ x= \uptheta^{2}-\uptheta =_{\rm gr} (a-1)\uptheta +1,\]
where the right-hand side is greedy by {\it Example} \ref{QuadraticEx}.  Thus
\[  (a+1)\uptheta^{2} =_{\rm gr} \uptheta^{3} +  (a-1)\uptheta +1,   \]
from which we deduce 
\[  (a+1)\uptheta^{n} =_{\rm gr}  \uptheta^{n+1} + (a-1)\uptheta^{n-1} + \uptheta^{n-2}. \]
Adding $(i-1) \uptheta^{n}$ to both sides
gives the desired formula.
\end{proof}


\begin{exam} For $\uptheta=\upvarphi$ = the golden ratio, $a=1$ and thus $i=1$, so the carrying formula (\ref{CarryN=-1}) reduces to the relation
\begin{align}\label{FibCarry}  2 \upvarphi^{n} = \upvarphi^{n+1}  + \upvarphi^{n-2} .\end{align}
\end{exam}

\vspace{3mm}

We now pass to the case $N(\uptheta )=1$ i.e.\  $\uptheta^{2} = a\uptheta -1$, which has the power recursion  (also referred to 
as the {\bf {\em Fibonacci relation}}):
\begin{align}\label{FibRelN=1}   a\uptheta^{n} = \uptheta^{n+1} + \uptheta^{n-1}. \end{align}
In this case, the maximum coefficient of $\uptheta^{n}$ in a greedy expansion is $a-1$ (as opposed to $a$ in the previous case). 
We obtain the following carrying formula:

\begin{prop}    For all $1\leq i\leq a-1$, the greedy expansion of $ (a-1+i) \uptheta^{n} $ is given by
\begin{align}\label{CarryN=1} (a-1+i) \uptheta^{n} =_{\rm gr}  \uptheta^{n+1} + (i-1)\uptheta^{n} + \uptheta^{n-1}.\end{align}
\end{prop}

\begin{proof}
Adding $(i-1)\uptheta^{n}$ to both sides of (\ref{FibRelN=1}) gives (\ref{CarryN=1}), greedy by 
 {\it Example} \ref{QuadraticEx}.
\end{proof}

Later on it will be useful to have the following result which indicates how to resolve in greedy fashion a prohibited block

\begin{prop}\label{ProhibitedBlockRes} For all $n$ and all $k\geq 1$, 
\[    (a-1)\uptheta^{n} + (a-2)\uptheta^{n+1} + \cdots +  (a-2)\uptheta^{n+k-1} + (a-1)\uptheta^{n+k}  = \uptheta^{n-1} + \uptheta^{n+k+1}.    \]
\end{prop}

\begin{proof} If $k=1$, then one sees using the Fibonacci relation $a\uptheta = \uptheta^{2} +1$ that
\[  (a-1) \uptheta^{n} + (a-1) \uptheta^{n+1} = \uptheta^{n-1} -\uptheta^{n}  + a\uptheta^{n+1}      =  \uptheta^{n-1}+ \uptheta^{n+2} .\]
The general result follows by induction.
\end{proof}
 
 In base 10 arithmetic (or any integer base arithmetic), using the usual carrying formulas, one may perform the sum of two decimal expansions inductively starting from the smallest powers of $10$ and progressing
 upwards.  This works well because in these cases, the carrying formula contributes only ``to the left'': that is for $1\leq i\leq 9$,
 \[ (9+i) 10^{n} = 10^{n+1} + (i-1)10^{n} .\]
 
 However, owing to the fact that no sum of coefficients of two $\uptheta$-series can equal $\uptheta$,  the carrying formulas (\ref{CarryN=-1}), (\ref{CarryN=1})  for quadratic $\uptheta$ involve carries to the left as well as to the right, and moreover, in the case of  (\ref{CarryN=-1}), the right carries
 advance two positions to the right.  Thus, when resolving the sum of greedy expansions into a greedy expansion, one may have to return to previously summed powers owing to contributions coming
 from carries taking place further down the line.  It is this phenomenon which can make $\uptheta$-adic arithmetic particularly complex.
 
\begin{exam} We illustrate the process of putting a sum of greedy $\upvarphi$-polynomials into greedy form: first applying the carrying formula (\ref{FibCarry}) for $\upvarphi$ to reduce all coefficients
to the admissible range $\{ 0,1\}$, then applying
the Fibonacci relation $\upvarphi^{n+1}= \upvarphi^{n} + \upvarphi^{n-1}$.
Consider $f=_{\rm gr} \upvarphi^{7} + \upvarphi^{5} + \upvarphi^{2}$ and $g=_{\rm gr} \upvarphi^{5} + \upvarphi^{2} +1$.  Then
 \begin{align*}
 f+ g  &=  \upvarphi^{7}+2 \upvarphi^{5} + 2 \upvarphi^{2} + 1 \\
 & =  \upvarphi^{7} +  \upvarphi^{6} +  2\upvarphi^{3} + 2  \\
 & = \upvarphi^{7} +  \upvarphi^{6} +  \upvarphi^{4} +  2\upvarphi + \upvarphi^{-2}  \\
 & =  \upvarphi^{7} +  \upvarphi^{6} +  \upvarphi^{4} +\upvarphi^{2} + \upvarphi^{-1} +  \upvarphi^{-2}\\
  & =  \upvarphi^{8} +  \upvarphi^{4} +\upvarphi^{2} +1.
 \end{align*}
 Notice that the carry at $\upvarphi^{2}$ in the first line forced an additional carry to the right at $\upvarphi^{0}$.
 \end{exam}
 
  \begin{coro} $\Z_{\uptheta}$ is not closed with respect to the operations  of sum and product.
 \end{coro}
 
\begin{proof} The carrying formulas show that $\Z_{\uptheta}$ is not closed with respect to the sum.  For $N(\uptheta )=-1$ and $a\geq 2$,
\[  a(2+\uptheta) = (2a-1) + \uptheta^{2} =_{\rm gr} \uptheta^{-2} + (a-1)\uptheta^{-1} + (a-2) +\uptheta +\uptheta^{2} \not\in\Z_{\uptheta} \]
and for $a=1$ and $\uptheta =\upvarphi$,
\[ (1+\upvarphi ^{2})^{2} = 1+ 2\upvarphi ^{2} +\upvarphi^{4} = 2+\upvarphi^{5} =_{\rm gr}  \upvarphi^{-2} +\upvarphi +\upvarphi^{5} \not\in\Z_{\uptheta} .
\]
When $N(\uptheta )=-1$,
\[ (a-1)(1+\uptheta) = \uptheta^{-1} +\uptheta^{2}\not\in\Z_{\uptheta}.  \]
\end{proof}

We close this section with an explicit description of $\Z_{\uptheta}$ as a {\it quasicrystal}, or more specifically, a {\it model set} in the sense of Meyer \cite{Meyer}, \cite{Moody}, \cite{BG}.  We recall these definitions now.

Let $\Upgamma$ be a lattice in $\R^{n}$, and suppose we have a direct sum decomposition, \[ \R^{n}=V_{1}\oplus V_{2}, \quad V_{1},  V_{2}
\subset \R^{n},
\] for which
the two associated projections $p_{1},p_{2}$ satisfy 
\[ {\rm ker}(p_{1})\cap\Upgamma=\{0\} \quad \text{and}\quad  p_{2}(\Upgamma ) \text{ is dense in } V_{2}.\]  Let ${\sf W}\subset V_{2}$ be a relatively compact subset with interior, called
the {\bf {\em acceptance window}}. The associated 
{\bf {\em model set}}  is
\[ {\sf Mod}(\Upgamma ; V_{1},V_{2} ;{\sf W}) :=\big\{ p_{1}(\upgamma) \in p_{1}(\Upgamma)\; : \;\; p_{2}(\upgamma)\in {\sf W}\big\} .\]

A {\bf {\em quasicrystal}} is a Delaunay set (a set which is relatively dense and uniformly discrete),  \[ \Upomega\subset\R^{n} ,\] which is an approximate abelian group: there exists a finite set $F\subset\R^{n}$ 
with \[ \Upomega - \Upomega\subset \Upomega + F.\]  By Theorem 1 of \cite{Meyer}, every model set is a quasicrystal.

Now let $\uptheta>1$ be a quadratic unit, $K=\Q(\uptheta)$ and $\mathcal{O}_{K}$ = the ring of $K$-integers.  Then $\mathcal{O}_{K}$ embeds as a lattice in $\R^{2}$; we suppose
that we have chosen an embedding $\mathcal{O}_{K}\subset\R$ and write elements of $\mathcal{O}_{K}\subset\R^{2}$ in the form $(\upalpha,\upalpha')$,  i.e.,\ $\upalpha'$ is the Galois conjugate of $\upalpha$.
If we take $V_{1}=\R\times\{0\}$
and $V_{2}=\{ 0\} \times\R$, $\mathcal{O}_{K}\cap {\rm ker}(p_{1})=\{ 0\}$ and $p_{2}(\mathcal{O}_{K})\subset V_{2}$ is dense.   

 In what follows, if $X\subset\R$ then $X^{+}$ = the non-negative part of $X$. The following appears in \cite{BFGK}:

\begin{theo}\label{IntegersAreQCs} Let $\uptheta$ be a real quadratic unit satisfying $\uptheta^{2}= a\uptheta+n$, $n=\pm1$.
 Then 
\[  \Z_{\uptheta}^{+} = {\sf Mod}(\mathcal{O}_{K};V_{1} ,V_{2}; {\sf W} )^{+} \]
where
\begin{enumerate}
\item[1.] If $n=1$ i.e. $N(\uptheta )=-1$ then
\[  {\sf W}:=\left[    \frac{a\uptheta'}{1-\uptheta'^{2}}, \frac{a}{1-\uptheta'^{2}}\right)  =\left[    \frac{-a}{\uptheta-\uptheta^{-1}}, \frac{a}{1-\uptheta^{-2}}\right)  = [-1,\uptheta ) .\]
\item[2.]  If $n=-1$ i.e. $N(\uptheta )=1$ then
\[ {\sf W}:=\left[   0, 1+\frac{a-2}{1-\uptheta'}\right)=\left[   0, 1+\frac{a-2}{1-\uptheta^{-1}}\right) = [0,\uptheta).\]
\end{enumerate}
\end{theo}

In Proposition 5.2 of \cite{Gaz1}, we note that it is proved that the only irrational numbers $\uptheta$ for which the analog of Theorem \ref{IntegersAreQCs} holds are quadratic Pisot units.  


 \section{The $\uptheta$-adic Infranorm }\label{InfraNormSection}
 
 In this section, we develop an archimedean analog of the $p$-adic norm, using the greedy expansions defined in \S \ref{GreedySection}.  
 
 \vspace{3mm}
 
 \noindent \fbox{$\boldsymbol N\boldsymbol(\boldsymbol\uptheta \boldsymbol)\boldsymbol=\boldsymbol-\boldsymbol1$}
 
  \vspace{3mm}
 
 In this case,
the Galois conjugate of $\uptheta$ satisfies $\uptheta'=-\uptheta^{-1}$.  Denote $K=\Q (\uptheta )\subset\R$
 and $\mathcal{O}_{K}$ the ring of $K$-integers.   Note that
if $0<x\in\R$ has a finite greedy $\uptheta$-expansion,
 \begin{align}\label{finiteexpansion} x =_{\rm gr}  \sum^{M}_{i=m} b_{i}\uptheta^{i}, \end{align}
 then $x\in\mathcal{O}_{K}$ (since $\uptheta\in \mathcal{O}_{K}^{\times}$).  The converse is also true: it is a special case of Theorem 2 of \cite{FS}, we nonetheless provide
 a proof here, since it illustrates the technique of using model sets in this setting.
 
 \begin{prop} The elements of $\R$ having finite greedy $\uptheta$-expansion are exactly the elements of  $\mathcal{O}_{K}$:
 \[ \mathcal{O}_{K}= \left\{ x\in\R\; :\;\;   |x|=_{\rm gr}  \sum^{M}_{m} b_{i}\uptheta^{i} \text{ for some }m,M\in\Z. \right\}.\]
 \end{prop}
 
 \begin{proof} By Theorem \ref{IntegersAreQCs}, the set of $x>0$ for which $x$ has a finite greedy $\uptheta$-expansion (\ref{finiteexpansion}) with $m\geq 0$ may be identified with the positive part of 
  ${\sf Mod}(\mathcal{O}_{K}; V_{1},V_{2}; {\sf W})$ where ${\sf W}=\left[    \frac{a\uptheta'}{1-\uptheta'^{2}}, \frac{a}{1-\uptheta'^{2}}\right)=[-1,\uptheta)$. 
 Now given $0<x \in \mathcal{O}_{K}$, there exists a multiple $\uptheta^{M}x>0$ with $\uptheta'^{M}x'\in {\sf W}$, hence $\uptheta^{M}x  =_{\rm gr} \sum_{i=m}^{N} b_{i}\uptheta^{i} $ for $m\geq 0$, and since the greedy
 property is preserved under multiplication by a power of $\uptheta$,  
\[ x=_{\rm gr}  \sum_{i=m-M}^{N-M} b_{i+M}\uptheta^{i} .\]
 Thus every element of $\mathcal{O}_{K}$ has  a finite greedy $\uptheta$-expansion.  Since $\uptheta\in\mathcal{O}^{\times}_{K}$, every greedy Laurent $\uptheta$-polynomial
 is in $\mathcal{O}_{K}$.
 \end{proof}
 
 Let $X$ be a set.  An {\bf {\em  inframetric}} (see \cite{FLV}, Definition 1) is a function $d:X\times X\rightarrow \R_{+}$
which is symmetric, non-degenerate and satisfies the  {\bf {\em infratriangle inequality}} : there exists $\uprho \geq 1$ such that for all $ x,y,w\in X$,
\[  d(x,y) \leq \uprho \max\left\{  d(x,w) ,\; d(w,y) \right\} .\]
 If $\uprho =1$ one recovers the notion of an ultrametric; an ordinary metric is an inframetric with $\uprho=2$.  To highlight $\uprho$ we will sometimes refer to $d$ as a  {\bf {\em  $\boldsymbol\uprho$-inframetric}}.

For $G$ an abelian group, an  {\bf {\em infranorm}} \cite{FLV} is a function 
 \[ | \cdot | :G\longrightarrow \R_{+},\]
for which $|x|=0$ $\Leftrightarrow x=0$, $|-x|=|x|$ and which satisfies the infratriangle inequality 
\[ |x\pm y| \leq \uprho\max \{| x| , |y|\} .  \] 
 To an infranorm we may associate the inframetric $d(x,y) = |x-y|$.
 
For $x\in\mathcal{O}_{K}$ for which $|x|$ has the expansion (\ref{finiteexpansion}), we define the $\boldsymbol\uptheta${\bf {\em -adic infranorm}} by
 \[  | x|_{\uptheta} : = \uptheta^{-m} . \]
We note that $|\cdot |_{\uptheta}$ fails the usual triangle inequality:  for example, for $\uptheta=\upvarphi=$ the golden ratio, we have 
 \[ |1+1|_{\upvarphi}  = \left|2=_{\rm gr} \upvarphi + \upvarphi^{-2}\right|_{\upvarphi} = \upvarphi^{2}\not\leq   |1|_{\upvarphi} +|1|_{\upvarphi}  =2. \]
 
  \begin{prop}\label{thetascaling} For all $x\in \mathcal{O}_{K}$ and $m\in\Z$, 
  \[ | \uptheta^{n} x|_{\uptheta} = \uptheta^{-n} |x|_{\uptheta}. \]
  \end{prop}
  
  \begin{proof}  Follows trivially from 
  \[ x =_{\rm gr} \sum_{i=m}^{M} b_{i}\uptheta^{i} \;\; \Longleftrightarrow \;\;\uptheta^{n} x =_{\rm gr} \sum_{i=m}^{M} b_{i}\uptheta^{i+n}.\]
  \end{proof}
 
 \begin{theo}[Infratriangle Inequality]\label{infratriangle}  For all $f,g\in\mathcal{O}_{K}$, 
\[   |f\pm g|_{\uptheta} \leq \uptheta^{2} \max\left\{  \; |f|_{\uptheta},    |g|_{\uptheta} \right\} . \]
\end{theo}

\begin{proof} Assume first $f,g>0$.   By Proposition \ref{thetascaling}, we may suppose (after scaling by an appropriate power $\uptheta^{n}$) that $f,g\in\Z_{\uptheta}$, and that furthermore, $|f|_{\uptheta}=1$.  Thus, 
 $f=_{\rm gr}\sum_{i=0}^{M} b_{i}\uptheta^{i}>0$,  $g=\sum_{i=n}^{N} c_{i}\uptheta^{i}>0$, so that $\max \{ |f|_{\uptheta}, |g|_{\uptheta} \}=1$ and 
\[ f+g\in {\sf Mod}(\mathcal{O}_{K};V_{1}, V_{2}; 2{\sf W})^{+}, \]
where ${\sf W}=[-1,\uptheta )$.  Since $0<\uptheta'^{2}=\uptheta^{-2}<1/2$, if we scale further by $\uptheta^{2}$, we obtain
$ \uptheta^{2} (f+g) \in {\sf Mod}(\mathcal{O}_{K};V_{1}, V_{2}; {\sf W})$.
Hence there exist $d_{i}\in \{ 0,\dots ,a\}$ and $r,R\in\N$ such that \[ 
f+g =_{\rm gr} \uptheta^{-2}\sum^{R}_{i= r} d_{i}\uptheta^{i},
\]
from which the infratriangle inequality follows for $f+g$.  Suppose now that one of the two is negative, i.e., we consider  $f>0$ and $-g<0$.  We may assume without loss of generality that $f-g>0$ and that both $f,g\in\Z_{\uptheta}$.  Then 
\[ (f-g)' \in  {\sf W}-{\sf W} = (-1 - \uptheta, 1+\uptheta)
 .\]
 In the above, we note that although ${\sf W}$ is half-open, half-closed, ${\sf W}-{\sf W}$ is in fact open.
Hence $\uptheta^{2}(f-g)$ has conjugate 
\[ \uptheta^{-2}(f'-g')\in
 \left( -\left(\uptheta^{-2} + \uptheta^{-1}\right), \uptheta^{-2} + \uptheta^{-1}      \right)\subset {\sf W},\]
 since $\uptheta^{2}\geq \uptheta+1$.
Again, the infratriangle inequality for $f-g$ follows.
(Note that for $\uptheta\not=\upvarphi$, although $|\uptheta'|<1/2$, there is a negative sign present in $\uptheta'$, which means we cannot
simply scale by $\uptheta$ in the above argument.)
\end{proof}

The proof of Theorem \ref{infratriangle} illustrates the power of the model set construction; without it, a proof by hand of the infratriangle inequality would be hopelessly convoluted, due to the complexities introduced by the carrying formulas.

\begin{coro} $|\cdot |_{\uptheta}$ is a $\uprho$ infranorm for $\uprho =\uptheta^{2}$.
\end{coro}

The $\mathfrak{p}$-adic norms of algebraic number theory are multiplicative; the analog of this property is the following notion.
We say an infranorm $|\cdot |$ on a ring $R$ is {\bf k-{\em inframultiplicative}} if there exists a constant $\varrho\geq 1$ such that for for all $x_{1},\dots x_{k}\in R$,
\[  \varrho^{-k} |x_{1}|\cdots |x_{k}| \leq |x_{1}\cdots x_{k}|\leq \varrho^{k} |x_{1}|\cdots |x_{k}| .  \]


 \begin{theo}[Inframultiplicativity]\label{inframultN=-1}  $|\cdot |_{\uptheta}$ is ${\rm k}$-inframultiplicative for all $k\geq 2$ with $\varrho =\uptheta$. 
\end{theo}

\begin{proof} It suffices by Proposition \ref{thetascaling} to take $|f_{i}|_{\uptheta}=1$ for $i=1,\dots ,k$.  We have
\[ f_{1}'\cdots f_{k}'  \in {\sf W}^{k} = \uptheta^{k-1}(-1, \uptheta),
\]
and since $\uptheta'^{k}= (-1)^{k}\uptheta^{-k}$,
\[   (\uptheta^{k} f_{1}\cdots f_{k})'  \in  \left\{ \begin{array}{ll} 
 \uptheta^{-k}{\sf W}^{k} = (-\uptheta^{-1},1)\subset {\sf W}  & \text{ if $k$ even} \\  
  \\
 - \uptheta^{-k}{\sf W}^{k} = (-1, \uptheta^{-1})\subset {\sf W}  & \text{ if $k$ odd}
 \end{array} \right. \]
Thus $\uptheta^{k} f_{1}\cdots f_{k} =_{\rm gr} \sum_{i=s}^{S} e_{i}\uptheta^{i}$ for $s,S\in\N$ and appropriate $e_{i}\in \{ 0,\dots ,a\}$, and the upper bound follows.

Now consider the other inequality.  
We claim that, for any $h>0$ with $h\in\mathcal{O}_{K}$ and $|h|_{\uptheta}=1$, 
\begin{align*} h' > \uptheta^{-1}.\end{align*}  
Indeed, since 1)  $h$ has a nontrivial constant term and 2) only the conjugates of odd powers of $\uptheta$ carry a sign, the greedy form of $h$ implies the
following strict inequality 
\begin{align} \label{hyponh'}
 h' & >  1- (a-1)\uptheta^{-1} -a\uptheta^{-3} - a\uptheta^{-5} - \cdots 
  = \uptheta^{-1}.
\end{align}
In the above, we have made use of 
\begin{enumerate}
\item[1.] the explicit characterization of greedy expansions given in  {\it Example} \ref{QuadraticEx} to conclude that the coefficient of $\uptheta$ is at worst $a-1$,
and that the remaining coefficients of the odd powers of $\uptheta$ are at worst $a$.  This gives the strict inequality in the first line.
\item[2.] the infinite Fibonacci relation (\ref{InfFibRel}) to conclude the equality in the last line.
\end{enumerate}
Thus for each $i$, $f_{i}'>\uptheta^{-1}$ and hence $(f_{1}\cdots f_{k})'> \uptheta^{-k}$.  On the other hand, if we were to have $|f_{1}\cdots f_{k}|_{\uptheta}= \uptheta^{-r}$ where $r\geq k+1$,  
then the greedy expansion of $f_{1}\cdots f_{k}$ would give
\[ |(f_{1}\cdots f_{k})'|  < a\uptheta^{-(k+1)} + a\uptheta^{-(k+3)} + \cdots =\frac{a\uptheta^{-(k+1)}}{1-\uptheta^{-2}} = \uptheta^{-k},\]
contradiction.  This proves the lower bound.
\end{proof}

\begin{exam} We give examples showing that the lower bound in Theorem \ref{inframultN=-1} is sharp for $k=2$. Let $f=g=1+\upvarphi^{3}$, where $\upvarphi$ is the golden ratio.
Then 
\[ f\cdot g= 1+2\upvarphi^{3} + \upvarphi^{6}  = 1+\upvarphi + \upvarphi^{4} +\upvarphi^{6} =_{\rm gr} \upvarphi^{2} +  \upvarphi^{4}+ \upvarphi^{6} ,  \]
hence the lower bound is sharp since
\[ |f\cdot g|_{\upvarphi} =\upvarphi^{-2} = \upvarphi^{-2}  |f|_{\upvarphi}\cdot    |g|_{\upvarphi}
. \]
Similarly, for $f=g=1+\upvarphi^{2}$ we have
\[ f\cdot g = 1 +2\upvarphi^{2}+ \upvarphi^{4} = 2+ \upvarphi^{5} =_{\rm gr} \upvarphi^{-2} +\upvarphi + \upvarphi^{5},\]
which shows that the upper bound is sharp, since
\[ |f\cdot g|_{\upvarphi} =\upvarphi^{2} = \upvarphi^{2}  |f|_{\upvarphi}\cdot    |g|_{\upvarphi}.\]
\end{exam}

In the special case $f_{i}=f$ for all $i$ we have 
\begin{coro}\label{InfraPower} $  \uptheta^{-k} |f|_{\uptheta}^{k} \leq  | f^{k}|_{\uptheta} \leq   \uptheta^{k}|f|_{\uptheta}^{k}  $.
 \end{coro} 
  


\begin{exam} In this example we show that different powers can have the same infranorm, illustrating a computational ambiguity 
inherent in inframultiplicativity.
Consider $g=2$ and $\uptheta$ with $a=3$.  Then $|g|_{\uptheta}=1$ and
\[ 2^{2} = 4 = \uptheta^{-2} + 2\uptheta^{-1} + \uptheta\]
so $|g^{2}|_{\uptheta} = \uptheta^{2}$.  On the other hand, 
\[2^{3} =2\uptheta^{-2} + 4\uptheta^{-1} + 2\uptheta=   \uptheta^{-3}+ 4\uptheta^{-2} + 1+ 2\uptheta  =\uptheta^{-4}+    3\uptheta^{-3}+  \uptheta^{-1}+1+ 2\uptheta = \uptheta^{-2} +\uptheta^{-1}+1+ 2\uptheta .\]
Thus 
\[ |g^{2}|_{\uptheta} = |g^{3}|_{\uptheta}.\]
\end{exam}


 
 \vspace{3mm}
 
\noindent \fbox{$\boldsymbol N\boldsymbol (\boldsymbol \uptheta \boldsymbol )\boldsymbol =\boldsymbol1$}
 
  \vspace{3mm}
  
  Here, the conjugate of $\uptheta$ satisfies $\uptheta' = \uptheta^{-1}$.  Note that this implies that for any {\it finite} greedy Laurent polynomial
  \[  \upalpha =_{\rm gr} \sum_{i=n}^{N}b_{i}\uptheta^{i}    \]
  we have 
   \[  \upalpha' =_{\rm gr} \sum_{i=n}^{N}b_{i}\uptheta^{-i}.    \]
   That is, the greedy condition on Laurent polynomials is invariant with respect to Galois conjugation.
   This follows since the obstruction to the greedy condition -- the existence of a forbidden block -- is symmetric with respect to the operation $\uptheta^{i}\mapsto \uptheta^{-i}$.  Of course, this observation is false in the case $N(\uptheta )=-1$, since the conjugate of $\uptheta$ acquires a sign.

Also in contrast with the case $N(\uptheta )=-1$, the set 
  \[  \mathcal{O}^{0}_{K} := \{ \upalpha \in \mathcal{O}_{K}\; :\;\; \upalpha \text{ has a finite greedy expansion}\}  = \bigcup_{N\in\Z} \uptheta^{N}\Z_{\uptheta} \]
  {\it does not} coincide with $\mathcal{O}_{K}$.    To see this, let us note that
  \begin{align}\label{signid}
   \mathcal{O}^{0}_{K} = \{ \upalpha\in\mathcal{O}_{K} \; :\;\; {\rm sgn}(\upalpha ) = {\rm sgn}(\upalpha ') \} \cup \{ 0\}.
  \end{align}
  Indeed, by Theorem \ref{IntegersAreQCs},  
  
  \[  \mathcal{O}^{0}_{K} = \bigcup_{N\in\Z} \uptheta^{N}\Z_{\uptheta}=  \pm \bigcup \uptheta^{N}\left[ {\sf Mod}(\mathcal{O}_{K}, V,V', [0,\uptheta) \right]^{+}
  =\pm  \left[ {\sf Mod}(\mathcal{O}_{K}, V,V', \R_{\geq 0})\right]^{+} ,
  \]
  which is in turn equal to the right-hand side of (\ref{signid}).   Thus an element $\upalpha\in \mathcal{O}_{K}$ for which ${\rm sgn}(\upalpha )\not={\rm sgn}(\upalpha' )$ does not have a finite greedy $\uptheta$-expansion.   An example of such 
  an element is 
  \[ T:= \uptheta -1 =_{\rm gr} (a-2) \sum_{i=0}^{\infty}\uptheta^{-i}, \]
where the equality follows from the R\'{e}nyi  expansion $1= (a-1)\uptheta^{-1} + (a-2)\sum_{i=2}^{\infty}\uptheta^{-i}$.

  By a {\bf ringoid} $R\subset\R$ we shall mean a set $R$ for which $(R,+)$ is an additive abelian groupoid and $(R,\cdot )$ is a multiplicative monoid.
  
  \begin{prop} $ \mathcal{O}^{0}_{K} \subsetneq \mathcal{O}_{K}$ is a ringoid.
  \end{prop}
  
  \begin{proof} The sum is only partially defined: indeed, as observed above,  $T=\uptheta-1>0$ does not have a finite greedy
  expansion since its conjugate is negative.  On the other hand, the product is globally defined by the identification (\ref{signid}).
  \end{proof}
  
  \begin{note}\label{SameSignSumDefined} The sum of two elements having the same sign is {\it always} defined in $ \mathcal{O}^{0}_{K}$.  In fact, the subsets 
  $ \mathcal{O}^{0+}_{K}$, $ \mathcal{O}^{0-}_{K}$ of non-negative and non-positive elements are additive monoids: thus, the partially defined character of the sum in $ \mathcal{O}^{0}_{K}$
only arises in the consideration of sums of elements of different sign e.g. $\uptheta + (-1) = T\not\in  \mathcal{O}^{0}_{K}$.
  \end{note}

  \begin{theo}\label{mixedsigns} If $\upalpha\in\mathcal{O}_{K}\setminus  \mathcal{O}^{0}_{K}$ then
  \[ |\upalpha|=_{\rm gr} \upbeta+ \uptheta^{m}T, \]
  where $ \upbeta\in  \mathcal{O}^{0}_{K}$ and $ m\in\Z$.
  \end{theo}
  
  \begin{proof} Without loss of generality, assume $\upalpha>0$ so that $\upalpha'<0$.  Let $m$ be the largest integer for which $(\upalpha +\uptheta^{m})'>0$.  Then
  \[ \upalpha  + \uptheta^{m}=_{\rm gr}  \sum_{i=n}^{N}b_{i}\uptheta^{i}  \quad\text{i.e.}\quad  \upalpha= \sum_{i=n}^{N}b_{i}\uptheta^{i} -\uptheta^{m},\]
  where as usual it is understood that $b_{n}\not=0$.  We claim that $m=n-1$.  Indeed, as we are in the case $N(\uptheta )=1$, the conjugate sum
  \[ \sum_{i=n}^{N}b_{i}\uptheta'^{i} =\sum_{i=n}^{N}b_{i}\uptheta^{-i}  \] 
  is also greedy, hence is strictly less than $\uptheta^{-(n-1)}$, from which the claim follows.  Therefore, 
  \[ \upalpha=_{\rm gr} \sum_{i=n+1}^{N}b_{i}\uptheta^{i}  + (b_{n}-1) \uptheta^{n} +\uptheta^{n-1}T .  \]
   \end{proof}
  Let us define
  \begin{align*} \mathcal{O}^{1}_{K} & := \{ \upalpha \in \mathcal{O}_{K} \; :\; |\upalpha | =_{\rm gr} \sum_{i=n>m}^{N}b_{i}\uptheta^{i} + \uptheta^{m}T,\;\; m, n, N\in\Z\} \cup \{ 0\}\\
  & =  \{ \upalpha \in \mathcal{O}_{K} \; : \; {\rm sgn} (\upalpha)\not={\rm sgn} (\upalpha') \} \cup \{0\},\end{align*}
  where the last equality is a consequence of Theorem \ref{mixedsigns}.
  Thus $(\mathcal{O}^{1}_{K} ,+)$ is an additive groupoid, and we have
\[ \mathcal{O}^{1}_{K}\cdot \mathcal{O}^{1}_{K}\subset \mathcal{O}^{0}_{K},\quad  \mathcal{O}^{1}_{K}\cdot \mathcal{O}^{0}_{K}\subset \mathcal{O}^{1}_{K}  \]
and \[  \mathcal{O}_{K}= \mathcal{O}^{0}_{K}\cup \mathcal{O}^{1}_{K},\quad \{0\} =\mathcal{O}^{0}_{K}\cap \mathcal{O}^{1}_{K} .\]
We may regard  $\mathcal{O}^{1}_{K}$ as a ``moduloid'' over the ringoid $\mathcal{O}^{0}_{K}$ and similarly we may regard   $\mathcal{O}_{K}$ as a $\Z/2\Z$ graded ringoid.

We will refer to a subset $R\subset(\R,+)$ as an {\bf approximate abelian group}\footnote{The original  notion of approximate group arose in the work of Yves Meyer \cite{Meyer0},
\cite{Meyer} in his definition of quasicrystal.  Later, Breuillard, Green, and Tao \cite{BGT}, as well as Hrushovski \cite{Hrush}, have developed these ideas in a somewhat more abstract setting.}
 if there exists a finite set $F\subset \R$ such that 
\[    R-R\subset R+F.  \] 
By an {\bf approximate ring} (with 1) $R\subset\R$ is meant an additive approximate group which is also a multiplicative monoid.

\begin{prop} $\mathcal{O}^{0}_{K}$ is an approximate ring in which 
\[  \mathcal{O}^{0}_{K} - \mathcal{O}^{0}_{K}=  \mathcal{O}^{0}_{K} + \mathcal{O}^{0}_{K} \subset \mathcal{O}^{0}_{K} +T. \]
\end{prop}

\begin{proof} 
Since
\[  \mathcal{O}^{0}_{K} + \mathcal{O}^{0}_{K} \subset \mathcal{O}^{0}_{K} + \bigcup_{N\in\Z} \{ \uptheta^{N}T\},\]   it will suffice to show that
for all $N\in\Z$, $\uptheta^{N}T\in \mathcal{O}^{0}_{K} +T$.    First assume $N\geq 0$; then if we write $\upbeta = \uptheta^{N}T-T$, we have
$\upbeta >0$ and moreover,
\[  \upbeta' = T' (\uptheta'^{N}  -1)  = (\uptheta^{-1}-1)( \uptheta^{-N}-1)>0 .\]
Thus $\upbeta \in  \mathcal{O}^{0}_{K}$ and $\uptheta^{N}T= \upbeta + T\in\mathcal{O}^{0}_{K} +T $.  In the case where $N<0$, the same argument shows ${\rm sgn}(\upbeta) = {\rm sgn}(\upbeta')=-$,
which again gives $\upbeta \in  \mathcal{O}^{0}_{K}$. 
\end{proof}

The $\uptheta$-adic infranorm is defined in $ \mathcal{O}^{0}_{K}$ exactly as in the case $N(\uptheta )=-1$; we extend this to $ \mathcal{O}^{1}_{K}$  as follows.
Since $T \mathcal{O}^{1}_{K} \subset \mathcal{O}^{0}_{K}$, every element $\upalpha\in  \mathcal{O}^{1}_{K}$ may be written uniquely in the form
\begin{align}\label{TinvFormO1}  \upalpha = T^{-1} \upalpha^{0},\quad \upalpha^{0} \in \mathcal{O}^{0}_{K}.  \end{align}
Then we define
\begin{align}\label{defnO1norm}   | \upalpha |_{\uptheta} := |T\upalpha |_{\uptheta} =   | \upalpha^{0} |_{\uptheta} . \end{align}

 The remainder of this section is devoted to proving the relevant versions of the infratriangle inequality and inframultiplicativity, which will differ from the norm $-1$ case with respect to the values of the adjustment constants
 but qualitatively will give the same structure.
 
 
 \begin{lemm}\label{ConjNormLemma} Let $0<f\in \mathcal{O}^{0}_{K}$.
 Then  \begin{align}\label{ConjNormComp}   |f|_{\uptheta}\leq  f' \leq  \uptheta |f|_{\uptheta}  . \end{align}
   \end{lemm}
     \begin{proof} For
  \[  f=_{\rm gr} \sum_{i=m}^{M} b_{i}\uptheta^{i} , \]
$0<f'\in \mathcal{O}^{0}_{K}$ has greedy expansion
  \begin{align*}  f' =_{\rm gr} \sum_{i=-M}^{-m} b_{-i}\uptheta^{i}  \end{align*} which gives (\ref{ConjNormComp}).
  \end{proof}

It what follows we will use the fact that multiplication by a power of $\uptheta$ is fully multiplicative: for any $ \upalpha \in\mathcal{O}_{K}$ (independently of whether it is in $\mathcal{O}^{0}_{K}$ or $\mathcal{O}^{1}_{K}$)
we have
  \[ \left| \uptheta^{N} \upalpha  \right|_{\uptheta} =  \left | \uptheta^{N} \right|_{\uptheta}  \left |  \upalpha  \right|_{\uptheta} =\uptheta^{-N} \left |  \upalpha  \right|_{\uptheta}.\]
  
  \begin{theo}[Inframultiplicativity]\label{N=1InfraMult}  Let $f,g\in \mathcal{O}_{K}$.  
  \begin{enumerate}
\item[1.] If $f,g\in    \mathcal{O}^{0}_{K}$, then
  \[   |f|_{\uptheta} |g|_{\uptheta}\leq |fg|_{\uptheta} \leq \uptheta  |f|_{\uptheta} |g|_{\uptheta}. \]
  \item[2.]   If $f,g\in    \mathcal{O}^{1}_{K}$, then
  \[  |f|_{\uptheta} |g|_{\uptheta}\leq |fg|_{\uptheta} \leq \uptheta^{2}  |f|_{\uptheta} |g|_{\uptheta}. \]
  \item[3.] If $f  \in  \mathcal{O}^{0}_{K}$ and  $g\in    \mathcal{O}^{1}_{K}$, then
   \[   |f|_{\uptheta} |g|_{\uptheta}\leq |fg|_{\uptheta} \leq \uptheta  |f|_{\uptheta} |g|_{\uptheta}. \]
    \end{enumerate}
  \end{theo}
  
  \begin{proof} 1. It will be enough to prove the result in the case $ |f|_{\uptheta} =1=|g|_{\uptheta}$ and $f,g>0$ (thus, in particular, $f,g\in\Z_{\uptheta}^{+}$).  This assumption means that we have 
  \[ f=\sum_{i=0}^{M} b_{i}\uptheta^{i}, \quad g =\sum_{i=0}^{N} c_{i}\uptheta^{i},\]
  with $b_{0}\not = 0 \not= c_{0}$.  
Since $\uptheta'=\uptheta^{-1}>0$, it follows that for such sums we have $f', g'\geq 1$.  In general, for $h=_{\rm gr}\sum_{i\geq r}^{R} d_{i}\uptheta^{i}$, $d_{r}\not=0$, we claim that   
  \begin{align}\label{f'obversation}   h' \geq 1 \Longleftrightarrow  r\leq 0. \end{align}
  The direction $\Longleftarrow$ is immediate.  As for $\Longrightarrow$,
 if it were the case that $r\geq 1$, then $h'$ is dominated by the R\'{e}nyi expansion of 1 (c.f.\ equation (\ref{InfForBlockGreedy})):
  \begin{align*}  h'  &  =_{\rm gr} d_{r}\uptheta^{-r} + \cdots + d_{R}\uptheta^{-R} \\
  &  =\uptheta^{-r+1} \left( d_{r}\uptheta^{-1} + \cdots + d_{R}\uptheta^{-R+r-1}\right) \\
  & \leq d_{r}\uptheta^{-1} + \cdots + d_{R}\uptheta^{-R+r-1} \\
  &< (a-1)\uptheta^{-1} + (a-2)(\uptheta^{-2}  + \uptheta^{-3}+\cdots ) =1,
  \end{align*}
  contradiction.
The inequality $1\leq |fg|_{\uptheta}$ now follows, since $(fg)'\geq 1$, and the latter implies, by (\ref{f'obversation}), that the greedy expansion of $fg$ is of the form $\sum_{i\geq r}^{R}d_{i}\uptheta^{i}$
  with $r\leq 0$.
 As for the other inequality, let  $f,g\in\Z_{\uptheta}^{+}$.  By Theorem \ref{IntegersAreQCs}, item 2., we have $f',g'\in \left[   0, \uptheta \right)$ implying
$  (fg)' <\uptheta^{2}$ .  Thus, $(\uptheta f g)'<\uptheta$, which implies $\uptheta f g\in \Z_{\uptheta}^{+}$,  i.e.,
\[ \uptheta^{-1}| f g|_{\uptheta} =  |\uptheta f g|_{\uptheta} \leq 1. \]

  
  \vspace{3mm}
  
  \noindent 2.  Note that $fg\in\mathcal{O}^{0}_{K}$ and
  \[ T^{2} = (\uptheta -1)^{2} = \uptheta^{2}-2\uptheta +1 = (a-2)\uptheta.\] Then, by part 1. above, 
  \[   \uptheta^{-1} |fg|_{\uptheta}  =  |(a-2)\uptheta|_{\uptheta} |fg|_{\uptheta}   \leq  |(a-2)\uptheta fg|_{\uptheta} =    |TfTg|_{\uptheta} \leq  \uptheta |f|_{\uptheta}|g|_{\uptheta},  \]
  which gives the upper bound for $|fg|_{\uptheta} $.  Note that in the last inequality the definition (\ref{defnO1norm}) is implicit. Similarly
  \[  |f|_{\uptheta}|g|_{\uptheta} :=|Tf|_{\uptheta}|Tg|_{\uptheta} \leq  |TfTg|_{\uptheta} \leq \uptheta |(a-2)\uptheta|_{\uptheta} |fg|_{\uptheta} =  |fg|_{\uptheta}\]
  giving the lower bound.
  
    \vspace{3mm}
  
  \noindent 3. Here we have, again by 1., 
  \[  |f|_{\uptheta}|g|_{\uptheta} :=|f|_{\uptheta}|Tg|_{\uptheta}  \leq |fTg|_{\uptheta} =: |fg|_{\uptheta}
  \leq  \uptheta |f|_{\uptheta} |Tg|_{\uptheta}  =:  \uptheta  |f|_{\uptheta} |g|_{\uptheta}.
    \]
  \end{proof}

\begin{theo}[infratriangle Inequality]\label{N=1ITE}  Let $f,g\in \mathcal{O}_{K}$.  
  \begin{enumerate}
\item[1.] If $f,g\in    \mathcal{O}^{0}_{K}$, then
  \[  |f+g|_{\uptheta} \leq \uptheta \max \{  | f|_{\uptheta},  | g|_{\uptheta}  \}. \]
  \item[2.]   If $f,g\in    \mathcal{O}^{1}_{K}$, then
  \[  |f+g|_{\uptheta} \leq \uptheta^{2} \max \{  | f|_{\uptheta},  | g|_{\uptheta}  \}. \]
  \item[3.] If $f  \in  \mathcal{O}^{0}_{K}$ and  $g\in    \mathcal{O}^{1}_{K}$, then
   \[  |f+g|_{\uptheta} \leq \uptheta^{4} \max \{  | f|_{\uptheta},  | g|_{\uptheta}  \}. \]
    \end{enumerate}
In particular, $|\cdot |_{\uptheta}$ is an infranorm with $\uprho = \uptheta^{4}$.
\end{theo}

\begin{proof} 1.  First assume that $f+g\in  \mathcal{O}^{0}_{K}$; then the proof is the same as that of Theorem \ref{infratriangle},
except one works now with the window $W=[0,\uptheta )$ and takes advantage of the fact that $\uptheta' = \uptheta^{-1}<1/2$.  In particular, $\uptheta^{-1} 2W\subset W$, which gives
the adjustment factor of $\uptheta$ rather than that of Theorem \ref{infratriangle}, which was $\uptheta^{2}$.  Suppose now that $f,g \in  \mathcal{O}^{0}_{K}$ but $f+g\not\in  \mathcal{O}^{0}_{K}$.
Then we must have ${\rm sgn}(f)\not={\rm sgn}(g)$ so we are essentially dealing with a difference.  Therefore, we will write the norm in question as $|f-g|_{\uptheta}$ where we assume
$f-g>0$ but $-(f'-g')=g'-f'>0$ and $f,g>0$ with $f',g'>0$.  Note that the condition $g'-f'>0$ and the fact that the conjugate of a greedy polynomial is greedy  implies $\max \{ |f|_{\uptheta},  |g|_{\uptheta}\} = |g|_{\uptheta}$.  With these elements in hand, as well as Lemma \ref{ConjNormLemma}, we now deduce
\begin{align*}
 |f-g|_{\uptheta}  := |T(f-g)|_{\uptheta} &  \leq T'(f'-g')   = (1-\uptheta^{-1})(g'-f')\\
 & < g'-f'  < g' \\
 &\leq \uptheta  |g|_{\uptheta}  = \uptheta \max \{ |f|_{\uptheta},  |g|_{\uptheta}\} .
\end{align*}

\vspace{3mm}

\noindent 2.  Assume first that $f+g\in  \mathcal{O}^{1}_{K}$; then $|f+g|_{\uptheta} := |Tf+ Tg|_{\uptheta}$ with $Tf, Tg\in \mathcal{O}^{0}_{K}$, so the result in 1. implies
\[   |f+g|_{\uptheta} \leq \uptheta \max\{ |f|_{\uptheta}, |g|_{\uptheta} \} . \]
Otherwise, if $f+g\in \mathcal{O}^{0}_{K}$, by Inframultiplicativity (Theorem \ref{N=1InfraMult}, Part 3.),
\[   |f+g|_{\uptheta} =\uptheta |T|_{\uptheta} |f+g|_{\uptheta}\leq \uptheta |Tf +Tg|_{\uptheta} \leq \uptheta^{2} \max \{  | f|_{\uptheta},  | g|_{\uptheta}  \};  \]
the last equality follows from 1.\ of the present Theorem, where we have used that $Tf, Tg\in\mathcal{O}^{0}_{K}$.

\vspace{3mm}

\noindent 3.  Assume first that $f+g\in \mathcal{O}^{0}_{K}$.   If $f,g$ have the same sign, we may assume both are positive, in which case:
\begin{align*}
|f+g|_{\uptheta} & = |f+ Tg-(\uptheta -2)g|_{\uptheta} .
\end{align*}
Since $N(\uptheta )=1$, its minimal polynomial has linear term $-a$, $a\geq 3$.  In particular, $\uptheta-a+\uptheta^{-1}=0$ implies $\uptheta-2>0$.
Thus $\uptheta-2\in \mathcal{O}_{K}^{1}$ and $(\uptheta -2)g\in  \mathcal{O}_{K}^{0}$.   Since $f, Tg>0$ and $Tg\in  \mathcal{O}_{K}^{0}$, we have likewise, 
by {\it Note} \ref{SameSignSumDefined}, that $f+ Tg\in  \mathcal{O}_{K}^{0}$.
Thus we may apply the result of part 1.\ of this Theorem, as well as Inframultiplicativity, to get
\begin{align*}
|f+g|_{\uptheta} & \leq \uptheta\max \left\{  |f+ Tg|_{\uptheta} , |(\uptheta -2)g|_{\uptheta} \right\} \\
&\leq  \max \left\{  \uptheta^{2} \max\left\{  |f|_{\uptheta} , |g|_{\uptheta}\right\}  , \uptheta^{3} |(\uptheta -2)|_{\uptheta}|g|_{\uptheta} \right\} .
\end{align*}
Since $|\uptheta-2|_{\uptheta} := | (\uptheta-2)T|_{\uptheta} = |(a-3)\uptheta +1|_{\uptheta}=1$ (where we recall that $a\geq 3$ when $N(\uptheta )=1$), we obtain finally
\[ |f+g|_{\uptheta} \leq \uptheta^{3}  \max\left\{  |f|_{\uptheta} , |g|_{\uptheta}\right\}.\]

If $f,g$ have opposite signs, we are essentially dealing with a difference, so we may assume that $f,g>0$ and the norm in question which we wish to calculate is $|f-g|_{\uptheta}$.  Then we write
\begin{align*}
|f-g|_{\uptheta} & = |f+(\uptheta -2)g- Tg|_{\uptheta} 
\end{align*}
and note that $(\uptheta -2)g, f+(\uptheta -2)g, Tg\in  \mathcal{O}_{K}^{0}$, wherein Part 1.\ of the present Theorem is available: with Inframultiplicativity, we deduce
\begin{align}\label{DiffSignPart3TriIneq}
|f-g|_{\uptheta} & \leq \uptheta\max \left\{  |f+ (\uptheta -2)g|_{\uptheta} , |g|_{\uptheta} \right\}  \nonumber \\
& \leq \uptheta\max \left\{ \uptheta \max\left\{     |f|_{\uptheta},  |(\uptheta -2)g|_{\uptheta}\right\}, |g|_{\uptheta} \right\} \nonumber  \\
& \leq \uptheta\max \left\{ \uptheta \max\left\{     |f|_{\uptheta},  \uptheta^{2}|g|_{\uptheta}\right\}, |g|_{\uptheta} \right\}\nonumber  \\
& \leq \uptheta^{4}\max \left\{ |f|_{\uptheta},  |g|_{\uptheta}\right\} .
\end{align}
Finally, we consider the case where $f+g\in \mathcal{O}^{1}_{K}$: then $|f+g|_{\uptheta} := |Tf +Tg|_{\uptheta}$, where now $Tf\in \mathcal{O}^{1}_{K}$ and $Tg\in \mathcal{O}^{0}_{K}$.  If $Tf, Tg$ have the same
sign, we may utilize what we have just proved to obtain
\begin{align*}
|f+g|_{\uptheta} & \leq \uptheta^{3}  \max\left\{  |Tf|_{\uptheta} , |Tg|_{\uptheta}\right\} \\
& =  \uptheta^{3}  \max\left\{  |Tf|_{\uptheta} , |g|_{\uptheta}\right\}  \\
& \leq  \uptheta^{3}  \max\left\{   |f|_{\uptheta} , |g|_{\uptheta}\right\} 
\end{align*}
where we have used the inequality $|Tf|_{\uptheta}\leq  |f|_{\uptheta}$, in its turn a consequence of Inframultiplicativity.
Similarly, if  $Tf, Tg$ have the opposite sign, we have, using (\ref{DiffSignPart3TriIneq}), 
\[ |f+g|_{\uptheta} := |Tf+Tg|_{\uptheta} \leq \uptheta^{4}  \max\left\{   |Tf|_{\uptheta} , |Tg|_{\uptheta}\right\}  \leq \uptheta^{4} \max\left\{   |f|_{\uptheta} , |g|_{\uptheta}\right\}    .\]
\end{proof}


\section{The Inframetric Completion}\label{ThetaCompletionSection}

Consider the function 
\[ d_{\uptheta}:\mathcal{O}_{K}\times\mathcal{O}_{K}\longrightarrow \R_{+},\quad   d_{\uptheta}(\upalpha, \upbeta ) = |\upalpha - \upbeta|_{\uptheta} ,\]
a $\uptheta^{2}$-inframetric ($\uptheta^{4}$-inframetric) when $N(\uptheta )=-1$ ($N(\uptheta )=1$), as defined in \S \ref{InfraNormSection}. 
The sets  
\[ \mathcal{R}_{\uptheta} := \{d_{\uptheta}\text{-Cauchy sequences}\} \supset \mathcal{M}_{\uptheta}:= \{d_{\uptheta}\text{-Cauchy sequences converging to $0$}\}, \] defined exactly as with a metric, are
trivially seen to be   
abelian groups using the infratriangle inequality.  Hence we may define the {\bf {\em inframetric completion}} of $\mathcal{O}_{K}$ as
\[ \widehat{\mathcal{O}}_{\uptheta} := \mathcal{R}_{\uptheta}/\mathcal{M}_{\uptheta}.\]  
In this section, we show that $\widehat{\mathcal{O}}_{\uptheta}$ is an integral domain; in \S \ref{FieldSection}, we show that it is in fact a field isomorphic to $\R$.  

The discussion is essentially independent
of whether $N(\uptheta )=-1$ or $1$: in what follows, we treat first the case when the norm is $-1$, and indicate later what changes need to be made in order to obtain the same results in the case when the norm is $1$.

 \vspace{3mm}
 
 \noindent \fbox{$\boldsymbol N\boldsymbol(\boldsymbol\uptheta \boldsymbol)\boldsymbol=\boldsymbol-\boldsymbol1$}
 
  \vspace{3mm}

We begin with the following infra version of the reverse triangle inequality:

\begin{prop}[Reverse Infratriangle Inequality]\label{revtri}  Let $\upalpha,\upbeta\in\mathcal{O}_{K}$ and suppose that $\uptheta^{2} |\upbeta|_{\uptheta}<|\upalpha|_{\uptheta}$.  Then
\[  |\upalpha |_{\uptheta}- |\upbeta|_{\uptheta} \leq |\upalpha |_{\uptheta} \leq \uptheta^{2} |\upalpha-\upbeta|_{\uptheta}. \]
\end{prop}

\begin{proof}  Applying the infratriangle inequality to $(\upalpha - \upbeta)+\upbeta$ gives
\[|\upalpha|_{\uptheta}  \leq \uptheta^{2}\max\{  |\upalpha-\upbeta|_{\uptheta},  |\upbeta|_{\uptheta} \} =\uptheta^{2} |\upalpha-\upbeta|_{\uptheta}, \]
the last equality being a consequence of the hypothesis $\uptheta^{2} |\upbeta|_{\uptheta}<|\upalpha|_{\uptheta}$ (for if the maximum were $\uptheta^{2}|\upbeta|_{\uptheta}$,
we would obtain $ |\upalpha|_{\uptheta} <|\upalpha|_{\uptheta}$).
\end{proof}

\begin{lemm}\label{limsupnormlemma} Let $\{ \upalpha_{i}\}$ be Cauchy.  Then $\lim\sup|\upalpha_{i}|_{\uptheta}<\infty$.  
\end{lemm}

\begin{proof} Suppose not: we may assume, after passing to a subsequence, that the sequence of infranorms is strictly monotonically increasing: 
\[ |\upalpha_{1}|_{\uptheta} < |\upalpha_{2}|_{\uptheta} < \cdots .\] 
Then for all $i$ and $k\geq 3$, $\uptheta^{2} |\upalpha_{i}|_{\uptheta} < |\upalpha_{i+k}|_{\uptheta}$.  By the Cauchy condition, for fixed
$\upvarepsilon>0$, there is $N$ such that $|\upalpha_{j}-\upalpha_{i}|_{\uptheta}<\upvarepsilon$ for all $j>i\geq N$.  
By Proposition \ref{revtri}, for fixed $i\geq N$ and all $k\geq 3$,
\[  |\upalpha_{i+k}|_{\uptheta}  \leq \uptheta^{2}   |\upalpha_{i+k}-\upalpha_{i}|_{\uptheta} \leq \uptheta^{2}\upvarepsilon,  \]
contradiction.
\end{proof}

We now consider the behavior of the sequence of infranorms $\{ |f_{n} |_{\uptheta} \}$, for $\{ f_{n}\}$ Cauchy, beginning with a pair of examples which illustrate a general phenomenon:


\begin{exam}\label{signednotlaurent}
Consider the sequence $\{ f_{n}\}$ where
\[  f_{n} =_{\rm gr} \left\{
\begin{array}{ll}
1 +(a-1)\uptheta +a \uptheta^{3} +a \uptheta^{5} +  \cdots +a\uptheta^{2i+1} & \text{if $n=2i+1$, $i\geq 1$} \\
& \\
a\uptheta^{2} + a\uptheta^{4} +  \cdots +a\uptheta^{2i} & \text{if $n=2i$, $i\geq 1$} 
\end{array}
\right.    \]
To see that $\{ f_{n}\}$ is Cauchy, we consider differences according to parities of indices, where in the calculations that follow, $j>i$.
Since $\uptheta-a= \uptheta^{-1}$,  
\begin{align*}    f_{2j+1}- f_{2i}  &= 1 + (a-1)\uptheta+( \uptheta^{3}-a\uptheta^{2}) +(a-1)\uptheta^{3} +( \uptheta^{5}-a\uptheta^{4}) \;+   \\
&\;\quad  \cdots + (\uptheta^{2i+1}-a\uptheta^{2i})+(a-1)\uptheta^{2i+1} +a\uptheta^{2(i+1)+1}\cdots +a\uptheta^{2j+1} \\
& = 1+a\uptheta +a\uptheta^{3} + \cdots +a \uptheta^{2i-1} +  (a-1)\uptheta^{2i+1} +a\uptheta^{2(i+1)+1}\cdots +a\uptheta^{2j+1} \\
& =_{\rm gr}  \uptheta^{2i}+    (a-1)\uptheta^{2i+1} +a\uptheta^{2(i+1)+1}\cdots +a\uptheta^{2j+1}  , 
\end{align*}
where the last equality is the result of cascading instances of the Fibonacci relation $\uptheta^{m}+a\uptheta^{m+1}= \uptheta^{m+2}$.  
On the other hand, we have
\begin{align*}    f_{2j}- f_{2i+1}  &= [a\uptheta^{2}  -(1+ (a-1)\uptheta )] + (a\uptheta^{4}-a\uptheta^{3}) + \cdots + (a\uptheta^{2(i+1)}-a\uptheta^{2i+1}) 
+  a\uptheta^{2(i+2)}+\cdots +a\uptheta^{2j} \\
&= [\uptheta+(a-1)\uptheta^{2}   ] + (a\uptheta^{4}-a\uptheta^{3}) + \cdots + (a\uptheta^{2(i+1)}-a\uptheta^{2i+1}) 
+  a\uptheta^{2(i+2)}+\cdots +a\uptheta^{2j}  \\
&= [a\uptheta^{4}  -(\uptheta^{2}+ (a-1)\uptheta^{3} )] + (a\uptheta^{6}-a\uptheta^{5}) + \cdots + (a\uptheta^{2(i+1)}-a\uptheta^{2i+1}) 
+  a\uptheta^{2(i+2)}+\cdots +a\uptheta^{2j} \\
&=_{\rm gr}  \uptheta^{2i+1} +(a-1) \uptheta^{2(i+1)} +  a \uptheta^{2(i+2)} + \cdots + a \uptheta^{2j}.
\end{align*}
Finally, the differences of equal parity indices satisfy
\[ f_{2j}-f_{2i} =_{\rm gr}  a\uptheta^{2(i+1)} +\cdots +a \uptheta^{2j} , \quad f_{2j+1}-f_{2i+1} =_{\rm gr}  a\uptheta^{2(i+1)+1} +\cdots +a \uptheta^{2j+1}.  \]
From this one sees that $\{ f_{n}\}$ is Cauchy.  On the other hand,
\[  |f_{n}|_{\uptheta} = \left\{
\begin{array}{ll}
1  & \text{if $n=2i+1$, $i\geq 1$} \\
&\\
\uptheta^{-2}  & \text{if $n=2i$, $i\geq 1$} 
\end{array}
\right. 
\]
so the limiting infranorm is not well-defined.   
\end{exam}

\begin{exam}\label{NotLaurent}  
 Define
\[  g_{n} = \left\{
\begin{array}{ll}
1 +(a-1)\uptheta +a \uptheta^{3} +a \uptheta^{5} +  \cdots +a\uptheta^{2i+1} & \text{if $n=2i+1$, $i\geq 1$} \\
-\uptheta & \text{if $n=2i$, $i\geq 1$} 
\end{array}
\right.    \]
Then, for $j>i$,
\begin{align*}    g_{2j+1}- g_{2i}  
& = 1+a\uptheta +a\uptheta^{3} + \cdots +a \uptheta^{2i-1} +  a\uptheta^{2i+1} + \cdots + a\uptheta^{2j+1}   \\
& =_{\rm gr}   \uptheta^{2i}+    a\uptheta^{2i+1}  + \cdots + a\uptheta^{2j+1}   .
\end{align*}
Again we obtain a Cauchy sequence, but this time
\[  |g_{n}|_{\uptheta} = \left\{
\begin{array}{ll}
1  & \text{if $n=2i+1$, $i\geq 1$} \\
\uptheta^{-1}  & \text{if $n=2i$, $i\geq 1$} 
\end{array}
\right. 
\]
\end{exam}

By the two above examples,  $|\cdot |_{\uptheta}$ does not extend to a well-defined function of $\widehat{\mathcal{O}}_{\uptheta}$.  Nonetheless, the following
result shows that the possible limits of $\{ |f_{n} |_{\uptheta} \}$, for $\{ f_{n}\}$ Cauchy, are controlled, differing by at most a factor of $\uptheta^{2}$.



\begin{theo}[Infraoscillation, $N(\uptheta )=-1$]\label{oscillation}  Let $\{ f_{n}\}$ be a non null Cauchy sequence such that
$ \lim_{n\rightarrow\infty} |f_{n} |_{\uptheta} $  does not exist.  Then there exists $m\in\Z$ such that
 \[  |f_{n} |_{\uptheta} \in \{ \uptheta^{m}, \uptheta^{m+1}, \uptheta^{m+2}\} \] 
In particular, the sequence  $\{  |f_{n}|_{\uptheta}\}$ is uniformly bounded, from above {\rm and} from below
away from $0$.
\end{theo}

\begin{proof}  Suppose that $\{ f_{n}\}$ is a Cauchy sequence for which $\lim | f_{n}|$ is not well-defined. 
By Lemma \ref{limsupnormlemma},  
the sequence $\{ |f_{n}|_{\uptheta}\}$ is uniformly bounded from above, so we must have that $\lim | f_{n}|$ produces multivalues that are finite. We claim that, by the Reverse Infratriangle Inequality (Proposition \ref{revtri}), these multivalues must lie
in a set of three consecutive powers:
$\uptheta^{m}, \uptheta^{m+1}$ and $\uptheta^{m+2}$.   For suppose not i.e.\ there exist subsequences $\{ f_{i_{l}}\}$ and $\{ f_{j_{l}}\}$ with constant infranorms which are distinct, such that for all $l$,
\[ c = |f_{i_{l}}|_{\uptheta}  >\uptheta^{2} d,\quad d= |f_{j_{l}}|_{\uptheta} .  \]
Thus, for each $l$, the hypothesis of the Reverse Infratriangle Inequality is satisfied.
 But then, since $\{ f_{n}\}$ is Cauchy,
\[ 0\not=c-d\leq 
 \uptheta^{2} |f_{i_{l}}- f_{j_{l}}|_{\uptheta}  \longrightarrow 0, \]
contradiction.
\end{proof}

\begin{theo}\label{DomainTheorem} $\mathcal{R}_{\uptheta}$ is a commutative ring with $1$ and $\mathcal{M}_{\uptheta}$ is a prime ideal.  Thus 
$ \widehat{\mathcal{O}}_{\uptheta}$
is an integral
 domain.
\end{theo}

\begin{proof} Let $\{ \upalpha_{i}\}, \{ \upbeta_{i}\} \in \mathcal{R}_{\uptheta}$.
Then by the infratriangle inequality, inframultiplicativity and Lemma \ref{limsupnormlemma},
\begin{align*}   |\upalpha_{i}\upbeta_{i}-\upalpha_{j}\upbeta_{j}|_{\uptheta} & = |\upalpha_{i}\upbeta_{i}-\upalpha_{i}\upbeta_{j} +\upalpha_{i}\upbeta_{j}- \upalpha_{j}\upbeta_{j}|_{\uptheta} \\ 
& \leq \uptheta^{4} 
\max\big\{ |\upalpha_{i}|_{\uptheta} |\upbeta_{i}-\upbeta_{j}|_{\uptheta}  ,  |\upbeta_{i}|_{\uptheta} |\upalpha_{i}-\upalpha_{j}|_{\uptheta} \big\}  \\
& \longrightarrow 0,
\end{align*}
hence $\mathcal{R}_{\uptheta}$ is a ring.  Likewise, $\mathcal{M}_{\uptheta}$ is seen to be an ideal (since, by definition, $\{ \upalpha_{i}\} \in \mathcal{M}_{\uptheta}$
implies $|\upalpha_{i}|_{\uptheta}\rightarrow 0$).   Suppose that $\{ f_{i}\}, \{ g_{i}\}\not\in \mathcal{M}_{\uptheta}$,  
but 
$\{ f_{i} g_{i}\}\in\mathcal{M}_{\uptheta}$.  
Then by inframultiplicativity, 
\[  \uptheta^{-2}  |f_{i} |_{\uptheta}  |g_{i} |_{\uptheta}  \leq  |f_{i}g_{i} |_{\uptheta}\longrightarrow  0.\]
By Theorem \ref{oscillation}, it follows that either  $ |f_{i} |_{\uptheta}$ or $ |g_{i} |_{\uptheta}$ $\rightarrow 0$,  contradiction.  Thus $ \mathcal{M}_{\uptheta}$ is prime and 
$ \widehat{\mathcal{O}}_{\uptheta}$ is an integral domain.
\end{proof}


$\widehat{\mathcal{O}}_{\uptheta}$ contains all classes of {\bf {\em greedy Laurent series}} in $\uptheta$: 
expressions of the form $\pm f$ where
\[ f=_{\rm gr} \sum_{i=m}^{\infty} a_{i}\uptheta^{i}  := \left\{ f_{n}=_{\rm gr}  \sum_{i=m}^{n} a_{i}\uptheta^{i}    \right\} ,\]
in which each partial sum $f_{n}$ is non-negative and greedy.  For such elements, 
we note that $|\cdot |_{\uptheta}$ is single valued:
\[ |f|_{\uptheta} = \lim | f_{n} |_{\uptheta} = \uptheta^{-m}.\]
Unlike
the $p$-adic case, the association $f\mapsto \hat{f}$, $\hat{f}=$ the class of $f$, where $f$ ranges over all greedy Laurent series,  is {\it not} one-to-one, although it is onto:


\begin{prop}\label{LaurSerRepProp} Every {\rm class} in $\widehat{\mathcal{O}}_{\uptheta}$ contains a representative which is a greedy Laurent series in $\uptheta$.
\end{prop}

\begin{proof} Let $\{ f_{n}\}$ be a non null Cauchy sequence.  After passing to a subsequence, we may assume that the $f_{n}$ have constant sign, say positive, and by Infraoscillation (Theorem \ref{oscillation}), we may also assume they have constant infranorm, 
say $\uptheta^{m}$.  We may then further assume that for all $n$ and all $k>0$,  $f_{n+k}$ and $f_{n}$ share the same initial sum of $n$ terms:
\[ g_{n}: = c_{m}\uptheta^{m} + \cdots + c_{m+n}\uptheta^{m+n} .\]
So, for example, the first term of all the $f_{n}$
is $g_{1}=c_{m}\uptheta^{m}$, the sum of the first two terms of $f_{2},f_{3},\dots $ is $g_{2}=c_{m}\uptheta^{m} + c_{m+1}\uptheta^{m+1}$, and so on.  
The sequence $\{ g_{n}\}$ evidently defines the sequence of partial sums of a greedy Laurent series which defines the same Cauchy class as $\{ f_{n}\}$ (since $f_{n}-g_{n}$ is
already in greedy form and has infranorm $\rightarrow 0$).
\end{proof}

Unlike the $p$-adics,  $\widehat{\mathcal{O}}_{\uptheta}$ contains classes containing equivalent Laurent series representatives with distinct infranorms, as seen e.g.\ in {\it Example} \ref{signednotlaurent}.

The following example shows that $\widehat{\mathcal{O}}_{\uptheta}$ contains elements of $K\setminus\mathcal{O}_{K}$.   In fact, we will see later on that $K\subset\widehat{\mathcal{O}}_{\uptheta}$, see Theorem \ref{Norm-1Field}  of \S \ref{FieldSection}.

\begin{exam} Consider the case $\uptheta=\upvarphi=$ the golden ratio and the sequence 
\begin{align}\label{Defnfn} f_{n} = 1-\upvarphi^{2}+\upvarphi^{4} -\cdots + (-1)^{n}\upvarphi^{2n}\end{align}
whose greedy expansion takes the shape (proved by induction using the Fibonacci relation),
\begin{align}\label{thetwonorms}   f_{n} =_{\rm gr} \left\{ 
\begin{array}{ll}
 \upvarphi^{2(2k)-1} +  \upvarphi^{2(2k-2)-1} + \cdots   + \upvarphi^{7}+\upvarphi^{3}+1 & \text{if $n=2k$} \\
 & \\
 - \left( \upvarphi^{2(2k) +1} + \upvarphi^{2(2k-2) +1} + \cdots +\upvarphi^{5}+\upvarphi \right) & \text{if $n=2k+1$} 
\end{array}
\right.\end{align}
This sequence is Cauchy: indeed for $n>m$ large, and using the representation (\ref{Defnfn}),
\begin{align*} f_{n}- f_{m} & = (-1)^{m+1}\upvarphi^{2(m+1)}  +\cdots +  (-1)^{n}\upvarphi^{2n}  \\
& = (-1)^{m+1}\upvarphi^{2(m+1)}  \left( 1 -\upvarphi^{2} + \cdots +  (-1)^{n-m-1}\upvarphi^{2(n-m-1)} \right) .
\end{align*}
Thus, by inframultiplicativity,
\begin{align*} |f_{n}- f_{m} |_{\upvarphi} & \leq \upvarphi^{2} \left|  \upvarphi^{2(m+1)}\right|_{\upvarphi} \left|   1 -\upvarphi^{2} + \cdots   (-1)^{n-m-1}\upvarphi^{2(n-m-1)} \right|_{\upvarphi}   \\
& = \upvarphi^{-2m}\left|   1 -\upvarphi^{2} + \cdots   (-1)^{n-m-1}\upvarphi^{2(n-m-1)}\right|_{\upvarphi} .
\end{align*} 
By  (\ref{thetwonorms}),
\[ |   1 -\upvarphi^{2} + \cdots   (-1)^{n-m-1}\upvarphi^{2(n-m-1)}|_{\upvarphi}\leq 1,\] hence
\[ |f_{n}- f_{m} |_{\upvarphi} \leq \upvarphi^{-2m}\longrightarrow 0.\]
This sequence may be viewed as giving a representation of $(1+\upvarphi^{2})^{-1}\not\in \mathcal{O}_{K}$
since 
\[   (1+\upvarphi^{2})f_{n} = 1 + (-1)^{n}\upvarphi^{2(n+1)} , \]
which converges to 1 in the infranorm.  Thus $\widehat{\mathcal{O}}_{\uptheta}$ contains elements of $K\setminus \mathcal{O}_{K}$.
\end{exam}

Although $|\cdot |_{\uptheta}$ does not extend to a well-defined function of $\widehat{\mathcal{O}}_{\uptheta}$, it does extend to a multivalued function with no more than three values
for any given element:

\begin{theo}\label{ExtensionOfInfraNorm} Let $ \{ f_{n}\},  \{ f'_{n}\}\in  \mathcal{R}_{\uptheta}\setminus \mathcal{M}_{\uptheta}$ define the same class in $\widehat{\mathcal{O}}_{\uptheta}$ and consider subsequences of each $\{ f_{n_{i}}\}$ resp.\ 
$\{ f'_{n'_{i}}\}$
for which the sequences of  infranorms  converge to $\uptheta^{m}$ resp. $\uptheta^{m+k}$, $k\geq 0$.  Then $k\leq 2$.  In particular, $|\cdot |_{\uptheta}$ extends to 
a multivalued function 
\begin{align}\label{MultiValExt} |\cdot |_{\uptheta}: \widehat{\mathcal{O}}_{\uptheta} \multimap \uptheta^{\Z}\end{align}
in which, for all $\hat{f}\in \widehat{\mathcal{O}}_{\uptheta}$, there exist $m\in\Z$ such that
\[   |\hat{f}|_{\uptheta}\subset \{ \uptheta^{m},\uptheta^{m+1}, \uptheta^{m+2}\}.   \]
If $\hat{f}=f\in \mathcal{O}_{K}$, then $|\hat{f}|_{\uptheta}=|f|_{\uptheta}$. 
\end{theo}

\begin{proof} Suppose that $k>2$. Then for $i$ large, the hypothesis of the Reverse Infratriangle Inequality (Proposition \ref{revtri}) is satisfied for $\upalpha= f'_{n'_{i}}$ and $\upbeta = f_{n_{i}}$.
But then $|f'_{n'_{i}}|_{\uptheta}=\uptheta^{m+k}\leq \uptheta^{2} |f'_{n'_{i}}- f_{n_{i}}|_{\uptheta}\rightarrow 0$, contradiction.
\end{proof}

We have the following version of the infratriangle inequality for the maximum of the multivalued extension (\ref{MultiValExt}).  

\begin{theo}\label{ExtendedInfraTriangle} For all $\hat{f}, \hat{g}\in  \widehat{\mathcal{O}}_{\uptheta} $, 
\[  \max  |\hat{f} -\hat{g}|_{\uptheta}  \leq \uptheta^{2} \max \left\{   \max  |\hat{f}|_{\uptheta},  \max  |\hat{g}|_{\uptheta}           \right\} .  \]
\end{theo}

\begin{proof} We assume that $\hat{f}-\hat{g}\not=0$, since  otherwise the statement is trivial. 
Let $\{ f_{i}-g_{i}\}$ be a Cauchy sequence in the class  $\hat{f} -\hat{g}$, where $\{ f_{i}\}$, $\{ g_{i}\}$ represent $\hat{f}$, $\hat{g}$, and which realizes $  \max  |\hat{f} -\hat{g}|_{\uptheta}$: 
that is, $|f_{i}-g_{i}|_{\uptheta}$ is constant and equal to $\max  |\hat{f} -\hat{g}|_{\uptheta}$ for all $i$ .  Then, 
\[  \max  |\hat{f} -\hat{g}|_{\uptheta} =|f_{i}-g_{i}|_{\uptheta} \leq
\uptheta^{2}\max\left\{  | f_{i}|_{\uptheta},  | g_{i}|_{\uptheta} \right\}\leq  \uptheta^{2}\max \left\{  \max | \hat{f}|_{\uptheta}, \max | \hat{g}|_{\uptheta} \right\}  .\]
\end{proof}

\begin{nota} Theorem \ref{ExtendedInfraTriangle} shows that $\max |\cdot |_{\uptheta}$ is an infranorm on $ \widehat{\mathcal{O}}_{\uptheta}$, which restricts to the infranorm 
$ |\cdot |_{\uptheta}$ on $\mathcal{O}_{K}$.   In what follows, we will omit mention of the $\max$ when considering the extension to $ \widehat{\mathcal{O}}_{\uptheta}$ of the infranorm, that is, we will write  
\begin{align}\label{CompletedInfraNorm}  |\hat{f}|_{\uptheta} :=\max_{ \{ f_{i}\} \in\hat{f}}  \left( \limsup_{ i}  |f_{i}|_{\uptheta} \right) .  \end{align}
\end{nota}

 \vspace{3mm}
 
 \noindent \fbox{$\boldsymbol N\boldsymbol(\boldsymbol\uptheta \boldsymbol)\boldsymbol=\boldsymbol1$}
 
  \vspace{3mm}

 The analog of the Reverse Infratriangle Inequality (Proposition \ref{revtri}) is obtained in this case by replacing $\uptheta^{2}$ by $\uptheta^{4}$ =  the maximum adjustment constant
in the Theorem \ref{N=1ITE} (the Infratriangle Inequality for $N(\uptheta )=1$).  Then, the arguments needed to prove 
the remaining results of this section in the case $N(\uptheta )=1$ are 
qualitatively identical, the only difference being that now the multivalues of $|\cdot |_{\uptheta}$ lie in a set of five consecutive powers, e.g.,

\begin{theo}[Infraoscillation, $N(\uptheta )=1$]\label{oscillationN=1}  Let $\{ f_{n}\}$ be a non null Cauchy sequence such that
$ \lim_{n\rightarrow\infty} |f_{n} |_{\uptheta} $  does not exist.  Then there exists $m\in\Z$ such that 
the values of $ |f_{n} |_{\uptheta} $ eventually lie in the set $\{ \uptheta^{m},  \uptheta^{m+1},\dots ,  \uptheta^{m+4}\}$.
\end{theo}

In particular, we may conclude 

\begin{theo}\label{DomainTheoremN1} 
$ \widehat{\mathcal{O}}_{\uptheta}$
is an integral
 domain.
\end{theo} 

 In the sequel, we will need the following analog of Proposition \ref{LaurSerRepProp}:
 \begin{prop}\label{LaurSerRepPropN1} Every {\rm class}  $\hat{\upalpha}\in \widehat{\mathcal{O}}_{\uptheta}$  contains
 a representative of the form
 \[  \pm T^{\upepsilon} \sum_{i=m}^{\infty} c_{i} \uptheta^{i} ,\quad \upepsilon\in \{ 0,-1\} ,\]
 where $\sum_{i=m}^{\infty} c_{i} \uptheta^{i}$ is a greedy Laurent series.
\end{prop}

 \begin{proof}  Same proof as Proposition \ref{LaurSerRepProp}, except that we must choose the subsequence of the form $T^{\upepsilon}g^{0}_{n}$,
$g^{0}_{n}\in \mathcal{O}_{K}^{0}$ and ${\upepsilon}\in \{ -1,0\}$ constant, where we are using the representations of elements
of $\mathcal{O}^{1}_{K}$ in (\ref{TinvFormO1}).

\end{proof}

 \vspace{3mm}
 
 \begin{center}
 $\circ\quad\circ\quad\circ$
  \end{center}

 \vspace{3mm}


In what follows, we define the topology on the completion using a uniform structure  \cite{Bourbaki}.
The proofs do not require explicit mention of the adjustment constants in the infratriangle inequalities (with respect to the cases $N(\uptheta )=-1$ and $N(\uptheta )=1$), so in the ensuing account, $\uptheta$ will be
an arbitrary real quadratic unit.

We begin by associating  to $|\cdot |_{\uptheta}$ a canonical uniform structure  on $\widehat{\mathcal{O}}_{\uptheta}$, by introducing the base of entourages 
\[  U_{\upvarepsilon} = \{ (\hat{f}, \hat{g})\; :\;\; |\hat{f}-\hat{g}|_{\uptheta}<\upvarepsilon\},  \]
which we note is countable since the infranorm is countably valued.
Since $U_{\updelta}\subset U_{\upvarepsilon}$ for $\updelta<\upvarepsilon$ and 
\[ U_{\upvarepsilon}^{-1} := \left\{ (\hat{g} , \hat{f}) \; :\;\; (\hat{f}, \hat{g} ) \in U_{\upvarepsilon}\right\}  =U_{\upvarepsilon}\] for all $\upvarepsilon>0$, the $U_{\upvarepsilon}$ generate a Hausdorff uniform structure $\mathfrak{U}$.

For each $U\in \mathfrak{U}$ and $\hat{f}\in  \widehat{\mathcal{O}}_{\uptheta}$, the $U$-ball centered at $\hat{f}$ is 
\begin{align}\label{FilterNabe}  U[\hat{f}] = \left\{ \hat{g}\; :\;\; (\hat{f}, \hat{g})\in U\right\}, \end{align}
and for each $X\subset \widehat{\mathcal{O}}_{\uptheta}$, define the $U$-ball centered at $X$ by
\[  U[X]:= \bigcup_{\hat{g}\in X} U[\hat{g}]  . \]
The collection \[ \bigg\{ U[\hat{f}] \bigg\}_{U\in\mathfrak{U}}\]
defines the neighborhood base filter about $\hat{f}$.  
By Proposition 1, II.2 of \cite{Bourbaki}, there exists a unique topology on $\widehat{\mathcal{O}}_{\uptheta}$  for which the  filter
of neighborhoods about $\hat{f}$ is given by (\ref{FilterNabe}).  We call this topology the {\bf inframetric topology}.  By construction,
\[ \hat{f}_{n}  \rightarrow \hat{f} \text{ w.r.t. } |\cdot |_{\uptheta} \Longleftrightarrow   \hat{f}_{n}  \rightarrow \hat{f} \text{ w.r.t.\ the inframetric topology}.\]

We warn the reader that we use the terminology {\it neighborhood} of $\hat{f}$ in the sense of \cite{Bourbaki}: that is, a set which contains an open set containing
$\hat{f}$.  In the topology referred to in the previous paragraph,  the $\upvarepsilon$ balls $U_{\upvarepsilon}[\hat{f}] $ -- which generate the neighborhood filter centered at $\hat{f}$ -- are {\it not} open, as 
{\it Example} \ref{NotOpen} below shows.  
The attraction of the approach via uniform structures is that it places emphasis on neighborhood filters rather than
open sets, which are difficult to describe explicitly using the infranorm.

\begin{exam}\label{NotOpen} Take $\uptheta=\upvarphi$ the golden ratio, $f=1+\upvarphi^{2}$, $g_{n}= \upvarphi^{n}$ and $\upvarepsilon = \upvarphi$.  Clearly
$g_{n}\rightarrow 0$ as $n\rightarrow \infty$  and  
\[   0 \in U_{\upvarphi}[f] .\]
We claim that $g_{n}\not\in U_{\upvarphi}[f]$ for all $n=2m$ even.  Indeed,
\begin{align*} g_{n} - f  =\upvarphi^{2m}- \upvarphi^{2}-1 & = \upvarphi^{2m-1} +\upvarphi^{2m-2} - \upvarphi^{2}-1   = \upvarphi^{2m-1} +\upvarphi^{2m-3} + \upvarphi^{2m-4}  - \upvarphi^{2}-1 \\
& =   \upvarphi^{2m-1} +\upvarphi^{2m-3}  + \cdots +  \upvarphi^{3} -1 \\ 
& =_{\rm gr}  \upvarphi^{2m-1} +\upvarphi^{2m-3}  + \cdots +  \upvarphi^{5} + \upvarphi^{2} + \upvarphi^{-1}
\end{align*}
so $|g_{n}-f|_{\upvarphi} =\upvarphi\not< \upvarphi$.  Therefore, $U_{\upvarphi}[f] $ is not open.
\end{exam}

\begin{prop}\label{GCWD} Galois conjugation induces a well-defined ring homomorphism 
\[ \prime: \widehat{\mathcal{O}}_{\uptheta}\longrightarrow \R.\]
\end{prop}

\begin{proof}First, suppose $N(\uptheta )=-1$.  For $\hat{x}\in\widehat{\mathcal{O}}_{\uptheta} $,  choose a representative Cauchy sequence $\{ x_{i}\}$.
We claim that $\{ x_{i}'\} \subset\R$ defines a real Cauchy sequence independent of the choice of $\{ x_{i}\}$ in the class of $\hat{x}$.  
Indeed, given $\upvarepsilon = \uptheta^{-r} >0$ choose $N$ such that $|x_{i} -x_{j}|_{\uptheta} \leq \upvarepsilon$ for all $i,j\geq N$.  Then
\[  |x_{i}' - x_{j}' | \leq \uptheta^{-r}a (1+ \uptheta^{-1} + \uptheta^{-2} + \cdots ) = a\uptheta^{-r}\frac{\uptheta}{\uptheta -1} < 
a\uptheta^{-r}\frac{\uptheta}{\uptheta -a}   
  =a \uptheta^{-r+2}\]
giving the Cauchy condition for the conjugates.  A similar argument shows that if $\{ y_{i}\}\in \hat{x}$, then $\{ y_{i}'\}$ defines the same Cauchy class
as $\{ x'_{i}\}$.  We denote the limit of $\{ x_{i}'\}$ by $\hat{x}'$.

Now consider the case $N(\uptheta )=1$ and take $\{ x_{i}\}\in \hat{x}$ as before a representative Cauchy sequence.  Choose $\upvarepsilon = \uptheta^{-r} >0$ and $N$ such that $|x_{i} -x_{j}|_{\uptheta} \leq \upvarepsilon$ for all $i,j\geq N$.  If $x_{i}-x_{j}\in \mathcal{O}^{0}_{K}$ then we have 
\[  |x_{i}'-x_{j}'| \leq \uptheta^{-r}| (a-1) + (a-2)\uptheta^{-1} +\cdots | = \uptheta^{-r+1}  , \]
where the final equality above follows from 
(\ref{InfForBlockGreedy}).
Otherwise, if $x_{i}-x_{j}\in \mathcal{O}^{1}_{K}$ then, using (\ref{TinvFormO1}), we have 
\[    |x_{i}'-x_{j}'| \leq |T'|^{-1} \uptheta^{-r}| (a-1) + (a-2)\uptheta^{-1} +\cdots |  = \frac{\uptheta^{-r+2}}{\uptheta-1}. \]
This shows that the corresponding sequence $\{ x'_{i}\}$ is Cauchy in $\R$ and again we can write $\hat{x}'$ unambiguously.  Thus Galois conjugation
defines a well-defined function $\widehat{\mathcal{O}}_{\uptheta}\rightarrow \R$.  Since the ring structure of $\widehat{\mathcal{O}}_{\uptheta}$
is defined, at the level of Cauchy sequences, component-by-component on entries in $\mathcal{O}_{K}$, it follows that Galois conjugation extends to a ring homomorphism of $\widehat{\mathcal{O}}_{\uptheta}$.
\end{proof}


\begin{lemm}\label{ContConj} There exists a constant $D$, which only depends on $\uptheta$, such that for all $\hat{x}\in \widehat{\mathcal{O}}_{\uptheta}$, 
\begin{align}\label{doubleinequality}  |\hat{x}'| D^{-1}\leq |\hat{x}|_{\uptheta} \leq D|\hat{x}'|.\end{align}
\end{lemm}

\begin{proof} Assume first that $N(\uptheta )=-1$.  Let  $\hat{x}'$ be defined as in the previous paragraph.  Choose a representative sequence $\{ x_{i}\}$ of $\hat{x}$ such that $|x_{i}|_{\uptheta} = |\hat{x}|_{\uptheta}$
for all $i$.  In fact, we may assume that  the $x_{i}$ are the partial sums of a Laurent series
\begin{align}\label{defnconjofx}    \sum_{i=M}^{\infty} c_{i}\uptheta^{i} \end{align}
 in the class of $\hat{x}$ (this follows from the proof of Proposition \ref{LaurSerRepProp}).
Then, 
\[ |\hat{x}'| = \uptheta^{-M} |c_{M}\pm c_{M+1}\uptheta^{-1} \pm \cdots |= |\hat{x}|_{\uptheta} |c_{M}\pm c_{M+1}\uptheta^{-1} \pm \cdots | .\]
Since the $c_{i}\in \big\{ 0,\dots ,a\}$ and satisfy the greedy condition, and the signs alternate according to parity of the index, the expression $|c_{M}\pm c_{M+1}\uptheta^{-1} \pm \cdots | $ is unformly bounded away from 0 and infinity.  
Indeed, by (\ref{InfForBlockGreedy}), 
\[ \uptheta^{-1}= 1-(a-1)\uptheta^{-1} - a\uptheta^{-3}-a\uptheta^{-5}-\cdots \leq   |c_{M}\pm c_{M+1}\uptheta^{-1} \pm \cdots | \leq a(1+\uptheta^{-2} +\uptheta^{-4}\cdots )  < \infty .\]
This proves 
(\ref{doubleinequality}) for  $N(\uptheta )=-1$.  Now suppose  $N(\uptheta )=1$.  Here $\hat{x}$ may be represented by $T^{\upepsilon}\times$ a Laurent series
of the shape (\ref{defnconjofx}), where $\upepsilon \in \{0,-1\}$.  Without loss of generality we may assume $\upepsilon =0$.  Then we have the bounds
\[  |\hat{x}|_{\uptheta}   \leq |\hat{x}'| \leq \uptheta^{-M}\left( (a-1) + (a-2)\uptheta^{-1} + \cdots  \right) = |\hat{x}|_{\uptheta} \uptheta ,\]
which give again (\ref{doubleinequality}).

\end{proof}

\begin{coro}\label{ConjIsCont} The map \begin{align}\label{ContConjMap}
  {}^{ \prime}: \widehat{\mathcal{O}}_{\uptheta}\longrightarrow \R, \quad x\longmapsto x',\end{align}is uniformly continuous,  hence continuous.
\end{coro}

\begin{proof}. Suppose first that $N(\uptheta )=-1$.
We recall  that uniform continuity (c.f.\ \cite{Bourbaki}, \S II.2) means in our setting that for every $\upvarepsilon >0$, there exists $\updelta >0$ such that for all  $(\hat{x},\hat{y})\in U_{\updelta}$, $|\hat{x}'-\hat{y}'|<\upvarepsilon$.  Choose $\updelta =D^{-1}\upvarepsilon$ and let $\{ x_{k}\}\in \hat{x}$, $\{ y_{k}\}\in \hat{y}$ be representative sequences of elements of $\mathcal{O}_{K}$, which, since $N(\uptheta )=-1$,
may be taken to be $\pm$ greedy polynomials in $\uptheta$.  Then $z_{k}=x_{k}-y_{k}$ satisfies that eventually 
\[ |z_{k}|_{\uptheta}\in \bigg\{ |z|_{\uptheta}, \uptheta^{-1}|z|_{\uptheta},  \uptheta^{-2}|z|_{\uptheta}\bigg\}\]  by Infraoscillation (Theorem \ref{oscillation}).
Therefore, by (\ref{doubleinequality}),  $|z'_{k}| <\upvarepsilon$ and we are done.  The proof for $N(\uptheta )=1$ is
essentially the same, where we use Infraoscillation appropriate to this case (Theorem  \ref{oscillationN=1}).
\end{proof}

\begin{prop} The operations of $+$, $\times$ are continuous on $\widehat{\mathcal{O}}_{\uptheta}$. 
\end{prop}

\begin{proof} The product uniform structure on $\widehat{\mathcal{O}}_{\uptheta}\times \widehat{\mathcal{O}}_{\uptheta}$ is generated by
sets of the form $U_{\upvarepsilon_{1}}\times U_{\upvarepsilon_{2}}$.  When $N(\uptheta )=-1$, the infratriangle inequality shows that the image of such a set by the sum is in 
$U_{\updelta}$ for $\updelta = \uptheta^{2}\max (\upvarepsilon_{1}, \upvarepsilon_{2})$, hence the sum is a uniformly continuous map and so is continuous.  The statement
for the product is proved similarly using  inframultiplicativity.  The case $N(\uptheta )=1$ is proved analogously.
\end{proof}

\section{The $\uptheta$-adic Numbers}\label{ThetaAdicHypSection} 

 \vspace{3mm}
 
 \noindent \fbox{$\boldsymbol N\boldsymbol(\boldsymbol\uptheta \boldsymbol)\boldsymbol=\boldsymbol-\boldsymbol1$}
 
  \vspace{3mm}

 Given a  greedy Laurent series with non-negative coefficients 
 \begin{align}\label{LaurentDefN-1} x=_{\rm gr} \sum_{i=m}^{\infty} b_{i}\uptheta^{i},  \quad b_{i}\geq 0, b_{m}\not=0,\end{align}
 we say that a sequence $\{ x_{n}\}$ is a {\bf {\em sequence of partial sums}} for $x$ if 
 \begin{enumerate}
\item[1.] for each $n\in \N$ there exists $M_{n}\in \Z$, $M_{n}>m$, with
 \begin{align}\label{PartSumPolit} x_{n} =_{\rm gr} \sum_{i=m}^{M_{n}} b_{i}\uptheta^{i}  + \updelta_{n}, \quad  \updelta_{n} :=_{\rm gr} \sum_{j= M_{n}+1} ^{N_{n}} c_{j,n}\uptheta^{j} \end{align}
\item[2.]   if $l\leq n$ then $M_{l}\leq M_{n}$ and
\item[3.]  $M_{n}\rightarrow \infty$.
\end{enumerate} 
In the event that $x$ is a finite sum, we allow the coefficient sequence $b_{i}$ to take the values $0$ for $i$ large in the first sum appearing in (\ref{PartSumPolit}).
Note that trivially any infinite subsequence of a sequence of partial sums is also a sequence of partial sums.   We refer to the sequence $\{\updelta_{n}\}$
as the {\bf {\em defect}} of the sequence of partial sums. 


Two sequences of partial sums $\{ x_{n}\}$ and $\{ y_{n}\}$ with respect to the same Laurent series are said to be equivalent.  The point of this equivalence is to allow us to represent a greedy Laurent series by a flexible family of partial sums, and not just by the ``canonical'' sequence of partial sums, e.g.,
\[  x_{n} =   \sum_{i=m}^{m+n} b_{i}\uptheta^{i}. \] 
 Then we define
\[   \mathcal{O}^{+}_{\uptheta} = \left\{   \{ x_{n} \} \text{ a partial sum of a greedy Laurent series with non-negative coefficients} \right\}\big/\text{equivalence} \]
and 
\[ \mathcal{O}_{\uptheta}  =\mathcal{O}^{+}_{\uptheta} \cup\left(-\mathcal{O}^{+}_{\uptheta} \right).   \]
In particular, an element $y\in \mathcal{O}_{\uptheta}$ is a class of sequences of partial sums $\{ y_{n}\}$ with either all non-negative or all non-positive coefficients.  Moreover, if
$\{ y_{n}\}\in y$ has non-negative (non-positive) coefficients, {\it all} sequences of partials sums belonging to the class $y$ have non-negative (non-positive) coefficients.
The elements of $\mathcal{O}_{\uptheta} $ will be denoted using $\pm$ (\ref{LaurentDefN-1}).

\begin{note} As we will see below,
$-1$ will be represented by {\it two} distinct classes in $\mathcal{O}^{+}_{\uptheta}$, while any positive element of $\mathcal{O}_{K}$ has a unique representative.  Thus
 the inclusion of $-\mathcal{O}^{+}_{\uptheta}$ in the definition of $\mathcal{O}_{\uptheta}$ allows us to canonically embed all of $\mathcal{O}_{K}$ in $ \mathcal{O}_{\uptheta} $, and not just the positive elements.   This is 
 in contrast with the situation in the $p$-adics, where $-1$ has the {\it unique} representation $(p-1) (1+p+p^{2}+\cdots )$. \end{note}

\begin{note} The inclusion of $\updelta_{n}\rightarrow 0$ in the definition of the partial sum $x_{n}$ (Item 2.\ above) is in keeping with the usual $p$-adic precedent.
In particular, to conclude that additive inverses exist, the element $0$ must be taken to be the class of all sequences $\{\updelta_{n}\}  $ 
converging to $0$ in $|\cdot |_{\uptheta}$: each of which defines a partial sum in the sense defined above. For example, an additive inverse of $1$ may be taken to be
\[ a(\uptheta + \uptheta^{3} + \uptheta^{5} + \cdots ),\]
since the sequence of partial sums $x_{n}=a(\uptheta + \cdots +\uptheta^{2n+1})$ satisfies  $x_{n}+1=\uptheta^{2n+2}\rightarrow 0$.  \end{note}

To conclude: the definition of $\mathcal{O}_{\uptheta}$ amounts essentially to refining
 Cauchy convergence with respect to $|\cdot |_{\uptheta}$, imposing the stronger requirement that equivalence classes contain a unique greedy Laurent series in $\uptheta$.

We now define the operations of product and sum in  $\mathcal{O}_{\uptheta} $.
Given $x,y\in  \mathcal{O}_{\uptheta} $,  for $w, z\in  \mathcal{O}_{\uptheta} $, we write
\[ w\in x+ y,\quad \text{resp.}\quad      z\in  x\cdot y,\]
if there exist representative sequences of partial sums $\{ x_{n}\}\in x$ and $\{ y_{n}\}\in y$ such that $\{ x_{n} + y_{n}\} \in w$, resp. $\{ x_{n}\cdot y_{n}\} \in z$.  
The notation indicates that the operations
of sum and product are  multivalued, which we show in the following two examples.

\begin{exam}[Multivaluedness of the Sum]\label{MultisumExample} Consider 
\[ f= 1+ (a-1)\uptheta+ \uptheta^{2} + \uptheta^{3}+\cdots ,\quad g= \uptheta+ (a-1)\uptheta^{2} + (a-1)\uptheta^{3} +\cdots ,\]
as well as the following partial sum scheme, obtained by summing the first $n$ summands of each:
\[  f_{n}= 1 +(a-1)\uptheta+ \uptheta^{2} + \cdots +\uptheta^{n-1}, \quad g_{n} = \uptheta+(a-1)\uptheta^{2} + \cdots + (a-1)\uptheta^{n} .\]
Now consider the sequence 
\[  h_{1}= f_{1}+g_{1}=_{\rm gr}1+\uptheta, \quad h_{n}:=f_{n}+g_{n} = 1+ a\uptheta + a\uptheta^{2} +\cdots +a\uptheta^{n-1} + (a-1)\uptheta^{n}, \;\; n\geq 2.\]
Then we have, for example, the first five elements of the sequence $\{h_{n}\}$ in greedy form are:
\begin{align*} h_{1} & =_{\rm gr}1+\uptheta \\  
h_{2} & =
1+a\uptheta+ (a-1)\uptheta^{2}  =_{\rm gr} a\uptheta^{2} \\
h_{3}  &= h_{2} + \uptheta^{2} + (a-1)\uptheta^{3}  = (a+1)\uptheta^{2} + (a-1)\uptheta^{3} =_{\rm gr} 1+(a-1)\uptheta + a\uptheta^{3}\\
h_{4} & =   h_{3} +   \uptheta^{3} + (a-1)\uptheta^{4} =   1+(a-1)\uptheta +( a+1)\uptheta^{3}  +(a-1)\uptheta^{4} =_{\rm gr} a\uptheta^{2}+ a\uptheta^{4}\\
 h_{5} &   = h_{4}+ \uptheta^{4} + (a-1)\uptheta^{5} =    a\uptheta^{2}+ (a+1)\uptheta^{4} + (a-1)\uptheta^{5} =_{\rm gr}   1+  (a-1)\uptheta + a \uptheta^{3}  +a \uptheta^{5} .
\end{align*}
Inductively, we see that the sequence $\{ h_{n}\}$ cannot be the sequence of partial sums of a greedy Laurent series, due to the oscillation of infranorms:
\[  |h_{n}|_{\uptheta} = \left\{ \begin{array}{ll}
1 &  \text{if $n$ is odd}  \\
\uptheta^{-2} & \text{if $n$ is even}
\end{array} \right.\]
On the other hand, the sequence $\{ h_{2n}\}$ defines the partial sums of the greedy Laurent series
\[   \sum^{\infty}_{i= 1} a\uptheta^{2i} , \]
and the sequence  $\{ h_{2n+1}\}_{n\geq 1}$
gives the greedy Laurent series
\[  1+ (a-1)\uptheta+ \sum^{\infty}_{i= 1} a\uptheta^{2i+1} . \]
Thus the cardinality of $f+g$ is $\geq 2$, and so the sum is multivalued.
\end{exam}

\begin{exam}[Multivaluedness of the Product] \label{MVofProdEx} Now consider product
\[ (1+(a-1)\uptheta ) \cdot (1+\uptheta + \uptheta^{2} + \cdots ),\]
represented by the sequence of partial sums 
\[  (1+ (a-1)\uptheta ) \cdot (1+\uptheta + \uptheta^{2} + \cdots +\uptheta^{n}) = h_{n+1}, \]
where $h_{n}$ is as in {\it Example} \ref{MultisumExample}.   Then, as shown in {\it Example} \ref{MultisumExample}, $\{ h_{n}\}$ contains two distinct greedy Laurent series in $\uptheta$,
and so we conclude that the product is also multivalued.
\end{exam}

In view of {\it Example} \ref{MVofProdEx}, $\mathcal{O}_{\uptheta}$ cannot be a {\it hyperring} in the sense of Krasner \cite{CC}, \cite{Viro}.  As it stands, there does not seem to exist, to our knowledge,
a more general notion of ``multivalued ring'' already existing in the literature, that is satisfied by $\mathcal{O}_{\uptheta}$.  For this reason, we propose the following definition, which
includes all the properties of $\mathcal{O}_{\uptheta}$ that are of interest to us, and due to the example of $\mathcal{O}_{\uptheta}$, we believe, is worthy of consideration.

We will make use of the notion of multigroup originally given by Marty \cite{Marty}, \cite{DO}.   A {\bf {\em Marty multigroup}}\footnote{In \cite{DO}, this notion is referred to as a {\it regular} multigroup.} is a set $M$ with an operation 
\[ \cdot  :M \times M \longrightarrow {\sf 2}^{M}  \]
satisfying the following properties:
\begin{enumerate}
\item[1.] There exists an element $e\in M$ which is a {\it two-sided multiunit}, i.e.,\  such that for all $ m \in M$, \[ m\in e\cdot m =  m\cdot e .\]
\item[2.] For all $m\in M$, there exists a {\it multiinverse}, i.e.,  $m'\in M$ satisfying \[e\in m\cdot m'.\]
\item[3.] For all $m,n,r\in M$, we have {\it set-theoretic} associativity:
\[   (m\cdot n)\cdot r = m\cdot (n\cdot r). \]
\end{enumerate}
$M$ is called {\bf {\em abelian}} or  {\bf {\em commutative}} if it satisfies set-theoretic commutativity: $m\cdot n = n\cdot m$.  If $M$ only satisfies properties $1.$ and $3.$, we will call it a {\bf {\em Marty multimonoid}},
commutative if it satisfies set-theoretic commutativity.

The notion of multigroup most commonly found in modern accounts stems from the work of Krasner, see \cite{Krasner1}, \cite{Krasner2}, \cite{CC}, \cite{Marshall},  \cite{Viro}.
In this variant,  the set theoretic axiom 1.\ above is replaced  by 
\begin{quote}
1'.  There exists $e\in M$ such that   $m\cdot e = e\cdot m=m$ .
\end{quote}
and axiom 2.\ above is replaced by 
\begin{quote}
2'.  For all $m\in M$, there exist  a {\it unique} element $m^{-1}\in M$ such that   $e\in m^{-1}\cdot m=m \cdot m^{-1}$ .
\end{quote}
Moreover, there is an additional inversability axiom: 
\begin{quote}
For all $a,b,c\in M$, $c\in a\cdot b$ $\Leftrightarrow$ $c^{-1}\in b^{-1}\cdot a^{-1}$. 
\end{quote}

 In our setting, these more restrictive axioms 
are inadequate.  

For example, if we take $x= \sum_{i=0}^{\infty} a\uptheta^{2i}$, represented by the partial sums $\left\{  \sum_{i=0}^{n} a\uptheta^{2i}\right\}$ and 
we represent $0$ by the sequence $\{ \uptheta^{2n}\}$, then
\[  \sum_{i=0}^{n} a\uptheta^{2i} + \uptheta^{2n} = \uptheta^{-2} + (a-1)\uptheta^{-1} + a\uptheta + a\uptheta^{3} + \cdots + a\uptheta^{2n-1} + \uptheta^{2n+1} ,\]
which defines the class 
\[ y=  \uptheta^{-2} + (a-1)\uptheta^{-1} + \sum_{i=0}^{\infty} a\uptheta^{2i+1} \not=x.  \]
So $y\in x+0$ hence $x+0 \not=x$ and $\mathcal{O}_{\uptheta}$ does not satisfy Krasner's axiom 1', yet $x\in x+0$, as required by the Marty axiom 1.
Moreover, 
we have at least three additive inverses of $1$:  
\begin{align}\label{Two-1}-1,\quad  (-1)_{1}:= a(\uptheta + \uptheta^{3}+\cdots ),\quad   (-1)_{2}:=  \uptheta^{-1} + (a-1) + a(\uptheta^{2} + \uptheta^{4}+\cdots ) .\end{align}
So Krasner's axiom 2' is not satisfied.   The inversablility axiom does not make sense as stated when inverses are not unique.


\begin{defi} A {\bf {\em Marty multiring}} is a set $M$ equipped with two operations 
\[  +: M \times M\longrightarrow {\sf 2}^{M} ,\quad \cdot  : M \times M\longrightarrow {\sf 2}^{M} \]
satisfying the following properties:
\begin{enumerate}
\item[1.] $(M, +)$ is an abelian Marty multigroup with a chosen\footnote{We emphasize that there may be other units in $(M,+)$.} unit $0$.
\item[2.] $(M\setminus \{0\} , \cdot )$ is a Marty multimonoid.
\item[3.] For all $x\in M$, $x\cdot 0 = 0$.
\item[4.] $M$ satisfies the following weak set-theoretic distributivity laws: 
\[  x\cdot (y+z) \subset x\cdot y + x\cdot z \quad \text{and} \quad (x+y)\cdot z \subset x\cdot z + y\cdot z.\]
\end{enumerate}
If the product is commutative, we call $M$ a {\bf {\em commutative Marty multiring}}.  
\end{defi}
If $M$ and $M'$ are Marty multigroups, a function 
$f:M\rightarrow M'$ is called a {\bf {\em homomorphism of Marty multigroups}}  if for all $x,y\in M$, \begin{align}\label{MultiHomDef} f(x\cdot y)\subset f(x) \cdot  f(y)\end{align}(this definition coincides with
that found in [DO], Chapter II, \S 5).
If $M$ and $M'$ are Marty multirings, a function 
$f:M\rightarrow M'$ is called a {\bf {\em homomorphism of Marty multirings}} if it defines a homomorphism of Marty multigroups $f:(M,+)\rightarrow (M',+)$ and if moreover
 $f(x\cdot y)\subset f(x) \cdot  f(y)$.

\begin{note} The weak set-theoretic distributivity axiom was introduced by Marshall  \cite{Marshall}, where a more restrictive notion of multiring is defined, in which, as in Krasner's version, multiplication is required to be uni-valued.
\end{note}

\begin{lemm}[Multioperations are Non Empty]\label{SubSeqLem} For all $\{ x_{n}\}\in x$, $\{ y_{n}\}\in y$ there exists an element $w\in \mathcal{O}_{\uptheta} $ and an infinite subsequence indexed by $I\subset \N$ so that 
\[   \{ x_{n}+y_{n}\}_{n\in I}\in w \in x+y.    \]
In particular, $x+y\not= \emptyset$.  The same statement holds true for the product as well.
\end{lemm}

\begin{proof} The argument is essentially the same as that of Proposition \ref{LaurSerRepProp}.  Assume first that $x,y\in\mathcal{O}^{+}_{\uptheta}$.  The sequence $\{ f_{n} \} =\{ x_{n}+y_{n}\}$ is Cauchy by Theorem \ref{DomainTheorem}.   If $f_{n}\rightarrow 0$
then we may take $w=0$.  So assume now that $f_{n}\not\rightarrow 0$.  After passing to a subsequence,  we may assume that the $\uptheta$-norms are constant, equal  
say to $\uptheta^{-m}$, and moreover, the coefficients of $\uptheta^{m}$ are all equal.  Thus 
\[ f_{n} =_{\rm gr} c_{m} \uptheta^{m}   + c_{m+1}(n) \uptheta^{m+1}+\cdots  ,\]
where the notation for the coefficients indicates that $c_{m}$ is independent of the index $n$, whereas the remaining coefficients will in general depend on $n$.
Inductively, for all $k$  there exists an infinite set of indices $I^{(k)}\subset \N$ so that for all $n\in I^{(k)}$, 
\[ f_{n}=_{\rm gr} c_{m} \uptheta^{m}   +\cdots + c_{m+k} \uptheta^{m+k} +c_{m+k+1}(n) \uptheta^{m+k+1} +\cdots,   \]
i.e., the first $k$ terms are the same for all $f_{n}$ with $n\in I^{(k)}$.
Then if we take 
\[ I=\{ n_{1}< n_{2}<\cdots \}, \quad  n_{k}\in I^{(k)},\] we obtain a sequence of partial sums of a Laurent series $\sum_{i=m}^{\infty}c_{i}\uptheta^{i}$ defining
a class $w$ with $w\in x+y$ and $\{ w_{n} \}  =\{  x_{n} + y_{n}\} \in w$. 
If $-x,-y\in -\mathcal{O}_{\uptheta}^{+}$ then $-w\in (-x)+(-y)$ for $w\in x+y$ as above.  If $x\in \mathcal{O}_{\uptheta}^{+}$ and $-y\in -\mathcal{O}_{\uptheta}^{+}$, then after passing
to subsequences if necessary, we may find $\{ x_{n}\}\in x$ and $\{ -y_{n}\}\in -y$ so that $\{ x_{n}-y_{n}\}$ is either always non-negative or always non-positive.  Then once again, an argument
in the style of Proposition \ref{LaurSerRepProp} gives $w\in x+(-y)$.
The analogous statement for the product is proved in exactly the same way.
 \end{proof}

\begin{theo}[Multiassociativity]\label{muliassoc} The multivalued operations $+$, $\cdot$ are associative in the set-theoretic sense that for all $x,y,z\in \mathcal{O}_{\uptheta} $, 
\[  (x+y)+z = x+(y+z), \quad (x\cdot y)\cdot z = x\cdot (y\cdot z).\] 
Moreover, the product is weakly distributive over the sum:
\[x\cdot (y+z) \subset (x\cdot y) + (x\cdot z).  \]
\end{theo}

\begin{proof}  Let $t\in   (x+y)+z$.  Then, there exists $w\in  \mathcal{O}_{\uptheta} $ with $w\in x+y$ along with representative sequences of partial sums $\{ x_{n}\}\in x$, $\{ y_{n}\}\in y$ and $\{ z_{n}\}\in z$
such that 
\[   \{ w_{n}\} = \{ x_{n}+y_{n}\}  \in w,\quad \{ w_{n} + z_{n}\} =\{ x_{n}+y_{n}+z_{n}\} \in t.     \]
Now, it may not be the case that there exists $r\in y+z$ with  $\{ r_{n}\} =\{ y_{n} +z_{n}\} \in r$, however, by Lemma \ref{SubSeqLem}, we may pass to a subsequence for which such an $r$ exists with 
$ \{ y_{n} +z_{n}\}_{n\in I} \in r$.  But then, by associativity in $\R$, we have 
\[    \{ t_{n}\}_{n\in I} = \{x_{n}+ y_{n} +z_{n}\}_{n\in I}   = \{x_{n}+ r_{n}\}_{n\in I}   \in t   ,\]
and since $\{ r_{n}\} \in r$,  it follows that $t\in x+(y+z)$.  Since the argument is symmetric, this proves associativity.  The associativity of the product is shown in the same way. The argument for distributivity is similar: choose $t\in x\cdot (y+z) $ represented
by the sequence of partial sums $\{ x_{n}(y_{n}+z_{n})\} $, where $\{ x_{n}\}\in x$, $\{ y_{n}\}\in y$, $\{ z_{n}\}\in z$ and the sum $\{y_{n}+z_{n}\}$ defines a sequence of partial sums in $y+z$.
By ordinary distributivity, 
\[   x_{n}(y_{n}+z_{n} ) = x_{n}y_{n} + x_{n}z_{n} ,  \]
and by Lemma \ref{SubSeqLem}, after passing to a subsequence if necessary, we may assume  $x_{n}y_{n} + x_{n}z_{n}$ defines a sequence of partial sums in $xy + xz$.  Hence $t\in xy + xz$.  
\end{proof}

\begin{note} The other inclusion $x\cdot (y+z) \supset (x\cdot y) + (x\cdot z)$ is problematic, as it may happen that when we choose sequences of partial sums  $\{ x_{n}\}, \{ x_{n}'\} \in x$ for each
of the products $x\cdot y$ and $x\cdot z$,
 $x_{n}\not= x_{n}'$ for all $n$, giving an expression of the shape
\[  x_{n}y_{n} + x_{n}'z_{n}, \]
for which the ordinary distributive law in $\R$ no longer applies.  Intuitively, $ (x\cdot y) + (x\cdot z)$, which consists of the sum of {\it two sets}, is "larger" than 
$x\cdot (y+z) $, which consists of the product of a {\it singleton} and a set.
\end{note}



\begin{theo}\label{IsMarty}  $ \mathcal{O}_{\uptheta}$ is a commutative Marty multiring.
\end{theo}

\begin{proof} 
By Theorem \ref{muliassoc}, both multioperations $+$ and $\cdot$ are associative and $\cdot $ distributes weakly over $+$. 
Thus, the only thing left to verify  is the existence of two-sided additive inverses, which is guaranteed by the inclusion of $-\mathcal{O}_{\uptheta}$.
\end{proof}

\begin{note}\label{-1inverses} 
We have already seen above that 1 has (at least) three inverses:
\[-1,\quad  (-1)_{1}:= a(\uptheta + \uptheta^{3}+\cdots ),\quad   (-1)_{2}:=  \uptheta^{-1} + (a-1) + a(\uptheta^{2} + \uptheta^{4}+\cdots ) .\]
For a general element $x\in \mathcal{O}_{\uptheta}$, consider the {\it set}
\[  (-x)_{1} :=  (-1)_{1} \cdot x.\]
By Lemma \ref{SubSeqLem}, if we write 
\[ \left\{ x_{n}= \sum_{i=m}^{n}b_{i}\uptheta^{i}\right\} ,\quad \left\{  (-1)_{1,n} = a\uptheta + \cdots + a\uptheta^{2n+1} \right\} ,\]
then there exists an infinite subset of indices $\{ n_{1} < n_{2} < \cdots \}$ so that the product of the associated subsequences of $\{ x_{n}\}$ and $\{  (-1)_{1,n} \}$
defines an element $x'\in\mathcal{O}_{\uptheta}$:
\[  \{ x'_{j}\} :=  \{ x_{n_{j}} \cdot  (-1)_{1,n_{j}}\}   \in x'  \in (-x)_{1}. \]
We then have
\[  x_{n_{j}}  + x_{j}' = x_{n_{j}} + (a\uptheta + \cdots + a\uptheta^{2n_{j}+1}) x_{n_{j}}  = x_{n_{j}}  \uptheta^{2n_{j}+2} \longrightarrow 0 ,    \]
where the convergence to $0$ is by inframultiplicativity.
Hence $\{ x_{n_{j}}  + x_{j}'\} \in 0$, which shows that $0\in x +x'$ and thus, $x'$ is also an inverse.  This shows that a general element may have several additive inverses.
\end{note}

As remarked above, by including $-\mathcal{O}_{\uptheta}^{+}\subset \mathcal{O}_{\uptheta}$, we may identify canonically $\mathcal{O}_{K}\subset  \mathcal{O}_{\uptheta}$.

\begin{theo}\label{RestrictionOfMultiOK} The restriction of the multioperations of $ \mathcal{O}_{\uptheta}$ to $\mathcal{O}_{K}$ coincides with the usual operations in the latter.  That is, $\mathcal{O}_{K}$ may be canonically identified
with a {\rm subring} of $ \mathcal{O}_{\uptheta}$.
\end{theo}

\begin{proof}  Let $x,y\in \mathcal{O}_{K}$.  Then any sequence of partial sums of $x$ resp.\ $y$ have the form $\{ x+ \upvarepsilon_{n}\}$ resp. $\{ y+ \updelta_{n}\}$
where $\upvarepsilon_{n}, \updelta_{n}\rightarrow 0$ with respect to $|\cdot |_{\uptheta}$.  By the infratriangle inequality,  $\upvarepsilon_{n}+ \updelta_{n}\rightarrow 0$.  Thus, 
if $( x+ \upvarepsilon_{n}) + ( y+ \updelta_{n})$ defines a sequence of partial sums of a Laurent series, this series must be in the class of the {\it element} $x+y\in\mathcal{O}_{K}$.  In particular, the restriction
of the sum to $\mathcal{O}_{K}$ is single valued and coincides with the usual sum in $\mathcal{O}_{K}$.  The product 
\[    ( x+ \upvarepsilon_{n}) \cdot ( y+ \updelta_{n})  = xy + x\updelta_{n}+y\upvarepsilon_{n} + \upvarepsilon_{n}\updelta_{n} \]
is treated similarly, where one uses  inframultiplicativity and the infratriangle inequality to assert that the final three terms $\rightarrow 0$.
\end{proof}


By definition, each class of Laurent series $x\in \mathcal{O}_{\uptheta}$ defines a {\it unique} class of Cauchy sequence.  Thus
we obtain a canonical map
\[  \uppi: \mathcal{O}_{\uptheta} \longrightarrow \widehat{\mathcal{O}}_{\uptheta} ,\quad x \longmapsto \hat{x}. \]
Moreover, by Proposition \ref{GCWD}, we also have a well-defined conjugation map
\[  \prime:  \mathcal{O}_{\uptheta} \longrightarrow \R , \quad x\longmapsto x ' := \hat{x}' .\]

\begin{theo}\label{CanMapEpi}  The maps $\uppi: \mathcal{O}_{\uptheta} \rightarrow \widehat{\mathcal{O}}_{\uptheta}$
and $ \prime:  \mathcal{O}_{\uptheta} \rightarrow \R $ are epimorphisms of Marty multirings.  \end{theo}

\begin{proof} Note that since the range of either map is a ring, the defining condition (\ref{MultiHomDef}) of a multihomomorphism 
is replaced by $f(x\cdot y)=f(x)\cdot f(y)$.
The statement of the Theorem now  follows immediately from the definition of the multioperations.  Indeed, for example, in the case of the sum, the definition starts with
the sum of partial sums $\{ x_{n} + y_{n}\}$ and then extracts subsequences which form coherent Laurent series.  But all of these series belong to the same Cauchy class.  Thus $\uppi (x+y) = \uppi (x) +  \uppi (y)$. 
 The conjugation map
on $\mathcal{O}_{\uptheta}$ is simply the composition of $\uppi$ with conjugation on $\widehat{\mathcal{O}}_{\uptheta}$, hence it is also an epimorphism by Proposition \ref{GCWD}.
\end{proof}

\begin{note}\label{CardNote} Since $\widehat{\mathcal{O}}_{\uptheta}$ is a ring in the usual sense, it follows from Theorem \ref{CanMapEpi} that the elements of the sets
$\upalpha +\upbeta$, $\upalpha\cdot\upbeta$, $\upalpha,\upbeta\in\mathcal{O}_{\uptheta}$ map onto the {\it singletons} $\hat{\upalpha}+\hat{\upbeta}, \hat{\upalpha}\cdot \hat{\upbeta}\in \widehat{\mathcal{O}}_{\uptheta}$: that is, 
\[   \upalpha +\upbeta \subset \uppi^{-1}(\hat{\upalpha}+\hat{\upbeta}),\quad   \upalpha\cdot\upbeta\subset \uppi^{-1}( \hat{\upalpha}\cdot \hat{\upbeta} ). \]
In particular, the cardinality of the multioperations of $\mathcal{O}_{\uptheta}$ is bounded by the cardinalities of the fibers of the map $\uppi$: which is the subject of the next section.
\end{note}

 \vspace{3mm}
 
 \noindent \fbox{$\boldsymbol N\boldsymbol(\boldsymbol\uptheta \boldsymbol)\boldsymbol=\boldsymbol1$}
 
  \vspace{3mm}

In what follows, using notation established in (\ref{TinvFormO1}),  we will work with expressions of the form $\pm x$ where
 \begin{align}\label{LaurentDef} x =_{\rm gr} T^{\upepsilon} x^{0} =_{\rm gr} T^{\upepsilon}\sum_{i=m}^{\infty} b_{i}\uptheta^{i},\quad \upepsilon \in \{ 0,-1\},\end{align}
 where \[ x^{0} =_{\rm gr} \sum_{i=m}^{\infty} b_{i}\uptheta^{i}\] is a greedy Laurent series with non-negative coefficients and $T =\uptheta -1$.

We will say that a sequence $\{ x_{n}=T^{\upepsilon} x^{0}_{n}\}$ is a {\bf {\em sequence of partial sums}} for $x$ if the factor $T^{\upepsilon} $ is constant for all $n$ and if
 \begin{enumerate}
\item[1.] for each $n\in \N$ there exists $M_{n}\in \Z$, $M_{n}>m$, with
 \begin{align}\label{PrinPartPolit} x^{0}_{n}=_{\rm gr} \sum_{i=m}^{M_{n}} b_{i}\uptheta^{i}  + \updelta^{0}_{n}, \quad  \updelta^{0}_{n} :=_{\rm gr} \sum_{j= M_{n}+1} ^{N_{n}} c_{j,n}\uptheta^{j}, \end{align}
\item[2.]   if $l\leq n$ then $M_{l}\leq M_{n}$ and
\item[3.]  $M_{n}\rightarrow \infty$.
\end{enumerate} 
In the event that $x^{0}$ is a finite sum, we allow the coefficient sequence $b_{i}$ to take the values $0$ for $i$ large in the first sum appearing in (\ref{PrinPartPolit}).
Note that  the defect sequence $\{ T^{\upepsilon} \updelta^{0}_{n}\}$ belongs
to $\mathcal{O}^{-\upepsilon}_{K}$  if and only if the principal part of the partial sum $T^{\upepsilon}\sum_{i=m}^{M_{n}} b_{i}\uptheta^{i}$ 
belongs
to $\mathcal{O}^{-\upepsilon}_{K}$.
As in the case $N(\uptheta )=-1$, we will write
 \[ \mathcal{O}_{\uptheta}  =\mathcal{O}^{+}_{\uptheta} \cup\left(-\mathcal{O}^{+}_{\uptheta} \right),   \]
 where $\mathcal{O}^{+}_{\uptheta}$ consists of Laurent series of the shape (\ref{LaurentDef}) mod equivalence.


We define the operations sum and product in $\mathcal{O}_{\uptheta} $.
Given $x,y\in  \mathcal{O}_{\uptheta} $,  for $w, z\in  \mathcal{O}_{\uptheta} $, we write
\[ w\in x+ y,\quad \text{resp.}\quad      z\in  x\cdot y\]
if there exist representative sequences of partial sums $\{ x_{n}\} \in x$ and $\{ y_{n}\}\in y$ such that $\{ x_{n} + y_{n}\} \in w$, resp. $\{ x_{n}\cdot y_{n}\} \in z$.   
 Notice that for the present case of $N(\uptheta )=1$,  this definition requires, for example, that the sequence of sums  $w_{n}=x_{n}+y_{n}$ produces a principal part and defect sequence of the same type, i.e., both in $\mathcal{O}_{K}^{-\upepsilon}$ for  $\upepsilon$ fixed, even if $x$ and $y$ do not share the same value of $\upepsilon$. 

\begin{exam}[Multivaluedness of the Sum]  Consider 
\[ f= (a-2) \sum_{i=0}^{\infty} \uptheta^{i}, \quad g=1. \]
If we take the obvious sequence of partial sums 
\[ f_{n} = (a-2) \sum_{i=0}^{n-1} \uptheta^{i} , \quad g_{n}=1,\] we obtain  
\[  f_{n} +g_{n}=   (a-1) + (a-2)  \sum_{i=1}^{n-1} \uptheta^{i} \longrightarrow (a-1) + (a-2)  \sum_{i=1}^{\infty} \uptheta^{i}  \in f+g.\]
On the other hand, taking 
\[ f_{n} =_{\rm gr} (a-2) \sum_{i=0}^{n-1} \uptheta^{i}  + (a-3)\uptheta^{n} + (a-1)\uptheta^{n+1},\quad  g_{n} =_{\rm gr} 1+ \uptheta^{n} ,\]
we obtain 
\[  f_{n}+g_{n} = (a-1) + (a-2)\sum_{i=1}^{n} \uptheta^{i}   + (a-1) \uptheta^{n+1} =_{\rm gr} \uptheta^{-1} + \uptheta^{n+2}\longrightarrow \uptheta^{-1},\]
where the second equality comes from the resolution of a prohibited block, see Proposition \ref{ProhibitedBlockRes}.
Thus, $\uptheta^{-1}\in f+g$, so the sum is multivalued; we note that the $\uptheta$-norm also distinguishes these two sums, as the first has norm $1$ and the second norm $\uptheta$.
\end{exam}

\begin{exam}[Multivaluedness of the Product] The following example shows that, as expected, the product is multivalued.  Take $f=1$ and 
\[ g= 1 + (a-1)\uptheta + (a-2)\sum_{i=2}^{\infty}\uptheta^{i}.\] 
If $f_{n}=1$,
then for any sequence of partial sums $\{ g_{n}\}$ of $g$,  $f_{n}g_{n}=g_{n}\rightarrow g$.  On the other hand, if we take 
\[   f_{n} = 1+ \uptheta^{n},\quad g_{n} = 1 + (a-1)\uptheta + (a-2)\sum_{i=2}^{n}\uptheta^{i},\]
then  by Proposition  \ref{ProhibitedBlockRes},
\begin{align*}
f_{n}g_{n} &  =  1+ (a-1)\uptheta + (a-2)\uptheta^{2} + \cdots (a-2)\uptheta^{n-1} + (a-1)\uptheta^{n} + (a-1)\uptheta^{n+1} + (a-2)\sum_{i=2}^{n}\uptheta^{i+n} \\
& = 2 + a\uptheta^{n+1} + (a-2)\sum_{i=2}^{n}\uptheta^{i+n} \\
& =_{\rm gr} 2+ \uptheta^{n} + (a-1)\uptheta^{n+2} +  (a-2)\sum_{i=3}^{n}\uptheta^{i+n} \longrightarrow 2.
\end{align*}
Thus $2\in 1\cdot g$.
Apart from showing that the product is multivalued, this example shows that the multiplicative identity only satisfies the weaker relation $g\in 1\cdot g$.  If one considers
$g'=g-1$, then it is not difficult to see that $g', 1\in g'\cdot 1$ and then $1= |1|_{\uptheta}\not= |g'|_{\uptheta} =\uptheta^{-1}$, so that the $\uptheta$-norm is also multivalued
on $g'\cdot 1$.
\end{exam}

\begin{lemm}[Multioperations are Non Empty]\label{SubSeqLemN=1} For all $\{ x_{n}\}\in x$, $\{ y_{n}\}\in y$ there exists an element $w\in \mathcal{O}_{\uptheta} $ and a subsequence indexed by $I\subset \N$ so that 
\[   \{ x_{n}+y_{n}\}_{n\in I}\in w \in x+y.    \]
In particular, $x+y\not= \emptyset$.  The same statement holds true for the product as well.
\end{lemm}

\begin{proof}  In $\widehat{\mathcal{O}}_{\uptheta}$ the sum $\hat{x}+ \hat{y}$ is well-defined and by Proposition  \ref{LaurSerRepPropN1},
it contains a sequence of partial sums of a Laurent series. 
\end{proof}

\begin{theo}[Multiassociativity]\label{muliassocN1} The multivalued operations $+$, $\cdot$ are associative in the set-theoretic sense that for all $x,y,z\in \mathcal{O}_{\uptheta} $, 
\[  (x+y)+z = x+(y+z), \quad (x\cdot y)\cdot z = x\cdot (y\cdot z).\] 
Moreover, the product is weakly distributive over the sum:
\[x\cdot (y+z) \subset (x\cdot y) + (x\cdot z).  \]
\end{theo}

\begin{proof}  Same as the proof of Theorem \ref{muliassoc}, using Lemma \ref{SubSeqLemN=1} above.
\end{proof}

\begin{theo}\label{IsMartyN1} $ \mathcal{O}_{\uptheta}$ is a commutative Marty multiring.
\end{theo}

\begin{proof} The proof follows the same structure as its counterpart Theorem \ref{IsMarty}. 
\end{proof}

Denote by
\[ \iota : \mathcal{O}_{K}\hookrightarrow \mathcal{O}_{\uptheta}\]
the canonical inclusion.  Unlike the case $N(\uptheta )=-1$, it is no longer true that the restrictions of the operations $+$, $\cdot$ to
$ \iota (\mathcal{O}_{K})$ agree with those of $\mathcal{O}_{K}$, as the following example illustrates.  Nevertheless, we will prove further below that 
$\iota$ is a monomorphism in the category of Marty multirings.

\begin{exam} Let $f=1+ \uptheta$ and $g=-1$.  Consider the following sequences of partial sums for each of the classes $\iota (f)$, $\iota (g)$ :
\[  f_{n} = f+ \upvarepsilon_{n} = 1 + \uptheta +  \uptheta^{n}, \quad  g_{n} = g+\updelta_{n} =-1 - \uptheta^{n+1}.  \]
Then
\begin{align*}
f_{n}+g_{n} & =  \uptheta +  \uptheta^{n}-   \uptheta^{n+1}.
\end{align*}
We claim that in $ \mathcal{O}_{\uptheta}$,
\[  x= -T^{-1} (a-2)\sum^{\infty}_{i=2} \uptheta^{i} \;\;\in  \;\; \iota (f) +  \iota (g).  \]
In fact, taking the sequence of partial sums (for $n\geq 4$)
\[ x_{n} = -T^{-1} \left[ (a-2) \big(\uptheta^{2} + \cdots + \uptheta^{n-1}  \big) +(a-1)\uptheta^{n} +(a-3)\uptheta^{n+1} \right] , \]
and using 
\[ T^{-1} = \frac{\uptheta -1}{(a-2)\uptheta}=\frac{1}{a-2} \left(1-\frac{1}{\uptheta}\right),  \] 
we have by routine calculation
\begin{align}\label{noiseformula}
x_{n} &= \uptheta  - \frac{1}{a-2} \left( -\uptheta^{n-1} + 2\uptheta^{n}+  (a-3) \uptheta^{n+1} \right) \nonumber \\
   &= \uptheta  - \frac{1}{a-2} \left( -\uptheta^{n-1} + 2\uptheta^{n} -\uptheta^{n+1} \right) - \uptheta^{n+1} \nonumber   \\
   &  =f_{n}+g_{n}  ,
\end{align}
where (\ref{noiseformula}) follows from inserting the identity $\uptheta^{n+1} = a\uptheta^{n}-\uptheta^{n-1}$ in the parenthesis of the penultimate line.
\end{exam}

On the other hand, depending on the classes of elements that we consider in the sum, we can  have that the cardinality
of $\iota (f) + \iota (g)$ is 1.  For example, in the case when both $x,y \in \mathcal{O}_{K}$ have the same type (same sign, same factor $T^{\upepsilon}$), the proof that the cardinality of the sum is 1
is identical to that appearing in Theorem \ref{RestrictionOfMultiOK}.

In what follows, a {\it monomorphism} of Marty multirings will be understood as an injective homomorphism of Marty multirings.

\begin{theo} The map  $ \iota : \mathcal{O}_{K}\hookrightarrow \mathcal{O}_{\uptheta}$ is a monomorphism of Marty multirings.
\end{theo}

\begin{proof}  Clearly $\iota (f+g) \in \iota (f) + \iota (g)$ and  $\iota (f\cdot g) \in \iota (f) \cdot \iota (g)$, so $\iota$ is an homomorphism of Marty multirings.  What remains to show is injectivity.
Suppose on the contrary that there exist $f,g\in \mathcal{O}_{K}$ with $\iota (f) = \iota (g)$.  Then there exist partial sums of each which are equal:
\[
f+ \updelta_{n} =g + \upvarepsilon_{n}. 
\]
Taking conjugates we have
\[ f'-  g'=\upvarepsilon_{n}' -\updelta_{n} '. \]
The right-hand side tends to 0 in the euclidean norm but the left hand side is constant, thus the right-hand side is $0$ and $f'=  g'$
which implies $f=  g$.
\end{proof}

As in the case $N(\uptheta )=1$, we have well-defined maps
\[  \uppi: \mathcal{O}_{\uptheta} \longrightarrow \widehat{\mathcal{O}}_{\uptheta} ,\quad x \longmapsto \hat{x}\quad \text{and}\quad
 \prime:  \mathcal{O}_{\uptheta} \longrightarrow \R , \quad x\longmapsto x ' := \hat{x}' .\]

\begin{theo}\label{CanMapEpi1}  The maps $\uppi: \mathcal{O}_{\uptheta} \rightarrow \widehat{\mathcal{O}}_{\uptheta}$
and $ \prime:  \mathcal{O}_{\uptheta} \rightarrow \R $ are epimorphisms of Marty multirings.  \end{theo}

\begin{proof}
  The proof is identical to that of Theorem \ref{CanMapEpi}.  \end{proof}

\section{Cardinality of the Multi-Operations}\label{CardSection}

 \vspace{3mm}
 
 \noindent \fbox{$\boldsymbol N\boldsymbol(\boldsymbol\uptheta \boldsymbol)\boldsymbol=\boldsymbol-\boldsymbol1$}
 
  \vspace{3mm}

Let $\uppi :\mathcal{O}_{\uptheta} \rightarrow \widehat{\mathcal{O}}_{\uptheta}$ be the epimorphism studied in Theorem \ref{CanMapEpi}.
In what follows, we calculate the cardinality of $\uppi^{-1} (\hat{\upalpha})$, $\hat{\upalpha}\in\widehat{\mathcal{O}}_{\uptheta}$.   By {\it Note} \ref{CardNote}, this will bound the cardinality
of the multioperations.

\begin{theo}\label{posunique} For $0<\upalpha \in \mathcal{O}_{K}$, its (polynomial) greedy expansion is the only {\rm positive} greedy Laurent series in its Cauchy class.  In other words, $\uppi^{-1}(\upalpha )$
has a unique pre-image in $\mathcal{O}_{\uptheta}^{+}$.
\end{theo}

\begin{proof} Multiplication by a power $\uptheta^{k}$ takes Cauchy classes to Cauchy classes and Laurent series to (shifted) Laurent series, hence it will be enough to consider $\upalpha$ with $|\upalpha|_{\uptheta}=1$ i.e. $\upalpha = c_{n}\uptheta^{n} + \cdots  +c_{0}$.   We may also assume that $n$ is even: if not, we add the term $0\cdot \uptheta^{n+1}$.
Suppose there exists a series 
$\upgamma=\sum_{i=0}^{\infty} b_{i}\uptheta^{i}$ for which $\upgamma - \upalpha$ defines the $0$ Cauchy class.  Then by Proposition \ref{GCWD}, we must have 
\[   \upgamma'-\upalpha' =0.   \]
Notice that $\upgamma$ is necessarily an infinite series, since Galois conjugation is injective on $\mathcal{O}_{K}$.
 Then
this implies 
\[   \sum_{i=0}^{n/2} c_{2i} \uptheta^{-2i} + \sum_{i=0}^{\infty} b_{2i-1}\uptheta^{-2i+1} =   \sum_{i=0}^{n/2} c_{2i-1}\uptheta^{-2i+1} + \sum_{i=0}^{\infty} b_{2i} \uptheta^{-2i}  . \]
After amalgamating terms of the same power -- grouping on either the left or right-hand side in such a way that the common coefficient is positive -- we obtain 
\[   \sum_{i=0}^{n/2} \tilde{c}_{2i} \uptheta^{-2i} + \sum_{i=0}^{\infty} \tilde{b}_{2i-1}\uptheta^{-2i+1} =   \sum_{i=0}^{n/2} \tilde{c}_{2i-1}\uptheta^{-2i+1} + \sum_{i=0}^{\infty} \tilde{b}_{2i} \uptheta^{-2i} , \]
in which all coefficients appearing are in the range $\{ 0,\dots ,a\}$ and satisfy
\[  \tilde{b}_{i} = 0 \text{ or } \leq b_{i} \text{ and is eventually }=b_{i} , \]
as well as 
\[  \tilde{c}_{i} = 0 \text{ or } \leq c_{i} . \]
Each of the series 
\[ \sum_{i=n/2+2}^{\infty} \tilde{b}_{2i-1}\uptheta^{-2i+1},  \quad   \sum_{i=n/2+1}^{\infty} \tilde{b}_{2i} \uptheta^{-2i} \]
is clearly greedy: since their coefficients $\leq a$ and since they skip powers, there can be no Fibonacci relations.
Let us analyze the remaining terms.  On the left hand side we have 
\begin{align}\label{LHS Poly} \tilde{c}_{0} + \tilde{b}_{1}\uptheta^{-1} + \tilde{c}_{2}\uptheta^{-2} + \tilde{b}_{3}\uptheta^{-3} +\cdots + \tilde{c}_{n}
+\tilde{b}_{n+1}
\uptheta^{-n} =: d_{0} + \cdots + d_{n+1}\uptheta^{-n-1} \end{align}
In what follows, we analyze an algorithm which puts (\ref{LHS Poly}) in greedy form.
Start with the first pair sum $d_{0} + d_{1}\uptheta^{-1}$.  Then either 
\begin{enumerate}
\item[A.]  {\it There is no Fibonacci relation}. 
Then either  
\begin{enumerate} 
\item[Ai.] $d_{0}<a$.  Here, we proceed to the next pair sum $d_{1}\uptheta^{-1} +  d_{2}\uptheta^{-2} $
and decide whether it gives or does not give a Fibonacci relation. 
\item[Aii.]  $d_{0}=a$ but $d_{1}=0$.  Then,  we pass to pair sum $d_{2}\uptheta^{-2} +  d_{3}\uptheta^{-3} $
and again determine whether there is a Fibonacci relation.
\end{enumerate}
\item[B.] {\it There is a Fibonacci relation}.  Thus, we suppose $ d_{0}=a$ and  $d_{1}\not=0$.   This gives
\[ \uptheta + (d_{1}-1)\uptheta^{-1} +  d_{2}\uptheta^{-2} + \cdots  .\]
Note then that the following pair $ (d_{1}-1)\uptheta^{-1} +  d_{2}\uptheta^{-2}$ cannot provide a Fibonacci relation since $d_{1}-1<a$.  
Thus for this pair,  one proceeds to item Ai.\ above, which instructs us to move on to $d_{2}\uptheta^{-2} + d_{3}\uptheta^{-3}$. 
\end{enumerate}

\begin{clai} The algorithm, iterated, produces a greedy polynomial whose lowest order term has power $\geq -n-1$.
\end{clai}

\vspace{3mm}

\noindent {\it Proof of Claim.} The intuitive idea is that Fibonacci relations imply changes upwards to higher powers, never lower powers.  If A.\ always holds at each step, then the expression (\ref{LHS Poly}) is already in greedy form.
Now suppose there are Fibonacci relations, the first one occuring at the index $i$,  i.e., we have 
\begin{align}\label{1stFibRes} \cdots + d_{i-1}\uptheta^{-i+1} +a\uptheta^{-i}  +d_{i+1}\uptheta^{-i-1} \cdots  =\cdots + (d_{i-1}+1)\uptheta^{-i+1}  +(d_{i+1}-1)\uptheta^{-i-1} \cdots   \end{align}
where we assume $d_{i+1}\not=0$.  Since this was the first such Fibonacci relation, $d_{i-1}<a$, hence the new coefficient at
$i-1$ satisfies $(d_{i-1}+1)\leq a$.  After this Fibonacci resolution, the only way that the sum of elements with indices $\geq i-1$ is not greedy would be if $d_{i-1}=0$ and
$d_{i-2}=a$: producing a new Fibonacci relation at the index $i-2$.  Resolving the latter would either produce a sum for indices $\geq i-1$ that is greedy, or else, $d_{i-3}=0$ and $d_{i-4}=a$: producing another Fibonacci relation at the index $i-4$. Thus, inductively,
while the indices $\geq i-1$ may change due to the occurrence of new Fibonacci relations provoked by  (\ref{1stFibRes}), the remaining sum in (\ref{1stFibRes}) for indices $>i-1$,
\[ (d_{i+1}-1)\uptheta^{-i-1} \cdots,\] is unchanged.   Inductively, this stability persists to the end, and the lowest order exponent satisfies $\geq -n$. $\quad\diamond$

\vspace{3mm}

By the {\it Claim},  (\ref{1stFibRes}) has greedy form  
\[ \cdots + \tilde{d}_{n} \uptheta^{-n}+ \tilde{d}_{n+1}\uptheta^{-n-1},\quad \tilde{d}_{n+1}\in \{0,\dots, a\} .\] When this is summed with  $\sum_{i=n/2+2}^{\infty} \tilde{b}_{2i-1}\uptheta^{-2i+1}$, the result is automatically greedy (since there is no $\uptheta^{-n-2}$ term in either sum) and
we obtain an expansion whose powers are eventually all odd.  
A similar analysis of the right-hand side
gives a greedy expansion in which the powers are eventually all even.  This contradicts the unicity of greedy expansions of elements of $\R$.
\end{proof}

We now consider  necessary conditions in order that two greedy Laurent series define the same Cauchy class:

\begin{prop}\label{ShapeOfMultPre}  Let $\upalpha_{1}\not= \upalpha_{2}\in\mathcal{O}_{\uptheta}^{+}$ such that $\uppi (\upalpha_{1})= \uppi (\upalpha_{2})$.  Then  
 \begin{align*} \upalpha_{1}=_{\rm gr} \upeta_{1} + a\left(\uptheta^{i+1}+ \uptheta^{i+3}  +\cdots \right) \end{align*}
  and 
  \begin{align*}  \upalpha_{2}=_{\rm gr} \upeta_{2} + a\left(\uptheta^{j+2}+ \uptheta^{j+4}  +\cdots \right),\end{align*}
where $\upeta_{1},\upeta_{2}$ are greedy Laurent polynomials in $\uptheta$ and $i\equiv j\mod 2$.  In particular, 
\[ \uppi (\upalpha_{1})=\uppi(\upalpha_{2})\in\mathcal{O}_{K}.\]
\end{prop}

\begin{proof} We will use the fact that $\uppi (\upalpha_{1})= \uppi (\upalpha_{2})$ implies  $\upalpha_{1}'= \upalpha_{2}'$: thus
\[ \upalpha_{1} = \sum_{i=m}^{\infty} b_{i}\uptheta^{i},\quad  \upalpha_{2} = \sum_{i=n}^{\infty} c_{i}\uptheta^{i},\]
 produce the same conjugates
 \[  \sum_{i=m}^{\infty} b_{i}(- \uptheta)^{-i} = \sum_{i=n}^{\infty} c_{i}(-\uptheta)^{-i}.   \]
First suppose that $m=n=0$ and that the first coefficient in which they differ is $i=0$.  Suppose $b_{0}>c_{0}$.  Then we have 
\[  (b_{0}-c_{0}) + c_{1}\uptheta^{-1} + b_{2}\uptheta^{-2} + c_{3}\uptheta^{-3} + b_{4} \uptheta^{-4} +\cdots  = b_{1}\uptheta^{-1} + c_{2}\uptheta^{-2} + b_{3}\uptheta^{-3} + c_{4} \uptheta^{-4} +\cdots .
  \]
  By the greedy condition $b_{1}\leq a-1$.  Then, the right-hand side of the above satisfies 
  \begin{align*}     \text{\rm RHS}   &\leq  (a-1)\uptheta^{-1}  + a(\uptheta^{-2} + \uptheta^{-3} +\cdots )
   = -\uptheta^{-1} + \frac{a}{\uptheta} \frac{1}{1-\uptheta^{-1}}   
   = -\uptheta^{-1} + \frac{a}{\uptheta-1}  \\ & =\frac{ -\uptheta +1 +a\uptheta}{\uptheta^{2}-\uptheta}  
    = 
  \frac{ -\uptheta +\uptheta^{2}}{\uptheta^{2}-\uptheta}
  =1.
  \end{align*}
  From this we conclude that $b_{0}-c_{0}=1$ and for this to happen we must have additionally $b_{1} =a-1$, $c_{2}=b_{3}=c_{4}=b_{5}=\cdots =a$ and $c_{1}=b_{2}=c_{3}=b_{4}=\cdots =0$.  We conclude in this case that $\upalpha_{1}, \upalpha_{2}$ are of the form claimed in the statement
  of the Proposition: 
  \[ \upalpha_{1} =(c_{0}+1) + (a-1)\uptheta + a (\uptheta^{3} + \uptheta^{5}+\cdots ),\quad \upalpha_{2} =c_{0} +  a (\uptheta^{2} + \uptheta^{4}+\cdots ) \]
  Notice that this argument works even when $c_{0}=0$, so we can always just assume that our series start from $i=0$, possibly allowing zero coefficients in the beginning.  
  More generally, allowing that $\upalpha_{1}$ and $\upalpha_{2}$ coincide in an initial Laurent polynomial $\upbeta$, the above 
argument shows that if $\upalpha_{1}'=\upalpha_{2}'$, then we have, say,
  \begin{align*} \upalpha_{1}=_{\rm gr} \upgamma + (c_{i}+1)\uptheta^{i} + (a-1)\uptheta^{i+1} + a\left(\uptheta^{i+3}+ \uptheta^{i+5}  +\cdots \right) \end{align*}
  and 
  \begin{align*} \upalpha_{2}=_{\rm gr}  \upgamma + c_{i}\uptheta^{i}  + a\left(\uptheta^{i+2}+ \uptheta^{i+4}  +\cdots \right),\end{align*}
where  \[\upgamma = \sum g_{k}\uptheta^{k}\] is of degree $\leq i-1$.  Notice the dichotomy in the ``tails'' of $\upalpha_{1}$ and $\upalpha_{2}$, which are of opposite power parity.
  Now suppose we have a third series 
  \[ \updelta = \sum_{j=r}^{\infty} d_{j}\uptheta^{j}  \]
  with $\updelta'=\upalpha_{1}'=\upalpha_{2}'$.  
 If $\updelta$ first differs from either $\upalpha_{1}$ or $\upalpha_{2}$ in an index $j<i$, then by the above arguments \[ d_{j}= g_{j} \pm 1. \]  But this implies that $\upalpha_{1}$ and $\upalpha_{2}$ would 
  have to have the same tail parity, opposite to that of $\updelta$, which contradicts the fact that $\upalpha_{1}$ and $\upalpha_{2}$ have opposite tail parity.    Thus $\updelta =\upgamma + d_{i}\uptheta^{i}+\cdots$.  But then $d_{i}$ must differ from one of $c_{i}$, $c_{i}+1$ , and must
  be equal to one of the two, which would imply that $\updelta$ is equal to one of either $\upalpha_{1}$ or $\upalpha_{2}$.
  That $\upalpha_{1}'=\upalpha_{2}'\in\mathcal{O}_{K}$ follows immediately from the calculation \[ a(1+\uptheta^{-2}+\uptheta^{-4}+\cdots  )=\uptheta.\]
  \end{proof}
  
  
  \begin{theo}\label{NegExactly2} If $\upalpha\in \mathcal{O}_{K}$ is negative, there are exactly two Laurent series $\upalpha_{1},\upalpha_{2}\in\mathcal{O}^{+}_{\uptheta}$ in its Cauchy class.
  In particular \[ \uppi (\upalpha )=\uppi (\upalpha_{1})=\uppi (\upalpha_{2}).\]
  \end{theo}
  
  \begin{proof}  If $\upalpha =-\uptheta^{N}$ then we may take $\upalpha_{1}=(-1)_{1}\uptheta^{N}$ and   $\upalpha_{2}=(-1)_{2}\uptheta^{N}$.
  where $(-1)_{1}, (-1)_{2}$ are as in {\it Note} \ref{-1inverses}.
    Thus we assume now
  that $\upalpha \not= -\uptheta^{N}$ for any $N$. 
Let
  $\uptheta^{N+1}$ be the smallest power so that $\upalpha+ \uptheta^{N+1}>0$.  Then we may write 
  \[  \upalpha = \sum_{i=n}^{N} b_{i}\uptheta^{i}  - \uptheta^{N+1} ,\]
  where the sum $\sum_{i=n}^{N} b_{i}\uptheta^{i}  >0$ is greedy and by the minimality of $N+1$, $b_{N}>0$.  Then we have 
  \[ \upalpha_{1}=\sum_{i=n}^{N} b_{i}\uptheta^{i} + (-1)_{1} \uptheta^{N+1}=_{\rm gr} \sum_{i=n}^{N} b_{i}\uptheta^{i} + a(\uptheta^{N+2}+\uptheta^{N+4} +\cdots )\]
  defines an element of $\mathcal{O}_{\uptheta}^{+}$ in the Cauchy class of $\upalpha$.  On the other hand, consider 
  \begin{align}\label{alpha2original}  \upalpha_{2} & =\sum_{i=n}^{N} b_{i}\uptheta^{i} + (-1)_{2} \uptheta^{N+1} \nonumber \\
  & = \sum_{i=n}^{N-1} b_{i}\uptheta^{i} + (b_{N}+1)\uptheta^{N}+ (a-1)\uptheta^{N+1} + a\left(\uptheta^{N+3}+\uptheta^{N+5}+\cdots \right). \end{align}
 If $b_{N}\leq a-2$, this expression is greedy.  Otherwise, if $b_{N}=a-1$ and $b_{N-1}>0$, then resolving the implied Fibonacci relation gives 
 \[ \upalpha_{2}= \sum_{i=n}^{N-2} b_{i}\uptheta^{i} + (b_{N-1}-1)\uptheta^{N-1}+ a\left(\uptheta^{N+1}+\uptheta^{N+3}+\cdots \right) , \]
 which is greedy.  If $b_{N}=a$ (so that $b_{N-1}=0$), we have a carry at the term $\uptheta^{N}$, which resolves to 
  \[   \upalpha_{2}=  \sum_{i=n}^{N-3} b_{i}\uptheta^{i}  +  (b_{N-2}+1)\uptheta^{N-2} +  (a-1)\uptheta^{N-1}+ a\left(\uptheta^{N+1}+\uptheta^{N+3}+\cdots \right)   \]
  Note that this is of the same shape as the original expression (\ref{alpha2original}): thus inductively we conclude that, put into greedy form, $\upalpha_{2}$ gives
  an element in $\mathcal{O}_{\uptheta}^{+}$ distinct from $\upalpha_{1}$ and defining the same Cauchy class as $\upalpha$.
 Now suppose that there was a third element $\updelta\in \mathcal{O}^{+}_{\uptheta}$ in the Cauchy class of $\upalpha$  distinct from $\upalpha_{1}$ and $\upalpha_{2}$. 
But then $\updelta'=\upalpha'= \upalpha_{1}'=\upalpha_{2}'$, and by Proposition \ref{ShapeOfMultPre},  this is not possible.
  \end{proof}

   \begin{coro}  The pre-image of $\hat{\upalpha}\in\widehat{\mathcal{O}}_{\uptheta}$ with respect to the homomorphism $\uppi: \mathcal{O}_{\uptheta}\rightarrow \widehat{\mathcal{O}}_{\uptheta}$ 
    has either one or three elements; the set of $\hat{\upalpha}\in\widehat{\mathcal{O}}_{\uptheta}$ with
  pre-image consisting of three elements is $\mathcal{O}_{K}$.
  \end{coro}
  
  \begin{proof} First note that every $\upalpha\in\mathcal{O}_{K}$ has exactly three pre-images.  If $\upalpha >0$, by Theorem \ref{posunique}, it has a unique representative in $\mathcal{O}_{\uptheta}^{+}$ (itself), and,
  since by Theorem \ref{NegExactly2}, $-\upalpha$ has exactly two distinct representatives $\upalpha_{1}, \upalpha_{2}\in \mathcal{O}_{\uptheta}^{+}$, it follows that $\upalpha$ is also represented by $-\upalpha_{1}$ and  $-\upalpha_{2}$.
  The same applies for $\upalpha<0$.  Finally, by Proposition \ref{ShapeOfMultPre}, there can be no other elements of $\widehat{\mathcal{O}}_{\uptheta}$ with multiple pre-images.
   \end{proof}
  
   
   For any set $X$, we denote by $\big\{X\big\}_{\leq n}$ the set of subsets of $X$ having $n$ or fewer elements.
  
  \begin{coro} For all $\upalpha,\upbeta \in \mathcal{O}_{\uptheta}$, the cardinality of $\upalpha + \upbeta$, $\upalpha\cdot\upbeta$ is at most $3$. 
   Thus
  the operations $+$ and $\cdot$ are maps of the form
  \[    +,\cdot\; : \mathcal{O}_{\uptheta} \times  \mathcal{O}_{\uptheta}  \longrightarrow  \big\{ \mathcal{O}_{\uptheta}\big\}_{\leq 3}.  \]
    Moreover, the cardinality of $\upalpha + \upbeta$ resp.\ $\upalpha\cdot\upbeta$ 
  is {\rm 1} if $\uppi (\upalpha + \upbeta)$ resp.\ $ \uppi (\upalpha \cdot  \upbeta)\not\in \mathcal{O}_{K}$.
  \end{coro}
  
  \begin{proof} By Theorem \ref{CanMapEpi}, $\uppi$ is a homomorphism of multirings, and since $\widehat{\mathcal{O}}_{\uptheta}$
  is an actual ring,
  $\uppi (\upalpha + \upbeta) = \uppi (\upalpha)+ \uppi ( \upbeta) = \hat{\upalpha}+\hat{\upbeta}$, hence
  $ \upalpha+\upbeta $ is in the pre-image by $\uppi$ of the unique element $  \hat{\upalpha}+\hat{\upbeta}$, hence   $ \upalpha+\upbeta $ consists of no more than three elements.  The same
  applies to the product.
  \end{proof}

 \vspace{3mm}
 
 \noindent \fbox{$\boldsymbol N\boldsymbol(\boldsymbol\uptheta \boldsymbol)\boldsymbol=\boldsymbol1$}
 
  \vspace{3mm}

We now investigate the cardinality of multi-operations for $N(\uptheta )=1$.  Unlike the case of  $N(\uptheta )=-1$, there are elements of $\mathcal{O}^{+}_{K}$ which have
 multiple pre-images with respect to conjugation e.g.\ $1$ has conjugation pre-image itself as well as 
 \[  u = (a-1)\uptheta + (a-2)(\uptheta^{2} + \uptheta^{3}+\cdots ), \]
 since 
 \begin{align*}u' & = (a-1) \uptheta^{-1} +(a-2)\left(\uptheta^{-2} + \uptheta^{-3}+\cdots \right)  \\
 & =\uptheta^{-1} + \frac{a-2}{\uptheta-1} \\
 & =\uptheta^{-1} + \frac{\uptheta -2+\uptheta^{-1}}{\uptheta-1} 
 =1.\end{align*}
 Let us write 
 \[ 1_{1}=1,\quad 1_{2}=(a-1)\uptheta + (a-2)\left(\uptheta^{2} + \uptheta^{3}+\cdots \right)  .\]
 By the Fibonacci relation, $\uptheta^{2}=a\uptheta-1$, the sum $(-1) + 1_{2}=0$:
 \[-1 +(a-1)\uptheta + (a-2)\left(\uptheta^{2} + \uptheta^{3}+\cdots \right)  =-\uptheta  +(a-1)\uptheta^{2} +  (a-2)\left(\uptheta^{3} + \uptheta^{4}+\cdots \right) =\cdots =0.\]
   On the other hand,  $-1$ has a   {\it positive} representative involving $T$: 
\[ (-1)_{3}:= T (1+\uptheta + \uptheta^{2}+\cdots )=  T^{-1} (a-2)\left(\uptheta + \uptheta^{2} +\cdots \right)  \]
since 
\[ T^{-1} (a-2)\left(\uptheta + \uptheta^{2} +\cdots \right) +1 =0 \quad \Longleftrightarrow \quad 
-1+ (a-1)\uptheta +  (a-2)\left(\uptheta^{2} + \uptheta^{3} +\cdots \right) =-1+1_{2}=0
. \]
Therefore, $-1$ has in total three representatives:
\[  -1=-1_{1}, -1_{2}, \text{ and }   (-1)_{3}:  \]
And the same is true of $1$: 
\[  1=1_{1}, 1_{2} \text{ and }  1_{3} := -T \left(1+\uptheta + \uptheta^{2}+\cdots \right) = -T^{-1} (a-2)\left(\uptheta + \uptheta^{2} +\cdots \right). \]
There are no more representatives of either 1 or -1, as the following result shows.

 
\begin{lemm}\label{Norm1ConjLemma} Let $\upalpha, \upbeta\in \mathcal{O}_{\uptheta}$ be of the same type (both with or without a factor of $T^{-1}$).  If $r=\upalpha'=\upbeta'$, then either $\upalpha=\upbeta$ or $r\in \mathcal{O}_{K}$.  In the latter case, there are exactly
two pre-images of $r$ of the same type.
\end{lemm}

\begin{proof} First suppose that we have two distinct Laurent series in $\uptheta$, neither having a factor of $T^{-1}$ or a sign $-$, having the same conjugate: 
 \[  \upalpha = \sum_{i=m}^{\infty} b_{i}\uptheta^{i},\quad  \upbeta = \sum_{i=n}^{\infty} c_{i}\uptheta^{i},  \quad \upalpha'=\upbeta'.\]
 Thus we have 
 \begin{align} \label{equalconj}  \sum_{i=m}^{\infty} b_{i}\uptheta^{-i}  = \sum_{i=n}^{\infty} c_{i}\uptheta^{-i}    .\end{align}
 If there are no infinite forbidden blocks in either expression in (\ref{equalconj}) (as defined in (\ref{InfForBlock})), then the conjugates, as beta expansions of elements of $\R$, are greedy. By unicity of the greedy expansion in $\R$, their coefficients are equal, hence $\upalpha=\upbeta$.
If one of them, say $\upalpha'$, contains an infinite forbidden block, it has the form
\begin{align} \label{N=1ConjEqual} \upalpha' & = \cdots + (a-2)\uptheta^{-k-2} + (a-2)\uptheta^{-k-1} + (a-1)\uptheta^{-k} + b_{k-1} \uptheta^{-k+1} + \cdots +b_{m}\uptheta^{-m} \nonumber \\
& =  (b_{k-1} +1) \uptheta^{-k+1} + b_{k-2} \uptheta^{-k+2}+ \cdots +b_{m}\uptheta^{-m} , 
\end{align}  
where $k\geq m$ and the second equality follows from (\ref{InfForBlockGreedy}).
We claim that the Laurent polynomial in (\ref{N=1ConjEqual}) is greedy.  Indeed, first we note that necessarily $ b_{k-1}\leq a-2$ 
(otherwise we would have a consecutive pair of $a-1$ coefficients, which forms a forbidden block).
Hence  $ b_{k-1}+1\leq a-1$.   If $ b_{k-1}+1 = a-1$, and if it were to form in addition a forbidden block in (\ref{N=1ConjEqual}), i.e., there exists $l\leq k-2$ with $b_{l}=a-1$, it means that we already had a forbidden block in $\upalpha$, of the shape
\[ \cdots (a-1)\uptheta^{k} + (a-2) \uptheta^{k-1} + \cdots + (a-2)\uptheta^{l-1} + (a-1)\uptheta^{l}+\cdots . \]
 So in greedy form $\upalpha'$ reduces to a polynomial.  Therefore the same is true of $\upbeta'$: either 
 \begin{enumerate}
\item[1.]  $\upbeta$ is a polynomial with $c_{i}=0$ for all $i\geq k$,  $c_{k-1}= b_{k-1}+1$ and $c_{l}=b_{l}$ for $l\leq k$ in which
 case we have $\upalpha\not=\upbeta$, or else, 
 \item[2.] $\upbeta$ like $\upalpha$ contains an infinite forbidden block.  In this case we claim that  $\upalpha=\upbeta$.  Indeed, suppose we have 
 \[   \upalpha' =\cdots + (a-2)\uptheta^{-2} + (a-1)  +  \upxi    = \cdots +  (a-2)\uptheta^{-k-1}   + (a-1)\uptheta^{-k} +\upeta   =\upbeta'   , \]
 where we suppose (after possibly multiplying $\upalpha, \upbeta$ by a suitable common power of $\uptheta$) that the infinite forbidden block
 of $\upalpha'$ begins in the power $0$, and where
$\upxi , \upeta$  are greedy polynomials whose lowest order powers are $\geq \uptheta, \uptheta^{-k+1}$, respectively.
 Upon resolving the infinite forbidden blocks, we obtain
 \[   \upalpha' =_{\rm gr}  \uptheta +\upxi  =    \uptheta^{-k+1} + \upeta =_{\rm gr}    \upbeta' .  \]
 This implies that $k=0$, and thus $\upxi = \upeta$, hence $\upalpha = \upbeta$.
 \end{enumerate}
 In particular, if there is a third element $\upgamma$, also positive and with no $T^{-1}$ factor, with $\upgamma'=\upalpha'=\upbeta'$, then $\upgamma$ must be equal to either $\upalpha$ or $\upbeta$.  
 The cases where the signs are negative or each series contains a factor of $T^{-1}$ are proved in the same way.
 We now indicate how the statement in the Lemma follows from the above arguments. First, if $r=\upalpha' = \upbeta'$ but $\upalpha\not=\upbeta$, then one of the two
 must have an infinite forbidden block, hence $r\in \mathcal{O}_{K}$.  In this case, we cannot have a third element $\upgamma\not=\upalpha, \upbeta$
 of the same type with
 with $r=\upgamma'$.
\end{proof}

\begin{lemm}\label{Norm1ConjLemma2} Let $r\in \mathcal{O}_{K}$.  If $r\in \mathcal{O}^{0}_{K}$ ($r\in \mathcal{O}^{1}_{K}$),  there exists a unique series $\upgamma\in \mathcal{O}_{\uptheta}$ having a factor of $T^{-1}$ (having no factor of  $T^{-1}$)  with $\upgamma'=r$.
\end{lemm}

\begin{proof} Consider $r \in \mathcal{O}_{K}^{0}$; without loss of generality we may assume $r>0$.  Let $\upalpha \in \mathcal{O}_{K}^{0}$  be such that $\upalpha' =r$.  Thus $\upalpha >0$.  Then the product
 \[
\upalpha\cdot 1_{3}= - T^{-1} \upalpha \cdot  (a-2)(\uptheta +\uptheta^{2} +\cdots ) \subset \mathcal{O}_{\uptheta}
 \]
 (which possibly is multivalued)
has  conjugate equal to $r$.  Since the product $\upalpha \cdot  (a-2)(\uptheta +\uptheta^{2} +\cdots ) $ is nonempty, there exists an element of $\mathcal{O}_{\uptheta}$
\[ \upgamma =- T^{-1} \sum_{i=m}^{\infty} b_{i} \uptheta^{i} \;\; \in \;\;  - T^{-1} \upalpha \cdot  (a-2)(\uptheta +\uptheta^{2} +\cdots ) .\]
We claim that $\upgamma$ is the unique element possessing the factor $T^{-1}$
such that $\upgamma'=r$. (We note that such an element must carry the sign $-$ to conjugate to a positive real number.)  For otherwise, if we have $ \upeta =- T^{-1} \sum_{i=n}^{\infty} c_{i} \uptheta^{i} $ with $\upeta'=r$,  then passing to conjugates, we have
\[ 
\sum_{i=m}^{\infty} b_{i} \uptheta^{-i} =\sum_{i=n}^{\infty} c_{i} \uptheta^{-i}  
= (-T)'r=
 -(T\upalpha)'.
\]
However,  $0<-(T\upalpha)'$ has conjugate $-T\upalpha<0$, hence $-(T\upalpha)' \not\in \mathcal{O}_{K}^{0}$.  It follows that neither of the conjugates $\upgamma'$,  $\upeta'$
contain an infinite forbidden block (otherwise they would be in $ \mathcal{O}_{K}^{0}$), hence they are both greedy.  Thus, by unicity of greedy expansions in $\R$, $ \upgamma=\upeta$.
For $r= T^{-1}h \in \mathcal{O}_{K, +}^{1}$, the argument is essentially the same: there exists a series $\upalpha$, with no $T^{-1} $ and with sign $-$,
producing $r$ as conjugate, and the argument needed to show that this element is unique is the same as that used in the previous case.
\end{proof}
 
 \begin{theo}\label{N=1ConjHasNoMore2} The conjugation homomorphism $': \mathcal{O}_{\uptheta}\rightarrow \R$ has exactly three pre-images at each point of $\mathcal{O}_{K}$ and otherwise has two pre-images.  
  \end{theo}
 
 \begin{proof}   
 First suppose that $r\in \mathcal{O}_{K}$.  By Lemmas \ref{Norm1ConjLemma} and  \ref{Norm1ConjLemma2},
there are exactly three elements of
$\mathcal{O}_{\uptheta}$ whose conjugate is $r$.  For $r\not\in\mathcal{O}_{K}$, we have the infinite greedy expansion 
\[ r=_{\rm gr} \sum_{i=m}^{\infty} b_{i}\uptheta^{-i} .\]
Then $\upalpha =_{\rm gr} \sum_{i=m}^{\infty} b_{i}\uptheta^{i}$ satisfies $\upalpha'=r$, and by Lemma \ref{Norm1ConjLemma},
it is the unique element of $\mathcal{O}_{\uptheta}$ with this property and not having a factor of  $T^{-1}$.  On the other hand, 
$\upbeta = \upalpha\cdot 1_{3}$ has a factor of $T^{-1}$ and,  again by Lemma \ref{Norm1ConjLemma}, it is the unique such element
with conjugate $r$.
 \end{proof}
 
 
 As in the case of $N=-1$, we have the
 
\begin{coro} For all $\upalpha,\upbeta \in \mathcal{O}_{\uptheta}$, the cardinality of $\upalpha + \upbeta$, $\upalpha\cdot\upbeta$ is at most three.  Thus
  the operations $+$ and $\cdot$ are maps of the form
  \[    +,\cdot\; : \mathcal{O}_{\uptheta}^{2} \longrightarrow   \big\{ \mathcal{O}_{\uptheta}\big\}_{\leq 3}.  \]
  \end{coro}

\section{The $\uptheta$-adic Topology and Multicontinuity}

 \vspace{3mm}
 
 \noindent \fbox{$\boldsymbol N\boldsymbol(\boldsymbol\uptheta \boldsymbol)\boldsymbol=\boldsymbol-\boldsymbol1$}
 
  \vspace{3mm}
  
  Given $\upalpha = \pm \sum_{i\geq n} b_{i}\uptheta^{i}\in \mathcal{O}_{\uptheta}$ with $b_{n}\not=0$, denote $v_{\uptheta}(\upalpha ) =n$
and define
\[ \mathcal{O}_{\uptheta, N} = \{ \upalpha \in \mathcal{O}_{\uptheta}\; :\;\; v_{\uptheta}(\upalpha )\geq N \}.\]
We embed $\mathcal{O}_{\uptheta, N}$ in the Cantor set 
\[ C_{N} = \{ 0,\dots , a\}^{[N,\infty)}\sqcup \{  -a,\dots, 0\}^{[N,\infty)},\quad \upalpha \longmapsto (b_{N}, b_{N+1}, \dots )   \]
where, if $n>N$, we define the coefficients as $b_{N}, \dots ,b_{n-1}$ to be $0$. 

\begin{prop}\label{BasisIsCantor} The image of $\mathcal{O}_{\uptheta, N}$ in $C_{N}$ is closed and perfect, and is thus itself a Cantor set in the subspace topology.
\end{prop}

\begin{proof} If $\{ \upalpha_{i} \} \subset \mathcal{O}_{\uptheta, N}$ and $\upalpha_{i}\rightarrow \upalpha$ in $C_{N}$ in the Tychonoff (weak) topology, then the coordinates of the limit
point $\upalpha$ must also satisfy the greedy condition, hence $\upalpha\in   \mathcal{O}_{\uptheta, N}$.  It is trivial to see that every point $\upalpha\in \mathcal{O}_{\uptheta, N}$ is a limit point
of points of  $\mathcal{O}_{\uptheta, N}$: any sequence of partial sums gives a sequence of elements of $ \mathcal{O}_{\uptheta, N}$ converging to $\upalpha$.
\end{proof}

The canonical inclusions $C_{N}\hookrightarrow C_{N-1}$ are continuous and thus we may give $ \mathcal{O}_{\uptheta}$ the induced direct limit topology coming from
the continuous inclusions $ \mathcal{O}_{\uptheta, N} \hookrightarrow  \mathcal{O}_{\uptheta, N-1}$.   This makes $ \mathcal{O}_{\uptheta}$ a locally Cantor topological space.  We call this topology the {\bf {\em $\boldsymbol\uptheta$-adic topology}}.

\begin{prop} The canonical map $\uppi: \mathcal{O}_{\uptheta}\rightarrow \widehat{\mathcal{O}}_{\uptheta}$ is continuous.
\end{prop}

\begin{proof} If $\upalpha_{i}\rightarrow \upalpha$ in the $\uptheta$-adic topology, then the associated Cauchy classes satisfy
$| \hat{\upalpha}_{i} -\hat{\upalpha}|_{\uptheta}\rightarrow 0$ by the infratriangle inequality, which implies convergence.
\end{proof}

\begin{prop}\label{ConjCont} The conjugation map 
\[  \upalpha = \sum_{i\geq n} b_{i}\uptheta^{i}\longmapsto \upalpha' =  \sum_{i\geq n} (-1)^{i}b_{i}\uptheta^{-i}\in\R \]
defines a continuous and closed map $\mathcal{O}_{\uptheta} \longrightarrow \R$.
\end{prop}

\begin{proof} It is evident that $\upalpha'$ defines a convergent Laurent series in $\uptheta^{-1}$ since its coefficients are bounded in absolute value by $a$.  If $\upalpha $ and $\upbeta$ are close
in the $\uptheta$-adic topology, it means that they agree up to some large index $N$, but then the conjugate of $\upalpha-\upbeta$ must be small.  In particular, if $\upalpha_{i}\rightarrow \upbeta$ then
$\upalpha_{i}'\rightarrow \upbeta'$.  Since $\mathcal{O}_{\uptheta}$ has a first countable topology, this implies continuity.
Now suppose that $B\subset \mathcal{O}_{\uptheta}$ is closed but non-compact,  we claim that $B'$ is also closed. If $\{ \upalpha_{i}\} \subset B$ is a sequence $\rightarrow \infty$
in the $\uptheta$-adic topology, then in particular $|\upalpha_{i}|_{\uptheta}\rightarrow \infty$ and hence $\upalpha'_{i}\rightarrow\pm\infty$ in $\R$.
Indeed,
\[  |\upalpha_{i}'|  \geq |\upalpha_{i}|_{\uptheta} \cdot \big(1- (a-1)\uptheta^{-1} - a\uptheta^{-3} -a\uptheta^{-5} -\cdots  \big) = |\upalpha_{i}|_{\uptheta}  \uptheta ^{-1}\longrightarrow \infty ,\]
where the last equality is that found in (\ref{hyponh'}).
Thus, if $x$ is a limit point of $B'$, so that there exists $\upbeta_{i}'\rightarrow x$ for $\upbeta_{i}\in B$, then, after passing to a subsequence, we may
assume $\upbeta_{i}\rightarrow \upbeta\in B$.  By continuity of conjugation, $\upbeta' = x$, hence $B'$ is closed.
\end{proof}

We now consider the topological properties of the multi-operations $+$ and $\cdot$, which take values in
\[ 
 \big\{ \mathcal{O}_{\uptheta}\big\}_{\leq 3}\subset 2^{\mathcal{O}_{\uptheta}}. \]  We topologize $ \big\{ \mathcal{O}_{\uptheta}\big\}_{\leq 3}$
 by giving it the subspace topology, where $2^{\mathcal{O}_{\uptheta}}$ is given the usual mapping space topology.
We will say that a multivalued function between spaces $X$ and $Y$, 
\[  F:X\rightarrow 2^{Y}\]
is {\bf {\em multicontinuous}} if it is {\it upper hemicontinuous}: that is, for every $U\subset Y$ open with $F(x)\subset U$, there exists $V\subset X$ an open neighborhood of $x$ such
that $F(v)\subset U$ for all $v\in V$.  

\begin{note} The motivation for the definition of multicontinuity appearing here comes from its appearance in the following Closed Graph Theorem (for more on the ubiquity of the latter in characterizations of
regularity, see the blog post of T. Tao \cite{Tao}):
\begin{theocgt}[Theorem 17.11 of \cite{AB}]  A multivalued map $  F:X\rightarrow 2^{Y}$ with $Y$ compact Hausdorff has a closed graph in $X\times 2^{Y}$ if and only if it is upper hemicontinuous and closed valued.
\end{theocgt}
\end{note}

\begin{theo}\label{hemicontnote}
 If $F(x)$ is compact for all $x$, then upper hemicontinuity is equivalent to the property that for any convergent net $x_{\upalpha}\rightarrow x$ and $y_{\upalpha}\in F(x_{\upalpha})$,
$y_{\upalpha}$ has a limit point in $F(x)$.
\end{theo}

\begin{proof}  See the proof of Theorem 17.16 of \cite{AB}.
\end{proof}

\begin{theo}\label{MultiContMaps} The operations $+$ and $\times$ define multicontinuous maps
 \[    +,\cdot\; : \mathcal{O}_{\uptheta}\times \mathcal{O}_{\uptheta} \longrightarrow 2^{\mathcal{O}_{\uptheta}}. \]
\end{theo}

\begin{proof}  We prove the multicontinuity of  $+$ ; that of $\cdot$ is analogous. 
Suppose that $x_{\upalpha}\rightarrow x$, $y_{\upalpha}\rightarrow y$ in $\mathcal{O}_{\uptheta}  $, and let $w_{\upalpha}\in x_{\upalpha} +y_{\upalpha}$.
By Theorem \ref{hemicontnote}, it suffices to show that there exists $w\in x+y$ and a subsequence of $\{ w_{\upalpha}\}$ converging to $w$.  
Since pre-images with respect to conjugation have bounded cardinality, and since $w_{\upalpha}' = x_{\upalpha}' + y_{\upalpha}' \rightarrow x'+y'$, we may
find some $w$  satisfying  $w ' = x' +y'$ for which $w_{\upalpha}\rightarrow w$ (after passing to a subsequence, if necessary).  Note that  $w ' = x' +y'$ is a necessary condition for $w\in x+y$.  To show that the latter is indeed
the case we must find sequences of partial sums  $\{x_{n} \} \in x$, $\{y_{n} \} \in y$ so that $\{ w_{n} = x_{n} + y_{n}\}\in w$.  
On the other hand, since  $w_{\upalpha}\in x_{\upalpha} + y_{\upalpha}$ for each $\upalpha$, we may find sequences of partial sums of all parties so that 
\[ \sum_{i=m}^{M_{\upalpha, n}} w_{\upalpha, i}\uptheta^{i}  +\updelta_{\upalpha, n} = \sum_{i=m'}^{M'_{\upalpha, n}} x_{\upalpha,i}\uptheta^{i}  + \upxi_{\upalpha, n} +  \sum_{i=m''}^{M''_{\upalpha, n}} y_{\upalpha,i} \uptheta^{i}   +\upeta_{\upalpha, n},\]
where for each $\upalpha$, the sequences $M_{\upalpha, n}$, $M_{\upalpha, n}'$ and $M_{\upalpha, n}''$ are monotone increasing and $\updelta_{\upalpha, n} $,  $ \upxi_{\upalpha, n}$ and $\upeta_{\upalpha, n}$ are the defects.
By definition of convergence in the $\uptheta$-topology, given $M > m,m',m''$, for $\upalpha$ sufficiently large and $i\leq M$, 
\begin{align}\label{convequ} x_{\upalpha, i}= x_{i},\;\;, y_{\upalpha, i}= y_{i} ,\;\;w_{\upalpha, i}= w_{i}.\end{align}
 Now take a diagonal subsequence $\left\{ (\widetilde{M}_{n}, \widetilde{M}_{n}' ,\widetilde{M}_{n} '')   \right\}$ of the bi-indexed sequence of triples $\left\{ (M_{\upalpha, n}, M'_{\upalpha, n},M''_{\upalpha, n})   \right\}$,
so that for $M=\widetilde{M}_{n}$, resp., $M=\widetilde{M}'_{n}$, resp., $M=\widetilde{M}_{n}''$ 
we have the respective equalities shown in (\ref{convequ}).    Denote the corresponding defects using the same notation.
Then
\begin{align*} 
\sum_{i=m}^{\widetilde{M}_{n}} w_{i}\uptheta^{i}  +\widetilde{\updelta}_{n} =
  \sum_{i=m'}^{\widetilde{M}_{n}}  x_{i}\uptheta^{i}  + \widetilde{\upxi}_{n} +  \sum_{i=m''}^{\widetilde{M}_{n}}  y_{i} \uptheta^{i}   +\widetilde{\upeta}_{n} , 
 \end{align*}
 which gives 
the relation $w\in x+y$. 
\end{proof}

Thus, we shall say that a Marty multi-ring $M$ is a {\bf {\em topological Marty multiring}} if $M$ has a topology with respect to which its operations are multi-continuous.  

\begin{coro} $\mathcal{O}_{\uptheta}$ is a topological Marty multiring.
\end{coro}  

 \vspace{3mm}
 
 \noindent \fbox{$\boldsymbol N\boldsymbol(\boldsymbol\uptheta \boldsymbol)\boldsymbol=\boldsymbol1$}
 
  \vspace{3mm}
  
    In keeping with the notation of the $N(\uptheta )=1$ case discussed in \S \ref{InfraNormSection}
let us write 
\[  \mathcal{O}_{\uptheta}^{0} = \{ \upalpha \in\mathcal{O}_{\uptheta}\; :\;\; \upalpha'\geq 0\} ,\quad \mathcal{O}_{\uptheta}^{1} = \{ \upalpha \in\mathcal{O}_{\uptheta}\; :\;\; \upalpha'\leq 0\}   \]
so that the elements $\upalpha\in  \mathcal{O}_{\uptheta}^{0} $ are Laurent series in $\uptheta$ and the elements $\upbeta\in  \mathcal{O}_{\uptheta}^{1} $ are of the form $T\upalpha$, $\upalpha
\in  \mathcal{O}_{\uptheta}^{0} $ .
Then, we define the basis element $\mathcal{O}_{\uptheta, N}^{0}$ of $ \mathcal{O}_{\uptheta}^{0}$ exactly as in the case of $N=-1$, and the corresponding basis element for $ \mathcal{O}_{\uptheta}^{1}$
is defined  $T\mathcal{O}_{\uptheta, N}^{0}$.   Each of these sets is then a Cantor set via the analog of Proposition \ref{BasisIsCantor} (proved using the same kind of argument),
hence each of $ \mathcal{O}_{\uptheta}^{0}$ and $ \mathcal{O}_{\uptheta}^{1}$ receives a locally Cantor topology by way of the union of the respective Cantor basis elements.  Finally
$\mathcal{O}_{\uptheta} =  \mathcal{O}_{\uptheta}^{0}\cup  \mathcal{O}_{\uptheta}^{1}$ is given the union topology.
The analog of Proposition 12 -- continuity and closedness of the conjugation map -- is proved again using the same sort of argument: where we observe that our topology is defined in such a way
that a non-zero sequence of elements in $\mathcal{O}_{\uptheta}^{0}$ cannot converge to a non-zero element of $\mathcal{O}_{\uptheta}^{1}$ and {\it vice versa}.   With this in place, the analogs of Theorem \ref{MultiContMaps} and its Corollary
follow.

\section{(Multi) Field Structure}\label{FieldSection}

In this section we will show that $\widehat{\mathcal{O}}_{\uptheta}$ is a field isomorphic to $\R$ and 
$\mathcal{O}_{\uptheta}$ is a Marty multifield.

 \vspace{3mm}
 
 \noindent \fbox{$\boldsymbol N\boldsymbol(\boldsymbol\uptheta \boldsymbol)\boldsymbol=\boldsymbol-\boldsymbol1$}
 
  \vspace{3mm}

We will first show that $\widehat{\mathcal{O}}_{\uptheta}$ is a field isomorphic to $\R$;  once we have accomplished that, it will be immediate that $\mathcal{O}_{\uptheta}$ is a Marty multifield. 

\begin{lemm}\label{InverseNorm} Let $g\in \mathcal{O}_{K}$ be invertible in $\widehat{\mathcal{O}}_{\uptheta}$,  having inverse $ h\in \widehat{\mathcal{O}}_{\uptheta}$.  Then 
 \[ |h |_{\uptheta}  =  \frac{1}{\;|g|_{\uptheta}}. \]
 \end{lemm}
 
 \begin{proof} Rescaling by a power of $\uptheta$ we may assume that 
 \[  g = b_{0} + b_{1}\uptheta + \cdots + b_{m}\uptheta^{m},\]
 so that $|g|_{\uptheta}=1$.  In what follows, we abbreviate $b=b_{0}$.  Represent $h$ as a Cuachy sequence
$ h= \{ h_{n}\} $ so that
 \begin{align}\label{InvCondForG}    gh_{n}=_{\rm gr} 1+ e_{n},\end{align}
 where $e_{n}\rightarrow 0$ in $|\cdot |_{\uptheta}$.
  Thus we wish to show that eventually $|h_{n}|_{\uptheta}=1$. Since $g$ starts with the constant term $b$,  the conjugate $g'$ is positive:
 indeed, we have the strict lower bound
 \begin{align}\label{g'lower} g' > 1 - (a-1)\uptheta^{-1} - a\uptheta^{-3} -a\uptheta^{-5}-\cdots = \uptheta^{-1} +1 - \frac{a}{\uptheta-\uptheta^{-1}} =\uptheta^{-1} >0, \end{align}
 where the lower bound involves only odd powers of $\uptheta$ so that all the signs of conjugates of terms following $1$ are negative.
 We have also the strict upper bound
 \begin{align}\label{g'upper}  g'< b+ a\uptheta^{-2} +a\uptheta^{-4} +\cdots = b+\uptheta^{-1} .\end{align}
 The condition (\ref{InvCondForG}) implies after conjugation that
 $g'h_{n}'\rightarrow 1$, thus  $h_{n}'\rightarrow g'^{-1}$ in the euclidean topology.  Moreover, by the uniform bounds (\ref{g'lower}),  (\ref{g'upper}) on $g'$,  for $n$ large we have
 \begin{align}\label{boundonhn}  \frac{1}{b+\uptheta^{-1}} < h'_{n}<  \uptheta  ,\end{align}
 which we emphasize are {\it strict} inequalities.
   The positivity condition on $h_{n}'$ implies that the lowest power of $\uptheta$ in $h_{n}$ is even (e.g.,  if $h_{n} =_{\rm gr} c_{1} \uptheta +\text{higher powers of $\uptheta$}$, then $h_{n}'< -\uptheta^{-1} +(a-1)\uptheta^{-2} + a\uptheta^{-4} + \cdots =-\uptheta^{-2}<0$).
Suppose the lowest
 power is negative, e.g., 
 \[ h_{n} = c_{-2}\uptheta^{-2}+\cdots; \]
 then we have 
 \[ h_{n} ' > \uptheta^{2} - (a-1)\uptheta -a(\uptheta^{-1}  + \uptheta^{-3}+\cdots ) =
 \uptheta
 .\]
 This contradicts the strict upper bound of (\ref{boundonhn}).  For a more general lowest negative power,  $h_{n}=c_{-2k}\uptheta^{-2k}+\cdots$, $k\geq 1$, we obtain $ h_{n} '>\uptheta^{2k-1}\geq \uptheta$
 which again contradicts (\ref{boundonhn}).
 Now suppose the lowest power is positive and $\geq 2$, e.g.,
 \[     h_{n} = c_{2}\uptheta^{2}+ \cdots .\] Then we have 
 \[   h'_{n} < a(\uptheta^{-2} + \uptheta^{-4} +\cdots )  =\uptheta^{-1}\leq \frac{1}{b+\uptheta^{-1}},\]
 since 
 \[    \uptheta^{-1}(b+\uptheta^{-1}) \leq \uptheta^{-1} ( a+ \uptheta^{-1}) = 1.\]
This contradicts the strict lower bound in (\ref{boundonhn}).  And for a general lowest positive power $h_{n} = c_{2k}\uptheta^{2k}+ \cdots$ we obtain $h'_{n}<\uptheta^{-2k+1}\leq\uptheta^{-1}$
which continues to contradict (\ref{boundonhn}).
 Therefore $h_{n}$ must start with a constant term, implying that its norm is $1$.
 \end{proof}
 
   
   
  We now consider the central technical point of this section: the invertibility of elements of $\Z$ in $\widehat{\mathcal{O}}_{\uptheta}$.  We first observe that 
  when $a\equiv 0 \mod n$, then $n$ is invertible.  Indeed, if $n|a$, an inverse of $n$ is
  \[ n^{-1} = 1 + \frac{a(n-1)}{n} ( \uptheta + \uptheta^{3} + \uptheta^{5}+\cdots), \]
  since 
\begin{align*}    n\cdot  n^{-1} & = 1 + (n-1)\left\{ 1+ a\uptheta + a\uptheta^{3} + \cdots       \right\}  \\
& = 1 + (n-1)\left\{ \uptheta^{2}+ a\uptheta^{3} + \cdots       \right\} \longrightarrow 1,
\end{align*}
where the convergence to $1$ follows from a cascading series of Fibonacci relations. Note that as $n^{-1}$ is presented as a Laurent series, 
it defines an element of
$\mathcal{O}_{\uptheta}$ inverting $n$ there as well.

Now suppose that $n\nmid a$.  
It will be enough to prove that $n$ is invertible when $n$ is prime, and by the example above, we
may suppose further that \[a\equiv b\mod n, \quad b\not=0.\]   

\begin{lemm}\label{FloorPisotCalc} Suppose $n$ is prime and $n\nmid a$ and let $c\in \{0,\dots ,n-1\}$.  Then  there exists 
$d\in \{1,\dots , n-1\}$ such that 
  \begin{align}\label{PisoFloorformula}   n   \left\lfloor \frac{ac}{n}\right\rfloor  = ac + d-n, \quad  n   \left\lceil \frac{ac}{n}\right\rceil  = ac + d .\end{align}  
\end{lemm}

\begin{proof} Since $n$ is prime and $c<n$, $n\nmid ac$, so  $ac/n = m+r/n$ for $r\in \{1,\dots , n-1\}$. Thus 
\[ n   \left\lfloor \frac{ac}{n}\right\rfloor = nm = ac -r = ac +d -n ,\quad d=n-r\in \{1,\dots , n-1\}. \]
Moreover, as $ac/n\not\in \Z$,  
\[ n\left\lceil\frac{ac}{n}\right\rceil = n \left\lfloor \frac{ac}{n}\right\rfloor  +n  = ac +d. \]
\end{proof}

We describe an algorithm for finding the inverse of $n\in\N$
\begin{align}\label{ninv} n^{-1} = 1 + e_{1} \uptheta + e_{2}\uptheta^{2}+\cdots ;\end{align} 
we break down the description of this algorithm into two parts.   As we introduce the algorithm, we illustrate it in the particular example
\begin{align}\label{PartEx} n=3, \quad a=3m+1,\quad b=1,\end{align}  before proceeding to prove that it does what it is claimed to do: provide an inverse.

\vspace{3mm}

\noindent \fbox{A}  Define a (periodic) sequence \[ \{ c_{i}\},\quad c_{i} \in\{ 0,1,\dots , n-1\} , \]   via the recursion
     \begin{align}\label{crecursion}  c_{0} = n-1,\;\; c_{1} = b,\quad c_{i} = c_{i-2}  - bc_{i-1} \mod n, \quad \text{ for $i\geq 2$}. \end{align}
     
 \begin{illexam}  In the context of the choices in (\ref{PartEx}), the first eight terms of the recursion produce
 \[   c_{0} =2, c_{1} =1,  c_{2} =1, c_{3} =0,  c_{4} =1, c_{5} =2,  c_{6} =2, c_{7} =0,\]
 after which the series repeats.
 \end{illexam}    
     
     \vspace{3mm}
     
\noindent \fbox{B}   Using Lemma \ref{FloorPisotCalc}, write
     \begin{align}\label{neiformula}   n   \left\lfloor \frac{ac_{i}}{n}\right\rfloor  = ac_{i} + d_{i}-n, \quad  n   \left\lceil \frac{ac_{i}}{n}\right\rceil  = ac_{i} + d_{i},\quad d_{i}\in \{ 1,\dots , n-1\}. \end{align}  
     Then we define
     \[e_{1}  := \left\lceil \frac{a(n-1)}{n}\right\rceil  \]
     and for $i\geq 2$,
        \[  e_{i} :=
     \left\{ \begin{array}{ll}  \left\lfloor \frac{ac_{i-1}}{n}\right\rfloor & \text{if $c_{i-2} +d_{i-1}-n\geq  0$} \\
      & \\
     \left\lceil \frac{ac_{i-1}}{n}\right\rceil  & \text{otherwise.} 
     \end{array} \right. 
     \]
     Note that $0\leq e_{i} \leq a$.
      \begin{illexamc} The sequence of floors and ceilings corresponding to the choices in (\ref{PartEx}) give
     \[ e_{1}= \left\lceil \frac{2a}{3}\right\rceil , \;\;e_{2} =e_{3} =\left\lfloor \frac{a}{3} \right\rfloor,   \;\; e_{4}=0, \;\; e_{5} =\left\lceil \frac{a}{3} \right\rceil,
      \;\;e_{6} =\left\lceil \frac{2a}{3}\right\rceil , \;\;e_{7} =\left\lfloor \frac{2a}{3} \right\rfloor , \;\; e_{8}=0.
        \]
        For example, to calculate $e_{2}$, we consider 
\[    3   \left\lceil \frac{ac_{1}}{3}\right\rceil  = 3   \left\lceil \frac{3m+1}{3}\right\rceil  = 3m+3=ac_{1} + d_{1}  = a\cdot 1 + 2 \Longrightarrow d_{1}=2.    \]
Since $c_{0} + d_{1}-3 = 1\geq 0$, $e_{1}$ is calculated with the floor, as indicated above.
 \end{illexamc} 

\begin{lemm}\label{cdlemm} We have 
\begin{align}\label{altci}  c_{i} = \left\{ \begin{array}{ll}  c_{i-2} +d_{i-1}-n & \text{if $c_{i-2} +d_{i-1}-n\geq  0$} \\
 & \\
c_{i-2} +d_{i-1} & \text{otherwise.} 
\end{array} \right.
\end{align}
In either of the two cases in {\rm (\ref{altci})}, we have
    \begin{align}\label{neIrecurssion} ne_{i} =      ac_{i-1}+c_{i} - c_{i-2} .\end{align}  
    \end{lemm}
    \begin{proof}
We claim that \begin{align}\label{littleclaim} c_{i}\equiv  c_{i-2} +d_{i-1}-n\equiv c_{i-2} +d_{i-1}\mod n.\end{align}
Indeed, by definition of the sequence $\{ c_{i}\}$,
$  c_{i} \equiv c_{i-2}  - bc_{i-1}  \mod n$.  On the other hand, by (\ref{neiformula}),  $d_{i-1}\equiv -ac_{i-1} \mod n$ and since $a\equiv b\mod n$,
$d_{i-1}\equiv -bc_{i-1} \mod n$, which gives (\ref{littleclaim}).
We now prove (\ref{altci}).
Either $c_{i-2} +d_{i-1}-n$ or $c_{i-2} +d_{i-1} \in \{ 0,\dots , n-1\} $,
since $c_{i-2} \in \{ 0,\dots , n-1\}$ and $d_{i-1} \in \{ 1,\dots , n-1\}$.  
 If  $c_{i-2} +d_{i-1}-n\geq  0$,  then it belongs to the set  $\{ 0,\dots , n-2\}$ and, therefore, by (\ref{littleclaim}), it is equal to $c_{i}$.  Otherwise, if 
 $c_{i-2} +d_{i-1}-n< 0$, then $c_{i-2} +d_{i-1}\in \{ 1,\dots , n-1\}$, and again by (\ref{littleclaim}), it is equal to $c_{i}$.
To show (\ref{neIrecurssion}), first assume $c_{i-2} +d_{i-1}-n\geq  0$: then by (\ref{neiformula}),  $e_{i}$ is calculated using the floor, and 
by  (\ref{altci}),
\[ ne_{i}=ac_{i-1} + d_{i-1}-n  =ac_{i-1}  + c_{i} - c_{i-2}. \]
Otherwise, if $c_{i-2} +d_{i-1}-n<0$,  by  (\ref{neiformula}) $e_{i}$ is calculated using the ceiling, and by (\ref{altci}), 
\[ ne_{i}=ac_{i-1} + d_{i-1} =ac_{i-1} + c_{i} - c_{i-2}.   \] 
\end{proof}

The series (\ref{ninv}), defined by the algorithm above, need not be greedy, and thus, as it stands, it does not define an element of $\mathcal{O}_{\uptheta}$.  Nevertheless, as we show in Theorem \ref{NInvertible} below,  
its partial sums define a Cauchy sequence giving rise to an element of $\widehat{\mathcal{O}}_{\uptheta}$ -- also denoted $n^{-1}$  -- and 
we will show that in the $\uptheta$-adic norm $n\cdot n^{-1}\rightarrow 1$, so that it defines an inverse of $n$ in  $\widehat{\mathcal{O}}_{\uptheta}$.

       \begin{illexamc}  
     Write
        \begin{align*}  
x 
& := \left\lceil \frac{2a}{3} \right\rceil \uptheta +  \left\lfloor \frac{a}{3} \right\rfloor\uptheta^{2} + \left\lfloor \frac{a}{3} \right\rfloor\uptheta^{3} +  \left\lceil \frac{a}{3} \right\rceil\uptheta^{5} +  \left\lceil \frac{2a}{3} \right\rceil \uptheta^{6} +\left\lfloor \frac{2a}{3} \right\rfloor\uptheta^{7} \\
& = (2m+1)\uptheta + m\uptheta^{2} + m\uptheta^{3} + (m+1)\uptheta^{5} + (2m+1)\uptheta^{6} +2m\uptheta^{7} .
   \end{align*}
   Then 
   \begin{align*} 3(1+x)  & = 1 + 2 + 3x \\
   & = 1 + 2 + (2a+1)\uptheta + (a-1)\uptheta^{2} + (a-1) \uptheta^{3} + (a+2)\uptheta^{5} + (2a+1)\uptheta^{6} + (2a-2)\uptheta^{7} \\
    & = 1 + \uptheta + (a+1)\uptheta^{2} + (a-1) \uptheta^{3} + (a+2)\uptheta^{5} + (2a+1)\uptheta^{6} + (2a-2)\uptheta^{7} \\
    & =1 +\uptheta^{4} +  (a+2)\uptheta^{5} + (2a+1)\uptheta^{6} + (2a-2)\uptheta^{7}\\
    &=  1 +2\uptheta^{5} + (2a+2)\uptheta^{6} + (2a-2)\uptheta^{7}\\
    & =1 +2\uptheta^{6}+ 2a\uptheta^{7} \\
    & = 1 +2\uptheta^{8}.
   \end{align*}
   Now replace $1+x$ by $1+ x + x\uptheta^{8} $.
   Then 
   \[   3(1+x + x \uptheta^{8}) = 1+2\uptheta^{8} + 3x\uptheta^{8} =1+\uptheta^{8}(2 + 3x) =1+ 2\uptheta^{16} . \]
  Iterating, the inverse of $3$ is thus
   \[    1+ x(1+ \uptheta^{8} +\cdots ) = 1+ x\sum_{n=0}^{\infty} \uptheta^{8n}  . \]
       
 \end{illexamc} 
 
 \begin{exam}  For $n=3$ and $a=3m+2$ (so that $b=2$), the reader may check that  
  the inverse of $3$ is 
   \[    1+ x(1+ \uptheta^{8} +\cdots ) = 1+ x\sum_{n=0}^{\infty} \uptheta^{8n}   \]
   where now 
   \begin{align*}
x 
& = \left\lceil \frac{2a}{3} \right\rceil \uptheta +  \left\lfloor \frac{2a}{3} \right\rfloor\uptheta^{2} + \left\lfloor \frac{a}{3} \right\rfloor\uptheta^{3} +  \left\lceil \frac{a}{3} \right\rceil\uptheta^{5} +  \left\lceil \frac{a}{3} \right\rceil \uptheta^{6} +\left\lfloor \frac{2a}{3} \right\rfloor\uptheta^{7}.
   \end{align*}

 \end{exam}

\begin{theo}\label{NInvertible} The series {\rm (\ref{ninv})}, defined by the algorithm above, converges in $\widehat{\mathcal{O}}_{\uptheta}$ and defines the inverse $n^{-1}$ of $n\in\N$.
\end{theo}
\begin{proof} As indicated above, we may assume without loss of generality that $n$ is prime.  We first show that the series (\ref{ninv}) converges. We note that $e_{i}\leq a$ for all $i$ since $ac_{i-1}/n < a$: thus, if the series
is not greedy, it is because there exist unresolved Fibonacci relations.   If we denote $n^{-1}_{j}=1 + e_{1}\uptheta + \cdots + e_{j}\uptheta^{j}$, we must show that as $i<j\rightarrow\infty$, 
\[  n^{-1}_{j} -n^{-1}_{i}  = e_{i+1}\uptheta^{i+1} + \cdots + e_{j}\uptheta^{j}\longrightarrow 0\]
with respect to the $\uptheta$-adic infranorm.  We may write 
\[ n^{-1}_{j} -n^{-1}_{i} = \left( a \sum_{\substack{k=i+1 \\ e_{k}>0}}^{j} \uptheta^{k} \right) -g \]
where 
\[ g = g_{i+1}\uptheta^{i+1} + \cdots + g_{j}\uptheta^{j},\quad 0\leq g_{k}  = a-e_{k}\leq a-1.\]
Now $g$ is greedy and has infranorm $\leq \uptheta^{-i-1}$, and by infra multiplicativity
\[ \left|a \sum_{\substack{k=i+1 \\ e_{k}>0}}^{j} \uptheta^{k} \right|_{\uptheta} \leq \uptheta^{2} |a|_{\uptheta}\left|\uptheta^{i+1} + \cdots + \uptheta^{j} \right|_{\uptheta} = \uptheta^{1-i} .\]
Then, by the infratriangle inequality,
\[  | n^{-1}_{j} -n^{-1}_{i} |_{\uptheta} =\bigg|a \sum_{\substack{k=i+1 \\ e_{k}>0}}^{j} \uptheta^{k} -g\bigg|_{\uptheta}\leq \uptheta^{3-i}\longrightarrow 0.  \]
Thus the expression for $n^{-1}$ is convergent.
We now consider
\[ n\cdot n^{-1} =  1 + (n-1) + ne_{1}\uptheta + \cdots  =: 1+x .\]
We claim first that
\[   ne_{1}=n   \left\lceil \frac{a(n-1)}{n}\right\rceil  = (n-1)a + b.     \]
Indeed, for $a=mn+b$,
     \[  n   \left\lceil \frac{(n-1)a}{n}\right\rceil   = n\left(  (n-1)m + \left\lceil \frac{(n-1)b}{n}\right\rceil \right) = (n-1)a -(n-1)b+ n \left\lceil \frac{(n-1)b}{n}\right\rceil ; \]
 by Hermite's identity \cite{Hermite},  \cite{KnuthEtAl} (see page 85, (3.24))
    \[    \left\lceil \frac{(n-1)b}{n}\right\rceil  = \left\lceil \frac{n-1}{n}\right\rceil  +   \left\lceil \frac{n-1}{n} - \frac{1}{b}\right\rceil   + \cdots +    \left\lceil \frac{n-1}{n} - \frac{b-1}{b}\right\rceil   = b.  \]
    This proves the claim.  Thus, the first two terms of $x$ are
    \[  (n-1) +  n   \left\lceil \frac{(n-1)a}{n}\right\rceil  \uptheta = b\uptheta + (n-1)\uptheta^{2} = c_{1}\uptheta + c_{0}\uptheta^{2}.  \]
    Now add $ne_{2}$; we have $c_{0} + d_{1}-n = d_{1}-1\geq 0$ since  $d_{1}>0$.     Thus,
    \[ e_{2} =  \left\lfloor \frac{ac_{1}}{n}\right\rfloor=\left\lfloor\frac{ab}{n}\right\rfloor \]
    and \[ n \left\lfloor \frac{ac_{1}}{n}\right\rfloor  = ac_{1}+d_{1}-n .\]
   By Lemma \ref{cdlemm},
    \[ c_{2}= c_{0}+d_{1}-n\]
    so that the first three terms of $x$ are 
    \[  c_{1}\uptheta + c_{0}\uptheta^{2}+  (ac_{1}+d_{1}-n)\uptheta^{2}  =c_{2}\uptheta^{2}+ c_{1}\uptheta^{3}.\]
    Inductively, if the first $k$ terms of $x$ resolve to 
    \[ c_{k-1}\uptheta^{k-1}+ c_{k-2}\uptheta^{k}  ,\]
    then adding to this $ne_{k} \uptheta^{k}$ and using (\ref{neIrecurssion}), we obtain
    \[c_{k-1}\uptheta^{k-1}+ c_{k-2}\uptheta^{k} + \big(ac_{k-1}+c_{k} - c_{k-2}\big)\uptheta^{k} = c_{k}\uptheta^{k}+  c_{k-1}\uptheta^{k+1} .\]
    We conclude that $x=0$ in $\widehat{\mathcal{O}}_{\uptheta}$.
\end{proof}

\begin{theo}\label{Norm-1Field} $\widehat{\mathcal{O}}_{\uptheta}$ is a topological field extending $K$.
\end{theo}

\begin{proof} By Theorem \ref{NInvertible}, $\Q\subset\widehat{\mathcal{O}}_{\uptheta}$, and since $\uptheta$ is a unit, $K\subset\widehat{\mathcal{O}}_{\uptheta}$.   Now let
$x = \{ \upalpha_{i}\}\in \widehat{\mathcal{O}}_{\uptheta}$, $\upalpha_{i}\in\mathcal{O}_{K}$.  Each $\upalpha_{i}$ has an inverse (as an element of $\widehat{\mathcal{O}}_{\uptheta}$) $\upalpha^{-1}_{i} = \{ \upalpha^{-1}_{ij}\}$; choose a diagonal subsequence
\[  x_{i}^{-1}:= \upalpha^{-1}_{ij}\]
so that $x_{i}^{-1}\upalpha_{i} =1 + \updelta_{i}\rightarrow 1$ and $|x_{i}^{-1}|_{\uptheta} = 1/ |\upalpha_{i}|_{\uptheta}$ (the latter choice is possible due to Lemma \ref{InverseNorm}).    Write $x^{-1}= \{ x_{i}^{-1}\}$; then the product of sequences $x\cdot x^{-1} = \{ 1+ \updelta_{i} \} $ converges to $1$ in $\widehat{\mathcal{O}}_{\uptheta}$.
What remains is to show that $x^{-1}$ is Cauchy.  Suppose otherwise, that there is a constant $C>0$ such that for all $N\in\N$, there exists $i,j>N$ with 
\begin{align} \label{notCauchy1}  |x_{i}^{-1} - x_{j}^{-1}|_{\uptheta}\geq C. \end{align}
By Theorem \ref{oscillation}, either $|x|_{\uptheta} = \lim |\upalpha_{i}|_{\uptheta}$ exists or $ |\upalpha_{i}|_{\uptheta}$ oscillates between at most two values, so for $x\not=0$, the limit points of the sequence $\{ |\upalpha_{i}|_{\uptheta}\}$
cannot be zero.  Thus (\ref{notCauchy1}) and inframultiplicativity imply that 
\[ |\upalpha_{i}x_{i}^{-1} - \upalpha_{i}x_{j}^{-1}|_{\uptheta}\geq C'>0\]
for some new constant $C'$.  On the other hand, 
\begin{align*} |\upalpha_{i}x_{i}^{-1} - \upalpha_{i}x_{j}^{-1}|_{\uptheta} & = |\upalpha_{i}x_{i}^{-1} - \upalpha_{j}x_{j}^{-1} + \upalpha_{j}x_{j}^{-1} -\upalpha_{i}x_{j}^{-1}|_{\uptheta} \\
&\leq  \uptheta^{2} \left(  |\upalpha_{i}x_{i}^{-1} - \upalpha_{j}x_{j}^{-1} |_{\uptheta}  + |  \upalpha_{j}x_{j}^{-1} - \upalpha_{i}x_{j}^{-1} |_{\uptheta}  \right) . \end{align*}
However, $|\upalpha_{i}x_{i}^{-1} - \upalpha_{j}x_{j}^{-1} |_{\uptheta} = |\updelta_{i}- \updelta_{j} |_{\uptheta} \rightarrow 0$ and 
\[  |  \upalpha_{j}x_{j}^{-1} - \upalpha_{i}x_{j}^{-1} |_{\uptheta}  \leq \frac{\uptheta^{2}}{|\upalpha_{j}|_{\uptheta}} | \upalpha_{j} -\upalpha_{i}|_{\uptheta} \longrightarrow  0,\]
since $|\upalpha_{i}|_{\uptheta} \rightarrow |x|_{\uptheta}$ and the latter assumes at most three (non zero) values.  Thus the sequence defining $x^{-1}$ must be Cauchy.
\end{proof}

\begin{theo}\label{IsoToR} The conjugation map {\rm (\ref{ContConjMap})} defines a topological isomorphism between $\widehat{\mathcal{O}}_{\uptheta}$ and $\R$.
\end{theo}

\begin{proof}  By Proposition \ref{GCWD} and Corollary \ref{ConjIsCont}, conjugation defines a continuous homomorphism of rings, and since 
$\widehat{\mathcal{O}}_{\uptheta}$ is a field by Theorem \ref{Norm-1Field},  is a monomorphism.  By Lemma 3, conjugation
is bicontinuous, hence the image is complete; it is all of $\R$ since it contains the dense subfield $K$.
\end{proof}

\begin{coro}\label{Norm-1MultiField} $\mathcal{O}_{\uptheta}$ is a Marty multifield extending $\mathcal{O}_{K}$.
\end{coro}

\begin{proof} Since we have already shown that $\mathcal{O}_{\uptheta}$ is a Marty multiring,  it will be enough to show that every element $x\in{\mathcal{O}}_{\uptheta}$ has an inverse (not necessarily unique).   Denote by $\hat{x}$ the corresponding class in $\widehat{\mathcal{O}}_{\uptheta}$, which by Theorem \ref{Norm-1Field} has an inverse $\hat{y}$.  Let $y\in \mathcal{O}_{\uptheta}$ be any
greedy Laurent series in the class of $\hat{y}$, represented by a sequence of partial sums $\{ y_{n}\}$.  Then if $\{ x_{n}\}$ is a sequence of partial sums for $x$, we have 
\[  x_{n}y_{n} = 1 + \upepsilon_{n}\]
where $\upepsilon_{n}\rightarrow 0$ in the $\uptheta$-adic infranorm.  But $\{ 1 + \upepsilon_{n}\}$ is a sequence of partial sums in the class of $1 \in  \mathcal{O}_{\uptheta}$, hence $1=xy$.

\end{proof}

 \vspace{3mm}
 
 \noindent \fbox{$\boldsymbol N\boldsymbol(\boldsymbol\uptheta \boldsymbol)\boldsymbol=\boldsymbol1$}
 
  \vspace{3mm}
  
 Once again, in this case, we will first show that $\widehat{\mathcal{O}}_{\uptheta}$ is a field, and then, as a corollary, that $\mathcal{O}_{\uptheta}$
 is a multifield.  In particular, when working in $\widehat{\mathcal{O}}_{\uptheta}$ , Laurent series in $\uptheta$ are considered as Cauchy classes of sequences.


\begin{lemm}\label{InverseNormN=1} Let $g\in \mathcal{O}_{K}$ have inverse $g^{-1}:=\{ h_{n}\}$ in $\widehat{\mathcal{O}}_{\uptheta}$.  Then 
\begin{itemize}
\item[Case 00:] If $g, h_{n}\in  \mathcal{O}_{K}^{0}$ eventually, then for $n$ large, 
$ | h_{n} |_{\uptheta}  =  \uptheta^{-1} \frac{1}{\;|g|_{\uptheta}}$.
\item[Case 01:] If $g \in  \mathcal{O}_{K}^{0}$ but $h_{n}\in  \mathcal{O}_{K}^{1}$ for infinitely many $n$, then for $n$ large, $ | h_{n} |_{\uptheta}  \in \{\uptheta^{-2}, \uptheta^{-1}   \} \frac{1}{|g|_{\uptheta}}$.
\item[Case 10:] If $g \in  \mathcal{O}_{K}^{1}$ and $h_{n}\in  \mathcal{O}_{K}^{0}$ eventually, then for $n$ large, $ | h_{n} |_{\uptheta}  \in \{\uptheta^{-2}, \uptheta^{-1}   \} \frac{1}{|g|_{\uptheta}}$.
\item[Case 11:] If $g \in  \mathcal{O}_{K}^{1}$ and $h_{n}\in  \mathcal{O}_{K}^{1}$ for infinitely many $n$, then for $n$ large, $ | h_{n} |_{\uptheta}  \in \{\uptheta^{-2}, \uptheta^{-1}   \} \frac{1}{|g|_{\uptheta}}$.
\end{itemize}
In summary, the $\uptheta$-infranorm of $g^{-1}$ is possibly multivalued with multivalues belonging to the set $ \{\uptheta^{-2}, \uptheta^{-1}   \} \frac{1}{|g|_{\uptheta}}$.
 \end{lemm}
 
 \begin{proof} First suppose $g\in\mathcal{O}^{0}_{K}$;  without loss of generality, we assume $g>0$ and $g\not=1$.  Rescaling by a power of $\uptheta$ we may assume that 
 \[  g = b_{0} + b_{1}\uptheta + \cdots + b_{m}\uptheta^{m}\]
 so that $g\in\Z_{\uptheta}$ and $|g|_{\uptheta}=1$.  Note that this implies that 
 $1<g' <\uptheta$.
  \vspace{3mm}
 
 \noindent {\em Case} $00$:  $ h_{n} \in \mathcal{O}^{0}_{K}$ eventually.
 
 \vspace{3mm}
 
Then, $h_{n}g = 1+e_{n}$  where $|e_{n}|_{\uptheta}\rightarrow 0$.  Note that $e_{n} =h_{n}g - 1$ can be in either $\mathcal{O}_{K}^{0}$ or $\mathcal{O}_{K}^{1}$.  The latter condition implies that $e'_{n}\rightarrow 0$ in the euclidean norm: indeed, for those elements $e_{n}\in \mathcal{O}_{K}^{0}$, then
$|e_{n}| =_{\rm gr} \sum_{i=N_{n}}^{R_{n}} b_{n,i}\uptheta^{i}$, hence $|e_{n}'| =_{\rm gr}   \sum_{i=N_{n}}^{R_{n}} b_{n,i}\uptheta^{-i}\rightarrow 0$ as $N_{n}\rightarrow \infty$ (by the greedy property of the expansion).
On the other hand, if $e_{n}\in \mathcal{O}_{K}^{1}$, then $Te_{n}\in \mathcal{O}_{K}^{0}$ ; by definition, $|e_{n}|_{\uptheta}=|Te_{n}|_{\uptheta}$, hence $T'e_{n}'\rightarrow 0$, which implies $e_{n}'\rightarrow 0$ in the euclidean norm.   Thus,  $h_{n}'g'\rightarrow 1$
implying $h_{n}'>0$ eventually, hence $h_{n}>0$ eventually (as $h_{n}\in  \mathcal{O}_{K}^{0}$).  Moreover, $h_{n}' \rightarrow 1/g'$ and thus eventually
$  \uptheta^{-1}< h_{n}'<1$, which forces $|h_{n}|_{\uptheta}=\uptheta^{-1}$.
 
  \vspace{3mm}
 
 \noindent {\em Case} $01$:  $ h_{n}\in \mathcal{O}^{1}_{K}$ for infinitely many $n$.
 
 \vspace{3mm}
 
 Let us pass to a subsequence so that $h_{n}\in  \mathcal{O}_{K}^{1}$ eventually.  In this case, the condition $h_{n}g = 1+e_{n}$ with $|e_{n}|_{\uptheta}\rightarrow 0$ forces $h_{n}'>0$ eventually, hence $h_{n}<0$ eventually.  As in the previous
 case, we have \begin{align}\label{h'ineq01}  \uptheta^{-1}<h_{n}'<1\end{align}eventually.  Now by the inequality (\ref{ConjNormComp}) of Lemma \ref{ConjNormLemma} (applied to $f=-Th_{n}>0$),
 \[   |h_{n}|_{\uptheta} := |T h_{n}|_{\uptheta} \leq  -T'h_{n}' \leq \uptheta  |T h_{n}|_{\uptheta}=: \uptheta | h_{n}|_{\uptheta},\]
 and since $-T'=1-\uptheta^{-1}$, (\ref{h'ineq01}) gives
 \[  \uptheta^{-1} (1-\uptheta^{-1}) <-T'h_{n}'<1-\uptheta^{-1}. \]
 From the above two inequalities, we deduce
 \[  \uptheta^{-2} (1-\uptheta^{-1})< | h_{n}|_{\uptheta}  <1-\uptheta^{-1} .\]
 Since $N(\uptheta )=1$,  $\uptheta>2$, so the only values of $| h_{n}|_{\uptheta}$ which are consistent with the above inequality are $\uptheta^{-1}$ and $\uptheta^{-2}$.
 
  \vspace{3mm}
 
 Now assume $g\in\mathcal{O}^{1}_{K}$, $g>0$.  After rescaling by a power of $\uptheta$, we may assume $|g|_{\uptheta}=|Tg|_{\uptheta}=1$.  Thus,
 since for $N(\uptheta )=1$ the greedy expansion of a polynomial in $\uptheta$ is also greedy,
 \[ Tg =_{\rm gr} b_{0} + b_{1}\uptheta  + \cdots +b_{n}\uptheta^{n}, \quad (Tg)' =_{\rm gr} b_{0} + b_{1}\uptheta^{-1}  + \cdots +b_{n}\uptheta^{-n},\]
which give the bounds
 \begin{align}\label{boundsonTG'}   1\leq  (Tg)' <\uptheta .\end{align}
The lower bound $1$ can be realized only when $a=3$:  $(Tg)' = 1$ implies $g= T^{-1}$, and since $N(T)= (\uptheta -1)(\uptheta^{-1} -1)= 2-a$,
 $T^{-1}\in \mathcal{O}_{K}$ precisely when $a=3$.  Otherwise, $T^{-1}\not\in \mathcal{O}_{K}$ and the value $g= T^{-1}$ is not possible.
 
   \vspace{3mm}
 
 \noindent {\em Case} $10$:  $h_{n}\in  \mathcal{O}^{0}_{K}$ eventually.
 
 \vspace{3mm}
 
 Since $g'<0$, the condition $h_{n}g=1+e_{n}$, $e_{n}\rightarrow 0$ in the $\uptheta$-adic norm, forces $h_{n}'<0$ eventually, i.e.,  $h_{n}<0$ and $h_{n}'\rightarrow 1/g'<0$ eventually. By (\ref{boundsonTG'}) we have  \[\frac{\uptheta}{\uptheta -1}\leq  -g'< \frac{\uptheta^{2}}{\uptheta -1}\]
 or 
 \begin{align}\label{ineqforhn'-}  \uptheta^{-1} -\uptheta^{-2}< -h_{n}' \leq 1-\uptheta^{-1} <1.\end{align}
 The right hand inequality in (\ref{ineqforhn'-}) implies immediately that $|h_{n}|_{\uptheta}<1$, but is consistent with $|h_{n}|_{\uptheta}\leq \uptheta^{-1}$.  Indeed, 
 $|h_{n}|_{\uptheta}= \uptheta^{-1}$ implies that $-h_{n}' =_{\rm gr} b_{1}\uptheta^{-1} +\cdots $ and if $(b_{1}+1)\uptheta^{-1} +\cdots $ is greedy,
 it is $<1$.  On the other hand, the left hand inequality is consistent with  $|h_{n}|_{\uptheta}\geq \uptheta^{-2}$,
but not with lower values.  Indeed, if $|h_{n}|_{\uptheta}= \uptheta^{-1}$, then \[-h_{n}' =_{\rm gr} b_{1}\uptheta^{-1} +\cdots > \uptheta^{-1} -\uptheta^{-2} .\]
In the case $|h_{n}|_{\uptheta}= \uptheta^{-2}$, we may have \[\uptheta^{-1} -\uptheta^{-2}< -h_{n}'  =_{\rm gr} b_{2}\uptheta^{-2} +\cdots, \]  provided that
$(b_{2}+1)\uptheta^{-2} +\cdots$ is {\it not} greedy (either because $b_{2}+1=a-1$ and a forbidden block is formed,  or $b_{2}+1=a$ and a carry is required to put it into greedy form).

 
  \vspace{3mm}
 
 \noindent {\em Case} $11$:  $h_{n}\in  \mathcal{O}_{K}^{1}$ for infinitely many $n$.
 
 \vspace{3mm}
 
Again,  we pass to a subsequence so that $h_{n}\in  \mathcal{O}_{K}^{1}$ eventually. Here, we are forced to have $h_{n}'<0$ and $h_{n}>0$, wherein $-h_{n}'\rightarrow -1/g'$, so as in {\it Case} 10, we have (\ref{ineqforhn'-}).
Multiplying the latter by $-T' = 1-\uptheta^{-1}$ we obtain
\[   \uptheta^{-1}(1-\uptheta^{-1})^{2}<T'h'_{n}  \leq (1-\uptheta^{-1})^{2}  . \]
The right hand inequality is consistent with $|h_{n}|_{\uptheta}:= |Th_{n}|_{\uptheta}\leq  \uptheta^{-1}$ but not with any larger values.  Indeed, if 
$ |Th_{n}|_{\uptheta} =1$, then we have 
\[ T'h_{n}'   =_{\rm gr} c_{0} + c_{1}\uptheta^{-1}+\cdots  > 1 - 2\uptheta^{-1} +  \uptheta^{-2} \]
since $ -2\uptheta^{-1} +  \uptheta^{-2} <0$.  On the other hand, if $ |Th_{n}|_{\uptheta} =\uptheta^{-1}$, then
\[T'h_{n}'   =_{\rm gr} c_{1}\uptheta^{-1}+c_{2} \uptheta^{-2} +\cdots  \leq 1 - 2\uptheta^{-1} +  \uptheta^{-2} \]
is consistent, e.g., if $(c_{1}+2)\uptheta^{-1}+c_{2} \uptheta^{-2} +\cdots$ is greedy.
As for the left hand inequality, it is consistent for $|h_{n}|_{\uptheta}:= |Th_{n}|_{\uptheta}\geq  \uptheta^{-2}$ but not for smaller values.  Indeed, if
$ |Th_{n}|_{\uptheta}= \uptheta^{-3}$ we have necessarily, for $a>3$,
\[T'h_{n}'   =_{\rm gr} c_{3}\uptheta^{-3}+c_{4} \uptheta^{-4} +\cdots  < \uptheta^{-1} - 2\uptheta^{-2} +  \uptheta^{-3} . \]
If $a=3$, this inequality still holds, even if a forbidden block is formed when adding $2\uptheta^{-2}$ to the left hand side: for then, the resolution
of such a block would be bounded from above by the short block  $2\uptheta^{-2} + 2\uptheta^{-3}= \uptheta^{-1}+\uptheta^{-4}$.  One might think that there is the possibility that the new
$\uptheta^{-4}$ term might combine with an existing one to give another forbidden block, producing $\uptheta^{-3}$: but for this to occur, 
we need that \[T'h_{n}'  = 2\uptheta^{-3} + \uptheta^{-4} +\cdots + \uptheta^{-k} +  2\uptheta^{-k} +\cdots ,\]
which is not greedy.
On the other hand, if $ |Th_{n}|_{\uptheta}= \uptheta^{-2}$ the inequality 
\[T'h_{n}'   =_{\rm gr} c_{2}\uptheta^{-2}+c_{3} \uptheta^{-3} +\cdots  > \uptheta^{-1} - 2\uptheta^{-2} +  \uptheta^{-3}  \]
is possible, for example, if $c_{2}+2=a$, since then the term $a\uptheta^{-2}$ resolves to greedy form via a carry, giving
\[ \uptheta^{-1} +  \uptheta^{-3} <a\uptheta^{-2}+c_{3} \uptheta^{-3} +\cdots  = \uptheta^{-1}  +   (c_{3}+1) \uptheta^{-3} +\cdots ,\] 
which holds if $c_{k}>0$ for any $k\geq 3$.


  \end{proof}
  
In what follows, we write
  \[ T^{-1} := -( 1+ \uptheta + \uptheta^{2} +\cdots ),\]
  and note that 
  at the level of Cauchy sequences, $T\cdot T^{-1} =1$:
 \[ T^{-1} \cdot T=  ( - 1- \uptheta - \uptheta^{2} -\cdots ) (\uptheta-1)= 1 -\uptheta  +\uptheta-\uptheta^{2}+\uptheta^{2} -\cdots \longrightarrow 1.\]
 This representation in $\widehat{\mathcal{O}}_{\uptheta}$ should not be confused with the formal factor of $T^{-1} =1/T$ appearing in the definition of the Laurent series which make up
 $\mathcal{O}_{\uptheta}$.

  We now consider the problem of inverting elements of $\Z$, just as we did in the previous section.  
  
  \begin{prop} If $n|(a-2)$,  then 
  \[     n^{-1} = T^{-1}x ,\]
where
  \[ x = T + \frac{(a-2)(n-1)}{n} \left\{ \uptheta + \uptheta^{2}+\cdots  \right\} = T - \frac{(a-2)(n-1)}{n} \uptheta T^{-1}.  \]
  \end{prop}
  
  \begin{proof} First, 
  \begin{align*} xn  & =   T + (n-1) \left(T+ (a-2)\uptheta + (a-2)\uptheta^{2} + \cdots \right)  \\ 
  & = T+ (n-1) \left(-1+ (a-1)\uptheta + (a-2)\uptheta^{2} + \cdots \right) . 
  \end{align*}
  Using the norm = $1$ Fibonacci identity, \[ \uptheta^{m+2} = a\uptheta^{m+1} -\uptheta^{m},\]
  we may write 
  \begin{align*}  -1+ (a-1)\uptheta + (a-2)\uptheta^{2} + \cdots  & =   -1+ a\uptheta -\uptheta + (a-2)\uptheta^{2} + \cdots  \\
  & =   -\uptheta + (a-1)\uptheta^{2} + (a-2)\uptheta^{3} \cdots \\
  &\longrightarrow 0.  \end{align*}
  Thus, $xn =T$, and therefore it defines an element of  $\widehat{\mathcal{O}}_{\uptheta}$. 
  \end{proof}

  We now describe a general procedure, consisting of three steps, for inverting $n\in\N$ in the ring $\widehat{\mathcal{O}}_{\uptheta}$.   As before, we may assume without loss of generality that $n$ is prime.
  
  \vspace{3mm}
  
  \noindent \fbox{Step 1}  Define $b$ to be the unique element of $\mathcal{I}_{1}=\{ 1, \dots ,  n\}$ such that
  \[ a\equiv b\mod n,\quad b\not= 2. \]
  Then define the recursive sequence $\{ c_{i}\} \subset \mathcal{I}_{1}$, 
  \[ c_{0}=1, c_{1}=b, \quad c_{i} \equiv bc_{i-1}-c_{i-2}\mod n.  \]
  This sequence is almost equal to its norm $-1$ counterpart, the difference being essentially the signs of the $c_{i}$.
  
    \vspace{3mm}
  
  \noindent \fbox{Step 2}   Define, for $i\geq 2$,
  \[ e_{i} = \frac{ -c_{i} + ac_{i-1} - c_{i-2}}{n}. \]
  Note that by the recursive definition of the sequence of $c_{i}$, we have
 \[  -c_{i} + ac_{i-1} - c_{i-2} \equiv (a-b)c_{i-1} \equiv 0 \mod n,\]
 so that $e_{i}\in\Z$.  Moreover $e_{i}\leq a-1$ since $ac_{i-1}/n\leq a$ and $c_{i-1}, c_{i-2}>0$.  On the other hand 
\[  -2 + \frac{a}{n}=\frac{-2n + a}{n} \leq e_{i} ,\]
so that $e_{i}\geq -1$. 

\vspace{3mm}
    Write
      \begin{align}\label{n^-1Norm1}  n^{-1} := \left\lfloor \frac{a-1}{n}  \right\rfloor \uptheta + e_{2} \uptheta^{2} + \cdots + 
   e_{i} \uptheta^{i} + \cdots   =: n_{+}^{-1}- n_{-}^{-1}  \end{align}
   where $n_{+}^{-1}$ is the series formed from the positive $e_{i}$'s and $-n_{-}^{-1}$ is the series corresponding
   to those $e_{i}=-1$.  Note that $n_{-}^{-1}$ is automatically greedy since its coefficients are $\leq 1$.
   
 \begin{theo}\label{nInversionN=1}  The series $n^{-1}=n_{+}^{-1}- n_{-}^{-1} $ yields a well-defined element of  $\widehat{\mathcal{O}}_{\uptheta}$ 
 representing the inverse of $n$.
  \end{theo}
  
  \begin{proof} First we show that formally $n\cdot n^{-1}$ gives a Cauchy sequence converging to 1.  Thus we study the expression
  \begin{align*} n\cdot n^{-1} & = n \left\lfloor \frac{a-1}{n}  \right\rfloor \uptheta + ne_{2}\uptheta^{2} + \cdots + ne_{i}\uptheta^{i} + \cdots  \\ 
  &  =mn \uptheta +  ne_{2}\uptheta^{2} + \cdots + ne_{i}\uptheta^{i} + \cdots  ,
  \end{align*}
  where the second line follows upon making the substitution $a=mn +b$.
   Using the identity $\uptheta^{2}+1=a\uptheta$, we have 
  \begin{align*}
  n\cdot n^{-1}  & = 1 -b\uptheta  + \left( ne_{2}+1\right)\uptheta^{2} +ne_{3}\uptheta^{3} + \cdots .\\
  \end{align*}
Note that,
  \[  ne_{2}+1= ne_{2} +c_{0} =  
  ac_{1}-c_{2} = ab-c_{2} ,  \]
  so we may write (using $b\uptheta^{3} +b\uptheta=ab\uptheta^{2}$),
    \begin{align*}
  n\cdot n^{-1}  & = 1 -c_{2}\uptheta^{2}  + \left( ne_{3}+c_{1}\right)\uptheta^{3} + ne_{4}\uptheta^{4} + \cdots  .\\
  \end{align*}
  Inductively, at the $i$th step, we obtain
   \begin{align*}
  n\cdot n^{-1}  & = 1 -c_{i-1}\uptheta^{i-1}  + \left( ne_{i}+c_{i-2}\right)\uptheta^{i} + ne_{i+1}\uptheta^{i+1} + \cdots , \\
  \end{align*}
  which may be resolved to the next higher power using the identity $c_{i-1}\uptheta^{i+1} + c_{i-1}\uptheta^{i-1} = ac_{i-1}\uptheta^{i}$.
  This show that $n^{-1}$ is formally an inverse of $n$.
 We end by showing that $n^{-1}$ indeed defines an element of $\widehat{\mathcal{O}}_{\uptheta}$.  It will suffice to show that  
 $n^{-1}_{+}$ defines an element of $\widehat{\mathcal{O}}_{\uptheta}$, since we have already remarked
 that   $n^{-1}_{-}$ defines an element of the ring $\widehat{\mathcal{O}}_{\uptheta}$.  Thus, without loss of generality, 
 we may assume  $n^{-1}=n^{-1}_{+}$.  With this assumption, we have $0\leq e_{i}\leq a-1$ for $i\geq 2$. 
 Moreover, clearly $ \left\lfloor \frac{a-1}{n}  \right\rfloor\in \{ 0,\dots ,a-1\}$ 
so all of the coefficients belong to the set $\{ 0,\dots ,a-1\}$.   Write $n^{-1}= \{ n^{-1}_{m}\}$, where $n^{-1}_{m}$ consists of the first $m$ terms of 
$n^{-1}$.  By definition, $ n^{-1}_{m}\in \mathcal{O}_{K}^{0}$, so $ n^{-1}_{m}$, if not already in greedy form, has a finite greedy expansion. 
  Now for $\tilde{m}>m$, 
  \begin{align*} | n^{-1}_{\tilde{m}}-n^{-1}_{m}|_{\uptheta} & =
   \uptheta^{-m-1} \left| e_{m+1} +e_{m+2} \uptheta +\cdots +e_{\tilde{m}}\uptheta^{\tilde{m} -m-1}\right|_{\uptheta}  \\ 
   &  =: |f+g|_{\uptheta} \end{align*}
   where 
   \[ f = d_{0} + d_{1}\uptheta + \cdots + d_{\tilde{m}-m-1}\uptheta^{\tilde{m}-m-1} \]
   and where 
   \[ d_{i} =\left\{ 
   \begin{array}{ll}
   e_{m+i+1}-1 & \text{if } e_{m+i+1} \not=0 \\
   & \\
   0 & \text{otherwise}
   \end{array} \right. \]
   Then $f$ is greedy since its coefficients are $\leq a-2$, and so is $g$, since its coefficients are all either $0$ or $1$.  By the 
   shape of $f$ and $g$, $|f|_{\uptheta}$, $|g|_{\uptheta}\leq 1$, so by the infratriangle inequality (Theorem \ref{N=1ITE}, item 1),  we have
   \[  |  n^{-1}_{\tilde{m}}-n^{-1}_{m}|_{\uptheta} \leq  \uptheta^{-m} \longrightarrow 0\]
   as $m\rightarrow\infty$.  This shows that  $n^{-1}= \{n^{-1}_{m}\}$ defines an element of $\widehat{\mathcal{O}}_{\uptheta}$.
  \end{proof}

    \begin{theo}\label{hatOisafield} $\widehat{\mathcal{O}}_{\uptheta}$ is a topological field extending $K$.
\end{theo}

\begin{proof}The proof has the same structure as that of Theorem \ref{Norm-1Field}. By Theorem \ref{nInversionN=1}, $\Q\subset\widehat{\mathcal{O}}_{\uptheta}$, and since $\uptheta$ is a unit, $K\subset\widehat{\mathcal{O}}_{\uptheta}$.   Now let
$x = \{ \upalpha_{i}\}\in \widehat{\mathcal{O}}_{\uptheta}$, $\upalpha_{i}\in\mathcal{O}_{K}$.  Each $\upalpha_{i}$ has an inverse (as an element of $\widehat{\mathcal{O}}_{\uptheta}$) $\upalpha^{-1}_{i} = \{ \upalpha^{-1}_{ij}\}$; choose a diagonal subsequence
\[  x_{i}^{-1}:= \upalpha^{-1}_{ij}\]
so that $x_{i}^{-1}\upalpha_{i} =1 + \updelta_{i}\rightarrow 1$ and $|x_{i}^{-1}|_{\uptheta} = c/ |\upalpha_{i}|_{\uptheta}$, where $c\in \{ \uptheta^{-2},\uptheta^{-1}\}$ (the latter choice is possible due to Lemma \ref{InverseNormN=1}).    Write $x^{-1}= \{ x_{i}^{-1}\}$; then the product of sequences $x\cdot x^{-1} = \{ 1+ \updelta_{i} \} $ converges to $1$ in $\widehat{\mathcal{O}}_{\uptheta}$.
What remains is to show that $x^{-1}$ is Cauchy.  Suppose otherwise, that there is a constant $C>0$ such that for all $N\in\N$, there exists $i,j>N$ with 
\begin{align} \label{notCauchy}  |x_{i}^{-1} - x_{j}^{-1}|_{\uptheta}\geq C. \end{align}
By Theorem \ref{oscillationN=1}, either $|x|_{\uptheta} = \lim |\upalpha_{i}|_{\uptheta}$ exists or $\{  |\upalpha_{i}|_{\uptheta}\}$ oscillates between at most five values, so for $x\not=0$, the limit points of the sequence $\{ |\upalpha_{i}|_{\uptheta}\}$
cannot be zero.  Thus (\ref{notCauchy}) and inframultiplicativity imply that 
\[ |\upalpha_{i}x_{i}^{-1} - \upalpha_{i}x_{j}^{-1}|_{\uptheta}\geq C'>0\]
for some new constant $C'$.  On the other hand, 
\begin{align*} |\upalpha_{i}x_{i}^{-1} - \upalpha_{i}x_{j}^{-1}|_{\uptheta} & = |\upalpha_{i}x_{i}^{-1} - \upalpha_{j}x_{j}^{-1} + \upalpha_{j}x_{j}^{-1} -\upalpha_{i}x_{j}^{-1}|_{\uptheta} \\
&\leq  \uptheta^{4} \left(  |\upalpha_{i}x_{i}^{-1} - \upalpha_{j}x_{j}^{-1} |_{\uptheta}  + |  \upalpha_{j}x_{j}^{-1} - \upalpha_{i}x_{j}^{-1} |_{\uptheta}  \right),\end{align*}
where the factor of $\uptheta^{4}$ comes from the infratriangle Inequality for $N(\uptheta )=1$, Theorem \ref{N=1ITE}.
However, $|\upalpha_{i}x_{i}^{-1} - \upalpha_{j}x_{j}^{-1} |_{\uptheta} = |\updelta_{i}- \updelta_{j} |_{\uptheta} \rightarrow 0$ and 
\[  |  \upalpha_{j}x_{j}^{-1} - \upalpha_{i}x_{j}^{-1} |_{\uptheta}  \leq \frac{c\uptheta^{2}}{|\upalpha_{j}|_{\uptheta}} | \upalpha_{j} -\upalpha_{i}|_{\uptheta} \longrightarrow  0\]
since $|\upalpha_{j}|_{\uptheta} $ and the latter assumes at most five (non zero) values.  (The factor $\uptheta^{2}$ comes from Inframultiplicativity in
the case $N(\uptheta )=1$, Theorem \ref{N=1InfraMult}.)  Thus the sequence defining $x^{-1}$ must be Cauchy.
\end{proof}

Just as in the previous section, Theorem \ref{hatOisafield} implies 

\begin{coro}\label{Norm1MultiField} $\mathcal{O}_{\uptheta}$ is a Marty multifield extending $\mathcal{O}_{K}$.
\end{coro}

\begin{proof} Since we have already shown that $\mathcal{O}_{\uptheta}$ is a Marty multiring,  it will be enough to show that every element $x\in{\mathcal{O}}_{\uptheta}$ has an inverse (not necessarily unique).   Denote by $\hat{x}$ the corresponding class in $\widehat{\mathcal{O}}_{\uptheta}$, which by Theorem \ref{hatOisafield} has an inverse $\hat{y}$.  Let $y\in \mathcal{O}_{\uptheta}$ be any
greedy Laurent series in the class of $\hat{y}$ (guaranteed by Proposition \ref{LaurSerRepPropN1}) represented by a sequence of partial sums $\{ y_{n}\}$.  Then if $\{ x_{n}\}$ is a sequence of partial sums for $x$, we have 
\[  x_{n}y_{n} = 1 + \upepsilon_{n}\]
where $\upepsilon_{n}\rightarrow 0$ in the $\uptheta$-adic infranorm.  But $\{ 1 + \upepsilon_{n}\}$ is a sequence of partial sums in the class of $1 \in  \mathcal{O}_{\uptheta}$, hence $1= xy$.
\end{proof}


\section{Quasicrystal Completions and the $\uptheta$-adics}\label{QCCompletionSection}

Basic definitions and concepts related to quasicrystals were introduced in \S \ref{GreedySection}.
Let $\Upomega\subset\R$ be a 1-dimensional quasicrystal; denote 
\[   \Upomega_{R}:= \Upomega\cap [-R,R] \]  
for any $R>0$. 
A sequence of translates $\{ \upalpha_{i}+ \Upomega \}$ is called Cauchy \cite{BG} if for all $R$ there exists $N_{R}$ such that for all $i,j>N$, 
\begin{align}\label{QCConvCond}  (\Upomega-\upalpha_{i})_{R}=  (\Upomega-\upalpha_{j})_{R} . \end{align}

\begin{prop} The union
\[  \Upomega_{\hat{\upalpha}}:=\bigcup_{ \substack{R>0 \\ i>N_{R}}}  (\Upomega-\upalpha_{i})_{R}    \]
is a quasicrystal.  
\end{prop}

\begin{proof}  First note that the constants of relative density and uniform discreteness of the translates $\Upomega-\upalpha_{i}$ are equal to those of $\Upomega$.  
Since $ \Upomega_{\hat{\upalpha}}$ is a nested union of the $ (\Upomega-\upalpha_{i})_{R} $, it is also uniformly discrete and relatively dense, having the same constants of relative density and uniform discreteness.    By the characterization
theorem of quasicrystals (see Theorem 9.1, (vi), of \cite{Moody}), if $\Uplambda\subset\R^{n}$ is relatively dense, then $\Uplambda$ is a quasicrystal if and only if $\Uplambda-\Uplambda$ is uniformly discrete.
Thus, it is enough to show that $ \Upomega_{\hat{\upalpha}}- \Upomega_{\hat{\upalpha}}$ is uniformly discrete.  
But any element $x-y\in \Upomega_{\hat{\upalpha}}- \Upomega_{\hat{\upalpha}}$ is contained
in an approximation 
\[  (\Upomega-\upalpha_{i})_{R} - (\Upomega-\upalpha_{i})_{R} \subset  (\Upomega-\upalpha_{i} ) -(\Upomega-\upalpha_{i} ) = \Upomega-\Upomega.\]
From this it follows that  $\Upomega_{\hat{\upalpha}}- \Upomega_{\hat{\upalpha}}\subset\Upomega-\Upomega$.  Again by Theorem 9.1, (vi), of \cite{Moody}, since $\Upomega$ is a quasicrystal, $\Upomega-\Upomega$ is uniformly
discrete, hence so is $\Upomega_{\hat{\upalpha}}- \Upomega_{\hat{\upalpha}}$.  We conclude that $ \Upomega_{\hat{\upalpha}}$ is a quasicrystal.
\end{proof}

The set of quasicrystals of the form $\Upomega_{\hat{\upalpha}}$ is denoted 
\[\contour{black}{\color{white}$\boldsymbol\Omega$}_{\;{\sf qc}} \]
and is called the {\bf {\em quasicrystal completion}}. 

\begin{lemm} The collection of sets
\[ \mathcal{O}_{R}( \Upomega_{\hat{\upalpha}}) = \{ \Upomega'\in\contour{black}{\color{white}$\boldsymbol\Omega$}_{\;{\sf qc}}\;:\;\; \Upomega_{R}'=( \Upomega_{\hat{\upalpha}})_{R}\},\quad R>0,\]
defines a basis for a topology on $\contour{black}{\color{white}$\boldsymbol\Omega$}_{\;{\sf qc}} $.
\end{lemm}
 
\begin{proof}
Given $R\leq S$ and $\Upomega ' \in  \mathcal{O}_{R}( \Upomega_{\hat{\upalpha}}) \cap  \mathcal{O}_{S}( \Upomega_{\hat{\upbeta}})$, 
then $\mathcal{O}_{S}( \Upomega')\subset  \mathcal{O}_{R}( \Upomega_{\hat{\upalpha}}) \cap  \mathcal{O}_{S}( \Upomega_{\hat{\upbeta}})$, since
$\Upomega''\in \mathcal{O}_{S}( \Upomega')$ if and only if
 $\Upomega''_{S}= \Upomega'_{S}$, which implies 1) $\Upomega''_{S}=( \Upomega_{\hat{\upbeta}} )_{S}$, i.e., $\Upomega'' \in   \mathcal{O}_{S}( \Upomega_{\hat{\upbeta}})$ and 2)
$\Upomega''_{R} =  \Upomega'_{R} = ( \Upomega_{\hat{\upalpha}})_{R}$, implying $\Upomega'' \in   \mathcal{O}_{R}( \Upomega_{\hat{\upalpha}})$.
\end{proof}
The topology defined by the basis in the Lemma is called the {\bf {\em quasicrystal topology}}.

Closely related to $\contour{black}{\color{white}$\boldsymbol\Omega$}_{\;{\sf qc}} $ is the completion
\[  \hat{\SI}_{\Upomega}  := \overline{ \{ r+\Upomega \; : \;\; r\in\R \}} , \]
defined with respect to the uniform structure given by
\[ {\sf U}_{\upvarepsilon,R} = \{ (\Upomega' , \Upomega'') \; : \;\;  d_{\rm Haus} ( \Upomega_{R}' , \Upomega_{R}'')<\upvarepsilon \} , \]
where
\[ \Upomega' = r+\Upomega, \;\;  \Upomega'' = s+\Upomega , \quad r,s\in\R,\] and
 $d_{\rm Haus} $ is the Hausdorff distance.   In the literature, $ \hat{\SI}_{\Upomega}$ is referred to as the {\bf {\em hull}} of $\Upomega$:
 it is a compact Hausdorff space  \cite{BG}.

\begin{prop}
With the quasicrystal topology, $\contour{black}{\color{white}$\boldsymbol\Omega$}_{\;{\sf qc}}$ is a Stone space (compact and totally disconnected).
\end{prop}

\begin{proof}  The topology of the uniform structure, restricted to the subset $\{ \upalpha +\Upomega\; :\;\; \upalpha\in\Upomega\}$, gives the quasicrystal topology.  Indeed, 
$\Upomega+\Upomega$ is also a quasicrystal (\cite{Moody}, Corollary 6.8, (ii)) and for 
 $\upalpha, \upbeta \in \Upomega$ we may view each of  $\Upomega':= \upalpha + \Upomega$, $\Upomega'': = \upbeta + \Upomega$ 
 in the quasicrystal $\upalpha +\upbeta + \Upomega+\Upomega$.  Then,
 for  $\upvarepsilon< $ constant of uniform discreteness of $\Upomega +\Upomega$, $(\Upomega', \Upomega'')\in {\sf U}_{\upvarepsilon,R}$ if and only if 
$ \Upomega_{R}' = \Upomega_{R}''$.  In particular,
the quasicrystal completion $\contour{black}{\color{white}$\boldsymbol\Omega$}_{\;{\sf qc}}$ is a closed subset of 
$  \hat{\SI}_{\Upomega} $, hence is compact.
 On the other hand, if $\Upomega'\not\in \mathcal{O}_{R}( \Upomega_{\hat{\upalpha}})$ i.e.\ $\Upomega_{R}'\not=( \Upomega_{\hat{\upalpha}})_{R}$,
 then $\mathcal{O}_{R}(\Upomega') \subset \mathcal{O}_{R}( \Upomega_{\hat{\upalpha}})^{\complement}$.  Thus $\mathcal{O}_{R}( \Upomega_{\hat{\upalpha}})$ is also closed, hence $\contour{black}{\color{white}$\boldsymbol\Omega$}_{\;{\sf qc}}$
 is totally disconnected.
\end{proof}

Now let us assume that $\Upomega $ is a 1-dimensional
quasicrystal defined as a model set based on the lattice $\mathcal{O}_{K}\subset \R^{2}$, with window
\[ W=W_{x}= (-\uptheta^{-x} , \uptheta^{-x} ) \quad \text{or}\quad \overline{W}=\overline{W}_{x} =[-\uptheta^{-x} , \uptheta^{-x} ] . \]
Note that if $\uptheta^{-x} \not\in \mathcal{O}_{K} $, then $W$ and $\overline{W}$ define the same model set, but otherwise, using the window $\overline{W}$ instead of $W$ will contribute two new elements: $\pm\uptheta^{-x}$.



We say that $\Upomega$  is  {\bf {\em repetitive}} if for each $R$, the set of $R$-periods 
\[   {\rm Per}_{R}(\Upomega ) := \{ x\in \R: \; (\Upomega-x)_{R} = \Upomega_{R} \} \]
is relatively dense.

 \begin{lemm}\label{replemma} $\Upomega$ is repetitive if and only if it may be defined using the open interval $W$.  
\end{lemm}

\begin{proof}  Suppose that $\Upomega$ is defined using $W$  and let 
\begin{align}\label{completelistinofOmega} \Upomega_{R}=\pm \{ 0=\upalpha_{0} <\upalpha_{1} < \cdots < \upalpha_{k} \} .\end{align} 
In what follows, to simplify certain formulas we will write $\upalpha_{k+1}:= R$ (abusively, since $R$ need not belong to $\Upomega$). In particular,  
\begin{enumerate}
\item[1.] By definition of $\Upomega$ and $\Upomega_{R}$, for all $i=0,\dots ,k$, $ \upalpha'_{i}\in (-\uptheta^{-x},\uptheta^{-x})$ and $\upalpha_{i}\in (0,R)$.
\item[2.] Since the ordered list (\ref{completelistinofOmega}) gives a complete account of the elements of $\Upomega_{R}$, there is no $0<\upgamma\in \mathcal{O}_{K}$   such that
for some $i\in \{0,\dots ,k\}$, 
\[ \upgamma < \upalpha_{i+1} -\upalpha_{i}\]
and
 \[  
  \upalpha_{i}+\upgamma \in \Upomega. \] 
\end{enumerate}
Item 2.\ implies that for all $0\leq i\leq k$ and $0<\upgamma< \upalpha_{i+1}-\upalpha_{i}$,  $ \upalpha_{i}+\upgamma\not\in\Upomega$, i.e., 
\[
\upgamma' \not\in W-\upalpha'_{i} = (-\uptheta^{-x}
-\upalpha'_{i} , \uptheta^{-x}
-\upalpha'_{i} ).\]
Consider the set  
\[ \Upsilon   = \big\{0\not= \upgamma \in\mathcal{O}_{K} \; : \;\;  0<\upgamma< \upalpha_{i+1}-\upalpha_{i}\text{ for some }i\leq k \big\}\subset (0,\upalpha),\]
where \[ \upalpha:= \max_{i\leq k}(\upalpha_{i+1}-\upalpha_{i}) .\]
The set of conjugates $\Upsilon'$ is a subset of the model set in $\{ 0\}\times\R$ defined using the window $ (0,\upalpha)$ in $\R\times\{0\}$, hence $\Upsilon'$ is uniformly discrete.  Therefore, we may choose
$\updelta>0$ small so that that for all $\upbeta\in\mathcal{O}_{K} $ with $|\upbeta'|<\updelta$,
\begin{enumerate}
\item[a.] $ \upbeta'\pm\upalpha_{i}' \in  (-\uptheta^{-x},\uptheta^{-x})$: that is to say, $\upbeta \pm \upalpha_{i}\in\Upomega$.  This follows since
the set of $\pm \upalpha_{i}'\in (-\uptheta^{-x},\uptheta^{-x})$ is a finite set.
\item[b.] For $\upgamma\in\Upsilon$ and all $i\in \{ 0,\dots ,k\}$  for which $0<\upgamma< \upalpha_{i+1}-\upalpha_{i}$,
\begin{align}\label{Statementb}  \upgamma' \not\in (-\uptheta^{-x}\-\upalpha'_{i} -\upbeta', \uptheta^{-x} -\upalpha'_{i} -\upbeta'). \end{align} Indeed,
we already have, for some $i$, \[  \upgamma' \not\in (-\uptheta^{-x}-\upalpha'_{i} , \uptheta^{-x} -\upalpha'_{i} )\]
and since $\Upsilon'$ is uniformly discrete, we can find a uniform lower bound for the distance of $\upgamma'$ to the interval 
$(-\uptheta^{-x}-\upalpha'_{i} , \uptheta^{-x} -\upalpha'_{i} )$; hence choosing $\updelta<$ the lower bounds for all $i$ implies (\ref{Statementb})
for any $|\upbeta'|<\updelta$.
 In particular, \begin{align} \label{thenoninclusion}\upbeta+ \upalpha_{i}+\upgamma\not\in \Upomega\end{align}
for $0<\upgamma < \upalpha_{i+1}-\upalpha_{i}$ and $i\leq k$. 
\end{enumerate}

Condition a.\ implies that $\pm\upalpha_{i} = (\pm\upalpha_{i} +\upbeta)-\upbeta\in (\Upomega-\upbeta)_{R}$.  Moreover, for $\updelta$ sufficiently small,
$\upbeta\in \Upomega$, so $0\in (\Upomega-\upbeta)_{R}$.  This gives the inclusion $\Upomega_{R}\subset (\Upomega-\upbeta)_{R}$.  Condition b.\
gives the opposite inclusion: indeed, suppose there exists \[ x=\upalpha-\upbeta\in  (\Upomega-\upbeta)_{R}\setminus \Upomega_{R},\quad \upalpha\in\Upomega.\]
We assume without loss of generality that $x>0$; the case $x<0$ can be handled using identical arguments.  Then, $x$ must fall within the gaps in $(0,R)$ demarcated by the $\upalpha_{i}$. 
Thus there exists a $\upgamma>0$ with $\upgamma<\upalpha_{i+1}-\upalpha_{i}$ and for which
\[  x = \upalpha -\upbeta =\upalpha_{i}+\upgamma.\]
But then, 
\[ \upalpha= \upbeta + \upalpha_{i}+\upgamma\in\Upomega ,\]
contradicting (\ref{thenoninclusion}).
We conclude that $(\Upomega-\upbeta)_{R} =\Upomega_{R}$, hence $\upbeta\in {\rm Per}_{R}(\Upomega )$.  The set of $\upbeta$ with $|\upbeta'|<\updelta$ is a quasicrystal, hence relatively dense, hence ${\rm Per}_{R}(\Upomega )$
is also relatively dense.  This proves repetitivity.
On the other hand, if $\Upomega$ is defined with $\overline{W}$ but cannot be defined with $W$, there exists $\upalpha\in\Upomega$ with $\upalpha'=\uptheta^{-x}$.  The set $\Upomega_{R}$ for $R>|\upalpha|$
 cannot occur anywhere else in $\Upomega$: that is, there cannot exist $\upbeta$ with $(\Upomega-\upbeta)_{R}=  \Upomega_{R}$. 
 Indeed, suppose on the contrary there exists $\upbeta\not=0$ with
  $(\Upomega-\upbeta)_{R}=  \Upomega_{R}$.  Since $\upalpha\in \Upomega_{R}$, there exists $\upgamma\in\Upomega$ with 
  $\upgamma' = \upbeta'+\upalpha' =  \upbeta'+\uptheta^{-x}$, which can happen only if $\upbeta'<0$.  But then, since we also have $-\upalpha\in \Upomega_{R}$,
  there exists $\upeta\in \Upomega$ with 
  $\upeta' = \upbeta'-\upalpha' =  \upbeta'-\uptheta^{-x}\not\in \overline{W}$, contradiction.
 In particular, for such $R$,  $  {\rm Per}_{R}(\Upomega ) = \{0\}$, so $\Upomega$ is not repetitive.
\end{proof}

\begin{lemm}\label{FIniteAlph} The set of diferences of consecutive elements of $\Upomega$ is finite.  In particular, $\Upomega$ may be described as a bi-infinite word on a finite alphabet.
\end{lemm}

\begin{proof}  Since $\Upomega$ is relatively dense, the set of differences of consecutive elements is uniformly bounded.  Thus, if there were an infinite set 
of distinct differences, we may find a non trivial convergent sequence $\updelta_{n}=\upbeta_{n}-\upalpha_{n}\in \Upomega-\Upomega$.  But since $\Upomega$ is a quasicrystal, $\Upomega-\Upomega$
is uniformly discrete, which is contradicted by the existence of the sequence $\updelta_{n}$.

\end{proof}

 \begin{coro}\label{minimalcoro} If $\Upomega$ can be defined by $W=W_{x}$, $\contour{black}{\color{white}$\boldsymbol\Omega$}_{\;{\sf qc}}$ is minimal: for all $\Upomega'\in\contour{black}{\color{white}$\boldsymbol\Omega$}_{\;{\sf qc}}$, $\contour{black}{\color{white}$\boldsymbol\Omega$}'_{\;{\sf qc}}=\contour{black}{\color{white}$\boldsymbol\Omega$}_{\;{\sf qc}}$.  
\end{coro}

\begin{proof} By Lemma \ref{FIniteAlph}, the  1-dimensional quasicrystal $\Upomega$ may be described as a bi-infinite word on a finite alphabet, where the letters of the alphabet index consecutive differences 
of elements of $\Upomega$.  For $\Upomega$ defined as a model set using $W$, by Lemma \ref{replemma}, $\Upomega$ is repetitive, which implies that the bi-infinite word indexing $\Upomega$ is repetitive.  The minimality
of $\contour{black}{\color{white}$\boldsymbol\Omega$}_{\;{\sf qc}}$ then follows from Proposition 4.3  on page 78 of \cite{BG}.
\end{proof}

 \begin{coro}\label{StoneCantorStructure} If $\Upomega$ is defined by 
\begin{enumerate}
\item[1.] $W$ then $\contour{black}{\color{white}$\boldsymbol\Omega$}_{\;{\sf qc}}$ is a Cantor set. 
\item[2.] $\overline{W}$ but is not definable by $W$, then $\contour{black}{\color{white}$\boldsymbol\Omega$}_{\;{\sf qc}}$ is a Stone space but not a Cantor set.  $\Upomega$
is dense and all of its points are isolated.
\end{enumerate}
\end{coro}

\begin{proof} Item 1.\ follows from Corollary \ref{minimalcoro}, since the latter shows that every point is a limit point, hence $\contour{black}{\color{white}$\boldsymbol\Omega$}_{\;{\sf qc}}$ is perfect.  As for item 2., since $\Upomega$ is not
repetitive, for any $\upalpha\in\Upomega$, there can be no non-trivial sequence $\upalpha_{i} +\Upomega$ converging to $\upalpha+\Upomega$.  Thus, the points of $\Upomega$ are isolated.
\end{proof}


  \begin{theo}\label{conjconvergence} If $\{ \upalpha_{i} +\Upomega\}$ converges in the quasicrystal topology, then the associated sequence of conjugates $\{ \upalpha'_{i}\}$ converges in $\R$.   That is
the map 
\[ \Upomega \longrightarrow W, \quad  \upalpha +\Upomega\longrightarrow \upalpha'\]
extends to a continuous map $\contour{black}{\color{white}$\boldsymbol\Omega$}_{\;{\sf qc}}\rightarrow \overline{W}$.
\end{theo}

\begin{proof}  First suppose that $\Upomega$ is defined by $W$, hence is repetitive.  Let 
\[ \Upomega_{\hat{\upalpha}} := \lim  \left( \upalpha_{i} +\Upomega\right) \in\contour{black}{\color{white}$\boldsymbol\Omega$}_{\;{\sf qc}}  .\]   Since $\upalpha_{i} +\Upomega\subset\mathcal{O}_{K}$ for all $i$,  $\Upomega_{\hat{\upalpha}}\subset \mathcal{O}_{K}$.  By Lemma 4.1 of \cite{Schlottmann}, 
\[   \bigcap_{t\in \Upomega_{\hat{\upalpha}}} ( t'-\overline{W}) = \{ c_{\hat{\upalpha}}\}  \]
for some $c_{\hat{\upalpha}}\in \R$.  Note that for $\Upomega_{\upalpha}:=\upalpha +\Upomega\in\contour{black}{\color{white}$\boldsymbol\Omega$}_{\;{\sf qc}}$, since $\upalpha' \in t'- W$ for all $t\in \upalpha +\Upomega$, 
$c_{\upalpha} = \upalpha'\in W$.   Therefore, the association
\[\contour{black}{\color{white}$\boldsymbol\Omega$}_{\;{\sf qc}}\longrightarrow \R,\quad  \Upomega_{\hat{\upalpha}}\longmapsto c_{\hat{\upalpha}}  \]
extends the conjugation map.  We must show that this association is continuous; once we have shown continuity it will follow that $c_{\hat{\upalpha}}\in \overline{W}$. 
 Let $V'\ni c_{\hat{\upalpha}}$ be an open neighborhood; then 
\[   \bigcap_{t\in \Upomega_{\hat{\upalpha}}}  [( t'-\overline{W}) \setminus V'  ] =\emptyset   .\]
The set of conjugates $t'$ for $t\in \Upomega_{\hat{\upalpha}}$ belongs to the relatively compact set $\bigcup (\upalpha_{i}' +W)$.
Thus, there exists a compact set such containing all of the closed sets $ ( t'-\overline{W}) \setminus V'   $, and we may therefore deduce that there is a finite subset $F\subset\Upomega_{\hat{\upalpha}}$ such that 
\[  \bigcap_{t\in F}  [( t'-\overline{W}) \setminus V'  ] =\emptyset . \] 
This shows that there is an $R>0$ such that
\[   \bigcap_{t\in (\Upomega_{\hat{\upalpha}})_{R}}  [( t'-\overline{W}) \setminus V'  ] =\emptyset    \]
This implies \[  \bigcap_{t\in(\Upomega_{\hat{\upalpha}})_{R}}  ( t'-\overline{W}) \subset V' .\]
Thus given any $\Upomega_{\hat{\upbeta}}\in\contour{black}{\color{white}$\boldsymbol\Omega$}_{\;{\sf qc}}$ for which
\[   (\Upomega_{\hat{\upbeta}})_{R} =  (\Upomega_{\hat{\upalpha}})_{R}  , \] we have
then $c_{\hat{\upbeta}} \in V'$.   This proves continuity when $\Upomega$ is defined by $W$.  If $\Upomega$ is defined by $\overline{W}$, it is
 not repetitive, however the above argument shows that the conjugation map on the Cantor set 
 \[ \contour{black}{\color{white}$\boldsymbol\Omega$}_{\;{\sf qc}}\setminus \{\text{i\small solated points}\}\] is continuous.  In particular, the conjugation map is thus continuous on all of $\hat{\Upomega}$.
\end{proof}

\begin{note} The argument we have used in the proof of Theorem \ref{conjconvergence} is essentially due to Schlottmann, see the proof of Proposition 4.3 of \cite{Schlottmann}. 
\end{note}

Consider a quasicrystal $\Upomega$ with defining window $W =(r,s)$, where $r,s\in \mathcal{O}_{K}$.
In the sequel we will need to understand the relationship between quasicrystal convergence and convergence of windows.  
More precisely given $\upalpha \in \Upomega$, $\upalpha + \Upomega$
is a model set having window $W_{\upalpha} = \upalpha'+W$.    By Theorem \ref{conjconvergence}, conjugation is continuous with respect to the quasicrystal topology: thus, given  a convergent sequence 
$\{ \upalpha_{i}+\Upomega\}$, the associated sequence of 
windows $\upalpha_{i}'+W$ will
have a well-defined limit $W_{\hat{\upalpha}} = \hat{\upalpha}' + W$.  Denote the limit quasicrystal  $\Upomega_{\hat{\upalpha}}\in\contour{black}{\color{white}$\boldsymbol\Omega$}_{\;{\sf qc}}$. If
the sequence defining $\hat{\upalpha}$ is eventually the trivial sequence $ \upalpha_{i}=\upalpha\in\mathcal{O}_{K}$, we will accordingly write 
$\hat{\upalpha}=\upalpha$ and $\Upomega_{\hat{\upalpha}}=\Upomega_{\upalpha}$ 
; if $ \upalpha_{i}\rightarrow \upalpha\in\mathcal{O}_{K}$ is not the constant sequence, we will continue to write 
$\Upomega_{\hat{\upalpha}}$ for the limit quasicrystal, since, as we shall see below, $\Upomega_{\hat{\upalpha}} \not= \Upomega_{\upalpha}$.

\begin{theo} $\Upomega_{\hat{\upalpha}}$ is a model set.  If $\hat{\upalpha}'\not\in \mathcal{O}_{K}$ or if $\hat{\upalpha}=\upalpha\in \mathcal{O}_{K}$,   it is defined by the window $W_{\hat{\upalpha}}$.  Otherwise, its window is either
 $[r+ \hat{\upalpha}', s + \hat{\upalpha}')$ or $(r+ \hat{\upalpha}', s + \hat{\upalpha}']$.
\end{theo}

\begin{proof}  The result is trivial if $\hat{\upalpha} =\upalpha \in\mathcal{O}_{K}$.  Now suppose that $\upbeta:=\hat{\upalpha}'\in \mathcal{O}_{K}$ but the defining sequence $\{ \upalpha_{i}\}$ is not trivial.  Then the sequence of conjugates $\{ \upalpha'_{i}\}$  must eventually converge monotonically: otherwise,
the sequence of windows $W_{\upalpha_{i}}$ will oscillate between containing or not containing $\upbeta$, which will mean that
for $R>|\upbeta'|$, the sets $\Upomega_{\upalpha_{i},R}$ will oscillate between containing or not containing $\upbeta'$.  Hence  the convergence condition (\ref{QCConvCond})
cannot be realized for $R>|\upbeta'|$.  Depending on whether the convergence is monotone increasing or decreasing, the $R$-balls in (\ref{QCConvCond})
will contain $r+ \hat{\upalpha}'$ but not $s + \hat{\upalpha}'$, or the contrary.
Otherwise, if $\hat{\upalpha}' \not\in \mathcal{O}_{K}$, the limit is allowed to be oscillatory on windows, since $\hat{\upalpha}' $ cannot belong to any quasicrystal associated to the lattice
$\mathcal{O}_{K}$.  Thus the model set having window $W_{\hat{\upalpha}}$ coincides with $\Upomega_{\hat{\upalpha}}$ .  \end{proof}

In what follows, we will consider the windows 
\[W=(-1,1)\subset \overline{W} =[-1,1] ,\]
which define the models sets 
\[ \mathfrak{m}\subset A =   \mathfrak{m} \cup \{\pm 1\}, \]
where $A$ is the quasicrystal ring which appears in (\ref{IntroQCRing}) in the Introduction and $\mathfrak{m}$ is the maximal quasicrystal ideal of $A$ defined
in (\ref{UniqMax}).
Each of these quasicrystals defines a completion, denoted
 \[\contour{black}{\color{white}$\boldsymbol{\mathfrak{m}}$}_{\;{\sf qc}},    \quad \contour{black}{\color{white}$\boldsymbol A$}_{\;{\sf qc}}.\]
 By Corollary \ref{StoneCantorStructure}, $\contour{black}{\color{white}$\boldsymbol{\mathfrak{m}}$}_{\;{\sf qc}}$ is a Cantor set
 and $\contour{black}{\color{white}$\boldsymbol A$}_{\;{\sf qc}}$ is a Stone set.
 We may consider an intermediate completion
 \[  \F_{\mathfrak{m}} :=    A/\mathfrak{m} := \overline{ \{  \upalpha+\mathfrak{m}\; : \;\; \upalpha\in A  \}  } , \]
 which may be thought of as the quasicrystal analog of the finite field $\F_{\mathfrak{p}} = \mathcal{O}_{K}/\mathfrak{p}$.   Note that the "new points" $\pm 1+\mathfrak{m}$ included in the set whose closure is  $\F_{\mathfrak{m}} $ are isolated: there are no non-trivial sequences $\upalpha_{i} +\mathfrak{m}$ converging to either.   Moreover, any sequence $\{ \upalpha_{i} +\mathfrak{m}\}$ containing infinitely may occurrences of $\upalpha_{i}=\pm 1$
 cannot converge unless it is eventually the constant sequence  $\upalpha_{i}=\pm 1$ (since for all $R$,  $0\not\in (\pm 1 +\mathfrak{m})_{R}$).
 Thus $\F_{\mathfrak{m}} $ is just $ \contour{black}{\color{white}$\boldsymbol{\mathfrak{m}}$}_{\;{\sf qc}}$ with two new and isolated points added: $\pm 1$.
We have the following maps relating these three structures
 \[
 \contour{black}{\color{white}$\boldsymbol{\mathfrak{m}}$}_{\;{\sf qc}}      \stackrel{i}{\longrightarrow}     \F_{\mathfrak{m}}    \stackrel{p}{\longleftarrow}   \quad \contour{black}{\color{white}$\boldsymbol A$}_{\;{\sf qc}} \\
  \]
 where the map $i$ is induced by the canonical inclusion 
 \[   \{  \upalpha+\mathfrak{m}\; : \;\; \upalpha\in \mathfrak{m} \} \subset \{  \upalpha+\mathfrak{m}\; : \;\; \upalpha\in A  \}    \]
 and the map $p$ is  induced by the association 
 \[ \upalpha + A\longmapsto   \upalpha + \mathfrak{m},\quad \upalpha\in A. \]
The map $i$ is injective mapping homeomorphically onto its image, since 
$\F_{\mathfrak{m}}= \contour{black}{\color{white}$\boldsymbol{\mathfrak{m}}$}_{\;{\sf qc}} \cup \{\pm 1\}$ with $\pm 1$ isolated points. The map $p$ is a continuous bijection.  Indeed, it is enough to show $p$ is continuous on the dense subset $A$ (as the domain and range of $p$ are compact metrizable spaces).
Then, convergence of $\{ \upalpha_{i}+A\}$  in $ \contour{black}{\color{white}$\boldsymbol A$}_{\;{\sf qc}} $ is defined by, for all $R$, stability
of the $R$-balls $(\upalpha_{i}+A)_{R}$, which in turn implies the stability for $(\upalpha_{i}+\mathfrak{m})_{R}$.  It is not bicontinuous since the Stone space $ \contour{black}{\color{white}$\boldsymbol A$}_{\;{\sf qc}} $ has infinitely many isolated points (those of $A$) whereas $\F_{\mathfrak{m}}$ has but two isolated points.

For $\Upomega $ equal to either $A$ or $\mathfrak{m}$, there are two inclusions
\[\contour{black}{\color{white}$\boldsymbol\Omega$}_{\;{\sf qc}} \supset \Upomega\subset  \mathcal{O}_{\uptheta}:\]
the first presents $\Upomega$ as a dense subset of $\contour{black}{\color{white}$\boldsymbol\Omega$}_{\;{\sf qc}}$, the second presents $\Upomega$ as a relatively compact subset of 
$ \mathcal{O}_{\uptheta}$.
Denote the completion of $\Upomega$ in $ \mathcal{O}_{\uptheta}$ by
\[ \contour{black}{\color{white}$\boldsymbol\Omega$}_{\uptheta}.\] 
Thus we have two additional completions: the $\uptheta$-adic completions 
\[ \contour{black}{\color{white}$\boldsymbol{\mathfrak{m}}$}_{\uptheta},\quad  \contour{black}{\color{white}$\boldsymbol A$}_{\uptheta}. \]

 \begin{theo}\label{InclusionTheorem}  The natural map $\upalpha\mapsto \upalpha +\mathfrak{m}$, $\upalpha\in  \mathfrak{m}$, induces a map
\begin{align}\label{themapinquestion} \contour{black}{\color{white}$\boldsymbol{\mathfrak{m}}$}_{\uptheta}\longrightarrow \contour{black}{\color{white}$\boldsymbol{\mathfrak{m}}$}_{\;{\sf qc}}.\end{align} 
which is
\begin{enumerate}
\item a homeomorphism when $N(\uptheta )=-1$.
\item a continuous surjection  when $N(\uptheta )=1$, which is bijective over points of $\contour{black}{\color{white}$\boldsymbol{\mathfrak{m}}$}_{\uptheta}$ with conjugate in $\mathcal{O}_{K}$ and otherwise two-to-one.  
 \end{enumerate}
\end{theo}

\begin{proof} \fbox{$\boldsymbol N\boldsymbol(\boldsymbol\uptheta \boldsymbol)\boldsymbol=\boldsymbol-\boldsymbol1$}   We first indicate how to define (\ref{themapinquestion}) on all $ \contour{black}{\color{white}$\boldsymbol{\mathfrak{m}}$}_{\uptheta}$, extending the association $\upalpha\mapsto \upalpha +\mathfrak{m}$ for $\upalpha\in \mathfrak{m}$.
Let $x= \sum_{m}^{\infty} b_{i}\uptheta^{i}\in \contour{black}{\color{white}$\boldsymbol{\mathfrak{m}}$}_{\uptheta}$. 
  If the conjugate of $x$ is not
in $\mathcal{O}_{K}$, we associate to $x$ the limit quasicrystal with window $W=(-1+x', 1+x')$ obtained from the sequence
$\{ x_{i} + \mathfrak{m}\}$ where $\{ x_{i}\}$ is a sequence of partial sums.  We note that there can be no other possible choice, as including either of the boundary points in the window does not change the model set, since $\pm 1+x'\not\in \mathcal{O}_{K}$.  Moreover, this definition clearly does not depend
on the choice of the sequence of partial sums $\{ x_{i}\}\in x$, since $x'$ does not depend on this choice.
Now suppose that $x$ is either of $\upalpha_{1}$, $\upalpha_{2}$ as in Proposition \ref{ShapeOfMultPre},
in which $\uppi (\upalpha_{1}) =\uppi (\upalpha_{2})=\upalpha \in \mathcal{O}_{K}\subset \widehat{\mathcal{O}}_{K}$.  One of the two, say
$\upalpha_{1}$, is such that its sequence of partial sums $\{ \upalpha_{1,i}\}$, satisfies $\upalpha_{1,i}'$ is (eventually) monotone decreasing, 
and then
$\upalpha_{2}$, is such that its sequence of partial sums $\{ \upalpha_{2,i}\}$, 
satisfies $\upalpha_{2,i}'$ is (eventually) monotone increasing.  
We assign to $\upalpha_{1}$ the limit quasicrystal $\mathfrak{m}_{\upalpha_{1}}$ having window 
$W_{R}=(-1+\upalpha', 1+\upalpha']$ and to $\upalpha_{2}$ the limit quasicrystal $\mathfrak{m}_{\upalpha_{2}}$ having window $W_{L}=[-1+\upalpha', 1+\upalpha')$.  This completes the definition of (\ref{themapinquestion}).  

The topologies on both  $\contour{black}{\color{white}$\boldsymbol{\mathfrak{m}}$}_{\uptheta}$ and $ \contour{black}{\color{white}$\boldsymbol{\mathfrak{m}}$}_{\;{\sf qc}}$ are the weak topologies: two series in 
$\contour{black}{\color{white}$\boldsymbol{\mathfrak{m}}$}_{\uptheta}$ are close if they agree up to a very high power of $\uptheta$ and two quasicrystals $\mathfrak{m}_{1}, \mathfrak{m}_{2}$
in the completion are close if they agree on a large interval $[-R,R]$.  

\vspace{3mm}

\noindent {\it Claim.} The two convergence conditions are equivalent: 
\begin{center}
$x_{i} \rightarrow x$ in 
$\contour{black}{\color{white}$\boldsymbol{\mathfrak{m}}$}_{\uptheta}$ $\Longleftrightarrow$if $\mathfrak{m}_{x_{i}} \rightarrow \mathfrak{m}_{x}$ in
$\contour{black}{\color{white}$\boldsymbol{\mathfrak{m}}$}_{\;{\sf qc}}$.
\end{center}

\vspace{3mm}

\noindent  \fbox{$\Rightarrow$} Let
$x_{i}\rightarrow x $ in $\contour{black}{\color{white}$\boldsymbol{\mathfrak{m}}$}_{\uptheta}$.  We note that If $x=\upalpha\in\mathcal{O}_{K}$ the result is trivial since the $\uptheta$-adic topology dictates
that eventually $x_{i}=x$.   Thus we may assume that $x\not\in \mathcal{O}_{K}$.  We divide the argument according to whether  $x'\in \mathcal{O}_{K}$ or not.  Suppose first that 
$x'\not\in\mathcal{O}_{K}$.  Denote by $W_{x_{i}}, W_{x}$ the windows defining $\mathfrak{m}_{x_{i}}$, $\mathfrak{m}_{x}$.
By the hypothesis $x'\not\in\mathcal{O}_{K}$, we may take $W_{x} = (-1+x',1+x')$.   The points of $(\mathfrak{m}_{x})_{R}$, which are finite in number, have conjugates belonging to a compact set $K\subset W_{x}$.  Therefore, for large $i$, $(\mathfrak{m}_{x})_{R}\subset (\mathfrak{m}_{x_{i}})_{R}$.  Suppose
that for arbitrarily large $i$, there exists $\upalpha_{i}\in  (\mathfrak{m}_{x_{i}})_{R}\setminus (\mathfrak{m}_{x})_{R}$.  Then, after passing to subsequences if necessary, the sequence of conjugates $\upalpha'_{i}$ converges (to either $-1+x'$ or $1+x'$) and $\upalpha_{i}$ converges to some $y\in [-R,R]$.  But this is impossible: for if we write
$\upalpha_{i}= m_{i} + n_{i}\uptheta$, $\upalpha_{i}' = m_{i}-n_{i}\uptheta^{-1}\rightarrow \pm 1+x'\not\in \mathcal{O}_{K}$ implies that both of $m_{i}$, $n_{i}$ tend to $\pm\infty$ with the same sign.  But then we could not have $m_{i} + n_{i}\uptheta$ convergent.  Therefore, eventually $(\mathfrak{m}_{x})_{R}=(\mathfrak{m}_{x_{i}})_{R}$ which gives convergence in this case. 

  If $x'\in \mathcal{O}_{K}$ but $x\not\in \mathcal{O}_{K}$, then per Proposition \ref{ShapeOfMultPre}, $x$ either has a tail of the form
$a(\uptheta^{2m} + \uptheta^{2m+2}+ \cdots)$ or $a(\uptheta^{2m+1} + \uptheta^{2m+3}+ \cdots)$.  Suppose the former, so that the window $W_{x}=[-1+x', 1+x')$.  If $x_{i}\rightarrow x$ then $x_{i}'\rightarrow x'$ {\it from below}.  Indeed, the difference $x'-x_{i}>0$ is positive since 
 it must either start with  odd power $-c_{2m_{i}+1}(-\uptheta)^{-(2m_{i}+1)} = c_{m_{i}}\uptheta^{-(2m_{i}+1)}$ or an even (positive) power 
 $(a-c_{2m_{i}})\uptheta^{-2m_{i}} \geq 0$, where the $c$'s are the relevant coefficients of $x_{i}$.  It follows then that $W_{x_{i}}\rightarrow W_{x}$.
Once again, in this case we have for large $i$, $(\mathfrak{m}_{x})_{R}\subset (\mathfrak{m}_{x_{i}})_{R}$.    If there were $\upalpha_{i}\in  (\mathfrak{m}_{x_{i}})_{R}\setminus (\mathfrak{m}_{x})_{R}$,  after passing to a subsequence, $\upalpha_{i}$ converges to some $r\in [-R,R]$.  At the same time, we have $\upalpha'_{i}\rightarrow \pm 1+ x'$, and since $\{\upalpha'_{i} = m_{i}-n_{i}\uptheta^{-1}\}$ is not the constant sequence, once again we must have $m_{i}, n_{i}\rightarrow \pm\infty$ with the same signs, violating the convergence of $\upalpha_{i}$.  The proves that convergence $\uptheta$-adic convergence implies quasicrystal convergence.

\vspace{3mm}

\noindent \fbox{$\Leftarrow$} Now suppose that $\mathfrak{m}_{x_{i}} \rightarrow \mathfrak{m}_{x}$ in
$\contour{black}{\color{white}$\boldsymbol{\mathfrak{m}}$}_{\;{\sf qc}}$: then we have convergence of the defining windows $W_{x_{i}}\rightarrow W_{x}$
implying convergence of the conjugates $x_{i}'\rightarrow x'$.  If $x'\not\in\mathcal{O}_{K}$, we must have $x_{i}\rightarrow x$ in the $\uptheta$-adic topology, since the conjugation map $\mathcal{O}_{\uptheta}\rightarrow \R$ is continuous with single pre-images over $\R\setminus \mathcal{O}_{K}$.  On the other hand,
if $x' =\upalpha'\in\mathcal{O}_{K}$,  there are three pre-images in $\mathcal{O}_{\uptheta}$: $\upalpha,\upalpha_{1}$ and $\upalpha_{2}$.  
We may suppose that $x\not= \upalpha$ so that $x$ is say $\upalpha_{1}$ with window $[-1+\upalpha' , 1+\upalpha')$, wherein the tail of $\upalpha_{1}$ 
consists of even powers, all with coefficient $a$. Therefore, the windows $W_{x_{i}}$
must converge to $W_{x}$ monotonically from below.  Therefore  $x'_{i}\nearrow x'$, and for this to be true,  on the level of partial sums we
must have $x_{i}\rightarrow x$ in the weak topology.  This completes the proof of the Claim.
Moreover, the map (\ref{themapinquestion}) is a bijection: indeed, if 
$\upalpha, \upbeta\in \contour{black}{\color{white}$\boldsymbol{\mathfrak{m}}$}_{\uptheta}$ have the same image in $\contour{black}{\color{white}$\boldsymbol{\mathfrak{m}}$}_{\;{\sf qc}}$, then $\upalpha'= \upbeta'$. If their common conjugate value is not in  
$\mathcal{O}_{K}$, $\upalpha=\upbeta$; otherwise, the definition of (\ref{themapinquestion}) takes the three possible elements in $ \contour{black}{\color{white}$\boldsymbol{\mathfrak{m}}$}_{\uptheta}$  to three distinct possible limit quasicrystals.
 Thus the map (\ref{themapinquestion}) is a homeomorphism.

\vspace{3mm}

 \fbox{$\boldsymbol N\boldsymbol(\boldsymbol\uptheta \boldsymbol)\boldsymbol=\boldsymbol1$}     If the conjugate of $x=\sum b_{i}\uptheta^{i}$ is not
in $\mathcal{O}_{K}$, we associate to $x$ the limit quasicrystal with window $W=(-1+x', 1+x')$ obtained from the sequence
$\{ x_{i} + \mathfrak{m}\}$ where $x_{i}$ is a sequence of partial sums.  Note since $r=x'\not\in \mathcal{O}_{K}$, there will be another series involving $T^{-1}$ having conjugate $r$ (by Theorem \ref{N=1ConjHasNoMore2}), and this series will produce the same quasicrystal limit.
If $\upalpha\in \mathcal{O}_{K,+}^{0}$, we associate to $\upalpha$ the quasicristal $\mathfrak{m}_{\upalpha}$, to the series $\upalpha_{1}=\sum b_{i}\uptheta^{i}$ with $\upalpha_{1}'=\upalpha'$ the quasicrystal with window $W_{L}$ and to the series  $\upalpha_{2}=-T^{-1}\sum c_{i}\uptheta^{i}$ with $\upalpha_{2}'=\upalpha'$ the quasicrystal with window $W_{R}$.  When  $\upalpha\in \mathcal{O}_{K,+}^{1}$, the series with $T^{-1}$ term is mapped to the quasicrystal with window $W_{L}$ and the series without $T^{-1}$   is mapped to the quasicrystal with window $W_{R}$.  
Note that this map is two-to-one over points whose conjugate is not in $\mathcal{O}_{K}$ and otherwise is bijective.
This defines the map (\ref{themapinquestion}); its continuity follows from an argument similar to that presented in the case $N(\uptheta )=-1$. 
\end{proof}

Since $ \contour{black}{\color{white}$\boldsymbol A$}_{\;{\rm qc}}$ is a Stone space in which the isolated points consist of $A\subset  \contour{black}{\color{white}$\boldsymbol A$}_{\;{\rm qc}}$, the analog of Theorem \ref{InclusionTheorem} is slightly different.   The same applies to $\F_{\mathfrak{m}}= \contour{black}{\color{white}$\boldsymbol{\mathfrak{m}}$}_{\;{\sf qc}}\cup \{\pm 1\}$, where
the points $\pm 1$ are isolated.

\begin{theo} There exist continuous maps
\begin{align} \label{qctotheta1} \contour{black}{\color{white}$\boldsymbol A$}_{\;{\rm qc}}\longrightarrow \contour{black}{\color{white}$\boldsymbol A$}_{\uptheta}    .\end{align}
and \begin{align} \label{f1theta}  \F_{\mathfrak{m}} \longrightarrow  \mathcal{O}_{\uptheta}  .\end{align}
When $N(\uptheta )=-1$ they are bijective onto their images and when $N(\uptheta )=1$ they are bijective on points having conjugate in $\mathcal{O}_{K}$, and otherwise are two-to-one.
\end{theo}
\begin{proof} The proof is essentially an addendum to the proof of Theorem \ref{InclusionTheorem}, rounded out by the following observations.
The $\uptheta$-adic completion $\contour{black}{\color{white}$\boldsymbol A$}_{\uptheta}$
contains extra points coming from Laurent series whose conjugates are in $\mathcal{O}_{K}$.  Thus the map (\ref{qctotheta1}) is not onto.  Moreover, it is not bicontinuous since 
$\contour{black}{\color{white}$\boldsymbol A$}_{\;{\rm qc}}$ is Stone whereas $\contour{black}{\color{white}$\boldsymbol A$}_{\uptheta}$ is Cantor.  Similar remarks apply to 
(\ref{f1theta}).
\end{proof}


\end{document}